\documentclass{article}

\usepackage{graphicx}
\usepackage{amsmath}
\usepackage{amsthm}
\usepackage{amssymb}
\usepackage{url}
\usepackage{amsmath,amscd}

 \DeclareFontEncoding{OT2}{}{} % to enable usage of cyrillic fonts
 \newcommand{\textcyr}[1]{%
    {\fontencoding{OT2}\fontfamily{wncyr}\fontseries{m}\fontshape{n}%
     \selectfont #1}}
 \newcommand{\Sha}{{\mbox{\textcyr{Sh}}}}

\newcommand{\dotcup}{\ensuremath{\mathaccent\cdot\cup}}

\newtheorem{thm}{Theorem}
\newtheorem{prop}{Proposition}
\newtheorem{cor}{Corollary}
\newtheorem{lem}{Lemma}

\DeclareMathOperator{\tr}{tr}

\DeclareMathOperator{\Gal}{Gal}
\DeclareMathOperator{\Hom}{Hom}
\DeclareMathOperator{\End}{End}
\DeclareMathOperator{\Ker}{Ker}

\DeclareMathOperator{\Aut}{Aut}

\DeclareMathOperator{\Frob}{Frob}

\DeclareMathOperator{\Image}{Im}
\DeclareMathOperator{\Homol}{H}
\DeclareMathOperator{\val}{val}
\DeclareMathOperator{\rank}{rank}
\DeclareMathOperator{\ord}{ord}
\DeclareMathOperator{\cha}{char}

\DeclareMathOperator{\lc}{lc}

\author{Fran\c cois Destrempes \\ \texttt{fdestrempes@bell.net}
\and Dmitry Malinin \\ \texttt{dmalinin@gmail.com}}
\title{On the vanishing of almost all primary components of the Shafarevich-Tate group of elliptic curves over the rationals}

\begin{document}

\maketitle

\begin{abstract}
The Shafarevich-Tate and Selmer groups arise in the context of Kummer theory for elliptic curves. The finiteness of the Shafarevich-Tate group of an elliptic curve $E$ over the field of rational numbers is included in the Birch and Swinnerton-Dyer conjectures, and is still an open question.

We present an overview of the Shafarevich-Tate and Selmer groups of an elliptic curve in the framework of Galois cohomology. Known results on the finiteness of the Shafarevich-Tate group are mentioned, including results of Coates and Wiles, Rubin, Gross and Zagier, and Kolyvagin. 

We then prove the vanishing of the $\ell$-primary component of the Shafarevich-Tate group for almost all primes $\ell$, 
for any elliptic curve $E$ over the rationals without complex multiplication.   
\end{abstract}

%%%%%%%%%%%%%%%%%%%%%%%%%%%%%%%%%%%%%%%%%%%%%%%%%%%%
%%%%%%%%%%%%%%%%%%%%%%%%%%%%%%%%%%%%%%%%%%%%%%%%%%%%

\section{Introduction}
\label{section:introduction}

%%%%%%%%%%%%%%%%%%%%%%%%%%%%%%%%%%%%%%%%%%%%%%%%%%%%

\subsection{Statement of the problem}
\label{subsection:statement} 

We recall Kummer theory for elliptic curves \cite[pp. 331--332]{silverman2009} to introduce the Shafarevich-Tate groups and the related Selmer groups.

Let $K$ be an algebraic number field with absolute Galois group $\mathcal{G}=\Gal(\overline{K}/K)$, where $\overline{K}$ denotes the algebraic closure of $K$. 
Given an elliptic curve $E$ over $K$ and a prime number $\ell$, one has the short exact sequence (isogeny property) of $\mathcal{G}$-modules:
\begin{equation}
\label{eq:eq1final6}
0 \rightarrow E[\ell] \rightarrow E(\overline{K}) \xrightarrow{[\ell]} E(\overline{K}) \rightarrow 0,
\end{equation}
where $E[\ell]$ denotes the group of $\ell$-torsion points in $E(\overline{K})$. 
This yields, from the long exact sequence of Galois cohomology, the following exact sequence of Abelian groups:
\begin{equation}
\label{eq:eq2final6}
0 \rightarrow E(K)/[\ell](E(K)) \xrightarrow{\partial} \Homol^{1}(\mathcal{G},E[\ell]) 
\rightarrow \Homol^{1}(\mathcal{G},E(\overline{K}))[\ell] \rightarrow 0.
\end{equation}
Recall that Galois cohomology (see \cite{gruenberg1967} and \cite[pp. 333-335]{silverman2009}) is based on the profinite group structure of Galois groups  
(see \cite{gruenberg1967} and \cite[pp. 4-6]{neukirch1986}). In particular, $1$-cocycles are continuous maps with respect to the Krull topology on Galois groups 
\cite[p. 2]{neukirch1986} and the discrete topology on Galois modules. When the Galois group is finite, Galois cohomology coincides with group cohomology. 

Given a place $v$ of $K$, let $K_v$ be the completion of $K$ at $v$, and denote $\mathcal{G}_v$ the absolute Galois group of $K_v$. Then, $\mathcal{G}_v$ may be viewed as a subgroup of $\mathcal{G}$ upon considering an embedding $\xi$ of $\overline{K}$ into 
$\overline{K}_v$. Moreover, the resulting embedding is continuous with respect to the Krull topology on Galois groups.  
One then has a commutative diagram:
\begin{equation}
\label{eq:eq3final6}
\begin{CD}
0 @>>> E(K)/[\ell](E(K)) @>{\partial}>> \Homol^1(\mathcal{G},E[\ell]) 
@>>> \Homol^1(\mathcal{G},E(\overline{K}))[\ell] @>>> 0\\
@. @VVV  @VV{Res^{\mathcal{G}}_{\mathcal{G}_v}}V @VV{Res^{\mathcal{G}}_{\mathcal{G}_v}}V \\
0 @>>> E(K_v)/[\ell](E(K_v)) @>{\partial}>> \Homol^1(\mathcal{G}_v,E[\ell]) 
@>>> \Homol^1(\mathcal{G}_v,E(\overline{K}_v))[\ell] @>>> 0,
\end{CD}
\end{equation}
where $Res^{\mathcal{G}}_{\mathcal{G}_v}$ denotes the restriction map of Galois cohomology \cite[pp. 331--332]{silverman2009}.
At this point, recall that the Shafarevich-Tate group of $E$ over $K$ is defined as \cite[p. 332]{silverman2009}: 
\begin{equation}
\label{eq:eq4final6} 
\Sha(E/K) = \Ker \Bigl \{ \Homol^1(\mathcal{G},E(\overline{K})) \rightarrow
\oplus_v  \Homol^1(\mathcal{G}_{v},E(\overline{K}_v)) \Bigr \}.
\end{equation}
Also, the $[\ell]$-Selmer group is defined as  
\cite[pp. 331--334]{silverman2009}:
\begin{equation}
\label{eq:eq5final6}
S^{[\ell]}(E/K) =\Ker \Bigl \{ \Homol^1(\mathcal{G},E[\ell]) \rightarrow \oplus_{v} 
\Homol^1(\mathcal{G}_v,E(\overline{K}_v)) \Bigr \}.
\end{equation}
In both equations (\ref{eq:eq4final6}) and (\ref{eq:eq5final6}), $v$ covers the set at all places of $K$. 
Note that each homomorphism $\Homol^1(\mathcal{G},E[\ell]) \rightarrow \Homol^1(\mathcal{G}_v,E(\overline{K}_v))$ 
maps into the $\ell$-torsion group $\Homol^1(\mathcal{G}_v,E(\overline{K}_v))[\ell]$. 
The above commutative diagram then yields a short exact sequence of the form \cite[p. 333]{silverman2009}:
\begin{equation} 
0 \rightarrow E(K)/[\ell](E(K)) \xrightarrow{\partial} S^{[\ell]}(E/K)
\rightarrow \Sha(E/K)[\ell] \rightarrow 0.
\end{equation}

One can show that the Selmer group $S^{[\ell]}(E/K)$ is finite \cite[pp. 333--334]{silverman2009}, from which follows 
the finiteness of $E(K)/[\ell](E(K))$ (Weak Mordell-Weil Theorem). A descent argument based on the notion of height function then shows the following result \cite[Theorem 6.7, p. 239]{silverman2009}.

%%%%%%%%%%%%%%

\begin{thm}[Mordell-Weil Theorem]
\label{thm:Theorem1}
For any elliptic curve over an algebraic number field $K$, the group $E(K)$ is finitely generated.
\end{thm}

%%%%%%%%%%%%%%

From the Mordell-Weil Theorem, one can define the rank of an elliptic curve:
\begin{equation}
\rank(E/K) = \rank_{\mathbb{Z}} E(K).
\end{equation} 

Given a prime number $\ell$, the Mordell-Weil Theorem implies that:
\begin{equation}
\label{eq:eq8final6}
\dim_{\mathbb{F}_\ell} E(K)/[\ell](E(K)) = \rank(E/K) + \dim_{\mathbb{F}_\ell} E[\ell](K),
\end{equation}
where $E[\ell](K)$ is the group of $\ell$-torsion points of $E(K)$. 
Considering the special case where $K=\mathbb{Q}$, 
a Theorem of Mazur implies that $\# E_{tor}(\mathbb{Q})\leq 16$ \cite[Theorem 7.5, p. 242]{silverman2009}, so that 
$\dim_{\mathbb{F}_\ell} E[\ell](\mathbb{Q})=0$ for almost all prime numbers $\ell$. 
Then, $\rank(E/\mathbb{Q}) = \dim_{\mathbb{F}_\ell} E(\mathbb{Q})/[\ell](E(\mathbb{Q}))$.

Next, we recall the following results on the $L$-series attached to elliptic curves. 
Given an elliptic curve $E$ over a global field $K$, 
one defines the auxiliary function \cite[p. 450]{silverman2009}:
\begin{equation}
\Lambda(E,s) = N_E^{s/2} (2 \pi)^{-s} \Gamma(s) L(E,s),
\end{equation} 
where $N_E$ denotes the conductor of $E/K$,  $\Gamma$ denotes the Euler gamma function, and $L(E,s)$ is the $L$-series attached to $E$
\cite[pp. 449--452]{silverman2009}.

A construction of Eichler and Shimura associates to special functions, called modular forms \cite[Section C.12]{silverman2009}, a type of elliptic curves, called modular elliptic curves \cite[Section C.13]{silverman2009}, for which analytic continuation of the auxiliary function $\Lambda(E,s)$ to the entire complex plane can be demonstrated \cite[Section C.16]{silverman2009}. See \cite{murty1995} for a brief introduction.

Now, the Taniyama-Shimura-Weil conjecture states that any elliptic curve over $\mathbb{Q}$ is modular. This conjecture was proved in 1995 for semi-stable elliptic curves over $\mathbb{Q}$ \cite{wiles1995,taylor1995} and then, the proof was extended to cover all elliptic curves over $\mathbb{Q}$ \cite{breuil2001}. From there, one can conclude that $L(E,s)$ has an analytic continuation on the entire complex plane. 

%%%%%%%%%%%%%%

\begin{thm}[Wiles 1995, Taylor and Wiles 1995, Breuil et al. 2001]
\label{thm:Theorem2}
Modularity Theorem: Any elliptic curve $E$ over $\mathbb{Q}$ is modular.
\end{thm}

%%%%%%%%%%%%%%

Based on Theorem \ref{thm:Theorem2}, one can define the analytic rank of an elliptic curve over $\mathbb{Q}$:
\begin{equation}
\rank_{an}(E/\mathbb{Q}) = \ord_{s=1} L(E,s).
\end{equation} 

The Birch and Swinnerton-Dyer conjectures include the following statement \cite[Conjecture 16.5, part a), p. 452]{silverman2009}:
\begin{itemize}
\item[] BSD-1: $\rank(E/\mathbb{Q}) = \rank_{an}(E/\mathbb{Q})$, for any elliptic curve $E$ over $\mathbb{Q}$.
\end{itemize}

The Birch and Swinnerton-Dyer conjectures assume also the following conjecture \cite[p. 341]{silverman2009} formulated  
independently by Shafarevich and Tate \cite{rubin2002}:
\begin{itemize}
\item[] S-T: the Shafarevich-Tate group $\Sha(E/\mathbb{Q})$ is finite, for any elliptic curve $E$ over $\mathbb{Q}$.
\end{itemize}
From S-T, one would have:
\begin{equation} 
\rank(E/\mathbb{Q}) = \dim_{\mathbb{F}_\ell} E(\mathbb{Q})/[\ell](E(\mathbb{Q})) = \dim_{\mathbb{F}_\ell} S^{[\ell]}(E/\mathbb{Q}),
\end{equation}
for all but finitely many prime numbers $\ell$, based on Mazur's Theorem on torsion points.

Conjecture S-T is assumed in the second part of the BSD conjectures (BSD-2) on the value of the leading coefficient in the Taylor expansion of $L(E,s)$ at $s=1$. See \cite[Conjecture 16.5, part b), p. 452]{silverman2009}.

A result of Coates and Wiles \cite{coates1977} states that BSD-1 holds in the case of an elliptic curve over $\mathbb{Q}$ with complex multiplication (CM), if the analytic rank is equal to $0$. Rubin \cite{rubin1987} proved that conjecture S-T holds under the same conditions. 

%%%%%%%%%%%%%%

\begin{thm}[Coates and Wiles 1977, Rubin 1987]
\label{thm:Theorem3} 
Let $E$ be an elliptic curve over $\mathbb{Q}$ with CM. Assume that $\rank_{an}(E/\mathbb{Q})= 0$. Then, conjectures BSD-1 and S-T hold.
\end{thm}

%%%%%%%%%%%%%%

Kolyvagin extended Rubin's result to any elliptic curve of analytic rank at most $1$ 
\cite{kolyvagin1988,kolyvagin1988b,kolyvagin1991}, building on a result of Gross and Zagier 
\cite{gross1986}; see also \cite[Theorem 1, p. 430]{kolyvagin1990}.

%%%%%%%%%%%%%%

\begin{thm}[Gross and Zagier 1986, Kolyvagin 1988--1991]
\label{thm:Theorem4}
Let $E$ be an elliptic curve over $\mathbb{Q}$. Assume that $\rank_{an}(E/\mathbb{Q})\leq 1$. Then, conjectures BSD-1 and S-T hold.
\end{thm}

%%%%%%%%%%%%%%

The reader may consult \cite{rubin1999} and \cite{rubin2002} for further reading on the Birch and Swinnerton-Dyer conjectures and the rank of elliptic curves. In particular, the notions of Heegner points and Euler systems are explained. In this work, these notions
 do not intervene.

Results on the finiteness of the Shafarevich-Tate group have been obtained by Kolyvagin and Logachev in the case of Abelian varieties
\cite{logachev1989,logachev1991}. 

The finiteness of $\Sha(E/\mathbb{Q})$ implies that its order is a perfect square, based on Cassels' pairing for elliptic curves 
\cite[p. 341]{silverman2009}. Poonen and Stoll have studied the Cassels-Tate pairing in the case of Abelian varieties 
\cite{poonen1999}.

%%%%%%%%%%%%%%%%%%%%%%%%%%%%%%%%%%%%%%%%%%%%%%%%%%%%

\subsection{Main theorem of this work and consequences}
\label{subsection:MainTheorems}

Firstly, as mentioned in \cite[Section 3]{poonen1999b}, the Shafarevich-Tate group is a torsion group, for the simple reason that it is a subgroup of the torsion group $\Homol^1(\mathcal{G},E(\overline{K}))$, having considered Galois cohomology. 

The following proposition states equivalent formulations.

%%%%%%%%%%%%%%

\begin{prop}
\label{thm:Proposition1}
Let $E$ be an elliptic curve over an algebraic number field $K$. Let $\ell$ be a prime number. Then, the following conditions are
equivalent:

a) \begin{equation}
\Sha(E/K)_{\ell} = 0,
\end{equation}
where $\Sha(E/K)_{\ell}$ denotes the $\ell$-primary component of $\Sha(E/K)$.

b) \begin{equation}
\Sha(E/K)[\ell] = 0.
\end{equation}

c) \begin{equation}
\partial: E(K)/[\ell](E(K)) \xrightarrow{\approx} S^{[\ell]}(E/K).
\end{equation}

d) \begin{equation}
S^{[\ell]}(E/K) = \Ker\{ \Homol^{1}(\mathcal{G},E[\ell]) \rightarrow 
 \Homol^{1}(\mathcal{G},E(\overline{K})) \}.
\end{equation}
\end{prop}

\begin{proof}
$a) \Rightarrow b)$. This is clear since $\Sha(E/K)[\ell] \subseteq \Sha(E/K)_{\ell}$.

$b) \Rightarrow a)$. Assume that $\Sha(E/K)[\ell]=0$. 
Let $A$ be any finite subgroup of $\Sha(E/K)_{\ell}$. Then, $A$ is of the form
$\oplus_{i=1}^{n} \mathbb{Z}/\ell^{m_i}\mathbb{Z}$, where $m_i\in \mathbb{N}$. Now, one must have 
$n \leq \dim_{\mathbb{F}_\ell} \Sha(E/K)[\ell]$. Thus, $n=0$ and $\Sha(E/K)_{\ell}=0$.

$b) \Leftrightarrow c)$. This follows from the short exact sequence:
\begin{equation}
\label{eq:eq16final6}
0 \rightarrow E(K)/[\ell](E(K)) \xrightarrow{\partial} S^{[\ell]}(E/K) 
\rightarrow \Sha(E/K)[\ell] \rightarrow 0.
\end{equation}

$c) \Leftrightarrow d)$. From the short exact sequence (\ref{eq:eq2final6}), one 
has an isomorphism:
\begin{equation}
\partial: E(K)/[\ell](E(K)) \xrightarrow{\approx} 
\Ker \Bigl \{ \Homol^{1}(\mathcal{G},E[\ell]) 
\rightarrow \Homol^{1}(\mathcal{G},E(\overline{K})) \Bigr \}.
\end{equation}
From the compositum of homomorphisms at any place $v$ of $K$:
\begin{equation}  
\Homol^{1}(\mathcal{G},E[\ell]) \rightarrow \Homol^{1}(\mathcal{G},E(\overline{K}))
\xrightarrow{Res^{\mathcal{G}}_{\mathcal{G}_v}} \Homol^{1}(\mathcal{G}_v,E(\overline{K}_v)),
\end{equation}
one obtains
\begin{equation}  
\partial: E(K)/[\ell](E(K)) \xrightarrow{\approx}
\Ker \Bigl \{ \Homol^{1}(\mathcal{G},E[\ell]) 
\rightarrow \Homol^{1}(\mathcal{G},E(\overline{K})) \Bigr \} \subseteq S^{[\ell]}(E/K).
\end{equation}
The equivalence is now clear.
\end{proof}

%%%%%%%%%%%%%%

{\bf Remark 1.}
\label{remark1}
Part d) of the proposition expresses a local-global principle: $f \in\Homol^{1}(\mathcal{G},E[\ell])$ 
splits in $E(\overline{K}_v)$ for all places $v$ if and only if it does in $E(\overline{K})$. 
The obstruction to this principle is thus $\Sha(E/K)_{\ell}$ in view of part a) of the proposition.\\

%%%%%%%%%%%%%%

We now state the main results of this paper. 
As in \cite[IV-2.1]{serre1968}, we say that an elliptic curve has complex multiplication (CM), if it does over some finite extension 
$F_{\textrm{\tiny CM}}/\mathbb{Q}$; {\em i.e.}, $\End_{F_{\textrm{\tiny CM}}}(E)$ is an order in an imaginary quadratic field 
$K_{\textrm{\tiny CM}}$ \cite[Section 5]{rubin1999}. 

Given a prime number $\ell$, one has a Galois representation $\rho_\ell: \mathcal{G} \rightarrow {\bf GL}_2(T_\ell)$, 
obtained by Galois action on the Tate module $T_\ell$ of $E$. This representation identifies 
$\Gal(L_{\infty}/\mathbb{Q})$ with 
$\rho_\ell(\mathcal{G})$, where $L_\infty$ denotes the field obtained by adjoining to $\mathbb{Q}$ the affine coordinates of all $\ell^n$-torsion points of $E$, with $n\geq 1$. 

%%%%%%%%%%%%%%

\begin{thm}
\label{thm:Theorem5}
Let $E$ be an elliptic curve over $\mathbb{Q}$ without CM, and consider a Weierstrass equation of the form $y^2=x^3+Ax+B$,
 with $A,B\in \mathbb{Z}$. 

Let $\ell \not = 2,3,5,7,13$ be a prime number. Assume that: i) $\rho_\ell(\mathcal{G})$ 
is the full linear group ${\bf GL}_2(\mathbb{Z}_\ell)$; and ii) $\ell \nmid \Delta^\prime := 4A^3+27 B^2$. 

Then, one has:
\begin{equation}
S^{[\ell]}(E/\mathbb{Q}) = \Ker \Bigl \{ \Homol^{1}(\mathcal{G},E[\ell]) \rightarrow 
 \Homol^{1}(\mathcal{G},E(\overline{\mathbb{Q}})) \Bigr \}.
\end{equation} 
\end{thm}

%%%%%%%%%%%%%%

Note that from Serre's Theorems \cite[Th\'eor\`eme 2, p. 294]{serre1972} and \cite[Th\'eor\`eme 4$^\prime$, p. 300]{serre1972}, 
it follows that 
the Galois group $\rho_\ell(\mathcal{G})$ is the full linear group for almost all primes, whenever $E$ has no CM. 
Furthermore, in the case $E$ is semi-stable ({\em i.e.}, with no additive reduction) without CM, 
Mazur's Theorem \cite[Theorem 4, p. 131]{mazur1978} implies that $\rho_\ell(\mathcal{G})$ is the full linear group for $\ell\geq 11$.
It follows that $\Sha(E/\mathbb{Q})_{\ell}=0$, for any $\ell \geq 17$ not dividing $\Delta^\prime$, whenever $E$ is semi-stable. 
  
Proposition \ref{thm:Proposition1} states that Theorem \ref{thm:Theorem5} implies the following consequences.

%%%%%%%%%%%%%%

\begin{cor}
\label{thm:Corollary1}
Let $E$ be any elliptic curve over $\mathbb{Q}$ without CM. 
Then, for almost all prime numbers $\ell$, one has:

a) \begin{equation}
\Sha(E/\mathbb{Q})_{\ell} = 0,
\end{equation}
where $\Sha(E/\mathbb{Q})_{\ell}$ denotes the $\ell$-primary component of $\Sha(E/\mathbb{Q})$;

b) 
\begin{equation}
\Sha(E/\mathbb{Q})[\ell] = 0;
\end{equation}

c) 
\begin{equation}
\partial: E(\mathbb{Q})/[\ell](E(\mathbb{Q})) \xrightarrow{\approx} S^{[\ell]}(E/\mathbb{Q}).
\end{equation}
\end{cor}

%%%%%%%%%%%%%%

Mazur's Theorem on torsion points \cite[Theorem 7.5, p. 242]{silverman2009} then implies the following result.

%%%%%%%%%%%%%%

\begin{cor}
\label{thm:Corollary2}
Let $E$ be any elliptic curve over $\mathbb{Q}$ without CM. 
Then, for almost all prime numbers $\ell$, one has:
\begin{equation}
\rank(E/\mathbb{Q}) = \dim_{\mathbb{F}_\ell} S^{[\ell]}(E/\mathbb{Q}).
\end{equation}
\end{cor}

%%%%%%%%%%%%%%

Since $\Sha(E/\mathbb{Q})[\ell]$ is finite for any prime $\ell$, 
as it is a quotient group of the finite group $S^{[\ell]}(E/\mathbb{Q})$, it follows that 
the $\ell$-primary component of $\Sha(E/\mathbb{Q})$ is of the form:
\begin{equation}
\Sha(E/\mathbb{Q})_{\ell} = \left ( \mathbb{Q}_\ell/\mathbb{Z}_\ell \right )^{n_\ell} \oplus T_\ell,
\end{equation}
where $n_\ell\geq 0$ and $T_\ell$ is a finite $\ell$-group \cite[Section 12]{poonen1999b}. 

Thus, based on Theorem \ref{thm:Theorem5},  
the only missing piece to proving that $\Sha(E/\mathbb{Q})$ is finite in the non-CM case, 
is a proof that $\Sha(E/\mathbb{Q})$ has no infinitely divisible element. 
See also \cite[p. 341]{silverman2009} on this issue.

Examples of CM elliptic curves of rank $2$ or $3$ with endomorphism ring $\mathbb{Z}[i]$ are studied in \cite{coates2010}.  
The statement of \cite[Theorems 1.2]{coates2010} assumes the condition $\ell \equiv 1 \mod 4$, and 
the very strong condition $\ell<30,000$ (and $\ell \not = 41$), in the case of a specific curve. 
In the case of \cite[Theorems 1.3]{coates2010}, the condition $\ell \equiv 1 \mod 4$ is also assumed, and the extra restriction that 
$\ell<30,000$ (except for finitely many exceptions), and the statement is valid for $5$ specific elliptic curves. 
In contrast, Theorem \ref{thm:Theorem5} is valid for any elliptic curve without CM, and all primes $\ell$, but finitely many. 
However, we have not succeeded in carrying out the strategy of our proof to the CM case, as of now.

In Section \ref{section:examples}, an example from \cite{penney1975} of an elliptic curve $E$ over the rationals without CM of rank at least $7$ is mentioned. 
Furthermore, we show that, in this example, $\ell=41$ is the smallest prime 
({\em i.e.}, based on the conditions of Theorem \ref{thm:Theorem5}) for which Corollary \ref{thm:Corollary2} applies. 
Therefore, one can in principle find out the exact rank of $E/\mathbb{Q}$ from a computation of 
$\dim_{\mathbb{F}_{41}} S^{[41]}(E/\mathbb{Q})$.

It can be noticed that this example solves the open problem mentioned in 
\cite[Problem 2.16, p. 27]{stein2007} in the non-CM case. 

An example that was communicated to us by Professor C. Wuthrich is also mentioned in Section \ref{section:examples}. 
This example shows that the condition $\rho_\ell(\mathcal{G})={\bf GL}_2(\mathbb{Z}_\ell)$ is not sufficient to conclude 
that $\Sha(E/\mathbb{Q})_\ell=0$, if ever $\ell$ is one of the exceptional ones ({\em i.e.}, $2$, $3$, $5$, $7$, or $13$). 
This issue is crucial, in view of BSD-2.

We end this paper with a complement to Proposition \ref{thm:Proposition2} that clarifies its proof, but that is not needed as such for 
the proof of Theorem \ref{thm:Theorem5} that is presented here.

%%%%%%%%%%%%%%%%%%%%%%%%%%%%%%%%%%%%%%%%%%%%%%%%%%%%
%%%%%%%%%%%%%%%%%%%%%%%%%%%%%%%%%%%%%%%%%%%%%%%%%%%%

\section{Background on elliptic curves}
\label{sec:background}

%%%%%%%%%%%%%%%%%%%%%%%%%%%%%%%%%%%%%%%%%%%%%%%%%%%%

\subsection{Basic notions}
\label{subsection:basicElliptic}

Let $K$ be a field and $E$ be an elliptic curve over $K$; {\em i.e.}, a smooth projective curve of genus $1$, together 
with a base point $O$. The elliptic curve admits a Weierstrass equation \cite[p. 42]{silverman2009}: 
\begin{equation}
\label{eq:eq26final6}
y^2+a_1xy+a_3y=x^3+a_2x^2+a_4x+a_6,
\end{equation}
with coefficients $a_i\in K$, $i=1,2,3,4,6$.   
One defines the quantities:
\begin{equation} 
\begin{cases}
b_2=a_1^2+4a_2;\\
b_4=2a_4+a_1a_3;\\
b_6=a_3^2+4a_6;\\
b_8=a_1^2a_6+4a_2a_6-a_1a_3a_4+a_2a_3^2-a_4^2.
\end{cases}
\end{equation}
We also set:
\begin{equation} 
\begin{cases}
c_4=b_2^2-24b_4;\\
c_6=-b_2^3+36b_2b_4-216 b_6.
\end{cases}
\end{equation}
Then, the discriminant $\Delta$ of $E$ corresponding to a given Weierstrass equation is equal to:
\begin{equation} 
\Delta(E):=-b_2^2b_8-8b_4^3-27b_6^2+9b_2b_4b_6,
\end{equation}
and its $j$-invariant (independent of the Weierstrass equation) is equal to:
\begin{equation} 
j(E):=c_4^3/\Delta.
\end{equation}

Given the cubic curve defined by a Weierstrass equation (\ref{eq:eq26final6}), there are three cases \cite[p. 45]{silverman2009}:

(1) The curve is non-singular if $\Delta\not =0$.

(2) The curve has a node if $\Delta=0$ and $c_4\not =0$.

(3) The curve has a cusp if $\Delta=0$ and $c_4=0$.

In cases (2) and (3), there is only one singular point. In case (1), the curve is an elliptic curve with base point $O=[0,1,0]$. 

The elliptic curve has also Weierstrass equation $y^2=x^2-27 c_4 x -54 c_6$, if the characteristic of $K$ is different from $2$ and $3$ \cite[p. 43]{silverman2009}. Thus, it is of the form $y^2=x^3+Ax+B$. Two elliptic curves are isomorphic over $\overline{K}$ if and only if they have the same $j$-invariant \cite[p. 45]{silverman2009}. If $K$ has characteristic different from $2$ and $3$, the proof of that result \cite[pp. 46--47]{silverman2009} shows that an isomorphism holds over a base extension obtained by adjoining $(A/A^\prime)^{1/4}$ (case $j=1728$) or $(B/B^\prime)^{1/6}$ (case $j=0$) or $(A/A^\prime)^{1/4}=(B/B^\prime)^{1/6}$ (other cases) to $K$, where the two curves have equations $y^2=x^3+Ax+B$ and $y^2=x^3+A^\prime x+B^\prime$ over $K$, respectively. So, unless $j=0$ or $1728$, the base field extension has degree dividing $2$ (the g.c.d. of $4$ and $6$). 

There is a group law defined on $E(K)$ that is a consequence of a special case of Bezout's Theorem, but that can also be defined 
explicitly. See \cite[Chapter II, Section \S 2]{silverman2009}. 

Now, let $\ell$ be a prime number. If the characteristic of $K$ is different from $\ell$, then the group $E[\ell]=E[\ell](\overline{K})$ of $\ell$-torsion points of $E$ is isomorphic to $\mathbb{Z}/\ell\mathbb{Z} \oplus Z/\ell \mathbb{Z}$.
 If $K$ has characteristic $\ell$, then $E[\ell]$ is isomorphic to $0$ or $\mathbb{Z}/\ell\mathbb{Z}$. See \cite[p. 86]{silverman2009}. 

If $m$ is a positive integer coprime with the characteristic of $K$, then there is the Weil pairing $e_m: E[m] \times E[m] \rightarrow \mu_m$, 
which is bilinear, alternating, non-degenerate, Galois invariant, and compatible  \cite[Proposition 8.1, p. 94]{silverman2009}. As a consequence, one deduces that $\mu_m \subset K$, if $E[m] \subset E(K)$, under the condition $\cha (K) \nmid m$ \cite[Corollary 8.1.1, p. 96]{silverman2009}. 

Given an elliptic curve over a field $K$, one constructs its formal group $F$ as in \cite[pp. 115-120]{silverman2009}. 
If $K$ has characteristic $\ell$, multiplication by $\ell$ in $F$ (denoted $\ell[X]\in K[[X]]$) is either $0$ or else is of the form $g(X^{\ell^h})$, where $g^\prime(0)\not = 0$ \cite{kolyvagin1980}. In the latter case, $h$ is called the height of $F$.

Let $k$ be a finite field of characteristic $\ell$ and $\widetilde E$ be an elliptic curve over $k$. Then, either \cite[p. 144--145]{silverman2009}:

(1) The formal group of $\widetilde E$ has height $h=2$ and $\widetilde E[\ell]=0$ (the Hasse invariant is $0$, or the curve is supersingular);

or

(2) The formal group of $\widetilde E$ has height $1$ and $\widetilde E[\ell]=\mathbb{Z}/\ell\mathbb{Z}$ (the Hasse invariant is $1$, or the curve is ordinary).

The first case occurs if and only if $j(\widetilde E)\in \mathbb{F}_{\ell^2}$ and the map $[\ell]$ is purely inseparable.

%%%%%%%%%%%%%%%%%%%%%%%%%%%%%%%%%%%%%%%%%%%%%%%%%%%%

\subsection{Elliptic curves over local fields}
\label{subsection:localElliptic}

Let $K$ be a finite extension of ${\mathbb{Q}}_p$ and let $\overline{K}$ be its algebraic closure. Let $v$ be the discrete valuation of $K$. Given an elliptic curve $E$ over $K$, we consider its minimal Weierstrass equation \cite[pp. 185--187]{silverman2009}. That is a Weierstrass equation with coefficients in the integer ring $\mathcal{O}_v$ of $K$ with minimal value of $v(\Delta)$ among all such equations. Therefore, one can look at its reduction $\widetilde E$ modulo a uniformizer $\pi_v$ of $K$ \cite[p. 187--188]{silverman2009}, defined over the residue field $k_v$ of $K$.  One says \cite[pp. 196--197]{silverman2009}:

1) $E$ has good reduction if $\widetilde E$ is non-singular ($v(\Delta)=0$).

2) $E$ has multiplicative reduction if $\widetilde E$ has a node ($v(\Delta)>0$ and $v(c_4)=0$).

3) $E$ has additive reduction if $\widetilde E$ has a cusp ($v(\Delta), v(c_4) >0$).

The set of non-singular points $\widetilde E_{ns}(\overline{k}_v)$ of the reduced curve forms a group \cite[p. 56]{silverman2009}. 
In the case of good reduction $\widetilde E_{ns}(\overline{k}_v)=\widetilde E(\overline{k}_v)$ is an elliptic curve 
defined over $k_v$. In the case of multiplicative reduction, $\widetilde E_{ns}(\overline{k}_v)\approx \overline{k}_v^*$. In the case of additive reduction, 
$\widetilde E_{ns}(\overline{k}_v)\approx \overline{k}_v^+$. See also \cite[Exercise 3.5, p. 105]{silverman2009}.

A sufficient condition for a Weierstrass equation to be minimal is that $v(\Delta)<12$ or that $v(c_4)<4$ \cite[Remark 1.1, p. 186]{silverman2009}. Therefore, in the case of good reduction ($v(\Delta)=0$) 
or multiplicative reduction ($v(c_4)=0$), a minimal Weierstrass equation remains minimal after base field extension 
\cite[Proposition 5.4.(b), p. 197]{silverman2009}. 
In the case of additive reduction, after a suitable finite base field extension (see below), the reduction turns either good  or multiplicative. For an example of the former case, see \cite[Example 5.2, p. 196--197]{silverman2009}. For an example of the latter case, let $p$ be a prime number greater than $3$ and consider $E\::\: y^2=x^3+\sqrt[3]{p}x^2+p^2$ over $K=\mathbb{Q}_p(\sqrt[3]{p})$; then, over $K(\sqrt[2]{p})$, $E$ has Weierstrass equation $y^2=x^3+x^2+p$, as can be seen with the change of variable $y=\sqrt[2]{p}y^\prime$ and $x=\sqrt[3]{p}x^\prime$.

Next, recall that $E$ has good reduction after a base extension (potential good reduction) if and only if its $j$-invariant is an integer of $K$ \cite[p. 197]{silverman2009}. The proof of this result in the case $\cha(k_v)\not = 2$ \cite[p. 199]{silverman2009} relies on a Weierstrass equation in Legendre form $y^2=x(x-1)(x-\lambda)$, $\lambda\not = 0,1$, \cite[p. 49]{silverman2009}. Such an equation can be obtained after adjoining the roots of the cubic polynomial $x^3+(b_2/4) x^2+(b_4/2) x+b_6/4=(x-e_1)(x-e_2)(x-e_3)$ and then adjoining the square root of $e_2-e_1$. Thus, the base field extension $K^\prime/K$ can be taken of degree dividing $12$. If $\cha(k_v) =2$, the proof relies on a Weierstrass equation in Deuring normal form $y^2+\alpha x y + y=x^3$, $\alpha^3\not = 27$. Such an equation
 is obtained after adjoining a root $\alpha$ of the polynomial $x^3(x^3-24)^3-(x^3-27)j(E)$, yielding a base field extension of degree dividing $3d^\prime$ with $1 \leq d^\prime \leq 4$, and then over an extra base field extension of degree $2$, $4$ or $6$ to obtain an isomorphism with the initial elliptic curve \cite[Proposition 1.3, p. 412, and p. 47]{silverman2009}. 
In all cases, the base field extension $K^\prime/K$ has degree $d$ divisible only by powers of $2$ and $3$. 

If $\cha(k_v)\not =2$, consideration of a Weierstrass equation in Legendre form over a field extension $K^\prime$ of degree dividing 
$12$ shows that $E$ has either good or multiplicative reduction over $K^\prime$ \cite[p. 198]{silverman2009}. 
If $\cha(k_v) =2$, one considers a Weierstrass equation in Deuring normal form over a field extension $K^\prime/K$ of degree $d$ 
with only $2$ or $3$ as prime factors \cite[p. 413]{silverman2009}.  

There is a well-defined reduction map $E(K)\rightarrow \widetilde E(k_v)$ \cite[p. 188]{silverman2009}. 
Let $E_0(K)$ denote the pre-image of $\widetilde E_{ns}(k_v)$ under the reduction map. 
 Then, there is an exact sequence 
\begin{equation}
0 \rightarrow E_1(K) \rightarrow E_0(K) \rightarrow \widetilde E_{ns}(k_v) \rightarrow 0,
\end{equation}
 where the second map is the reduction map, the first map is inclusion, and $E_1(K)$ consists of all points that reduce to the point $\widetilde O$ of $\widetilde E(k_v)$ \cite[pp. 187--188]{silverman2009}. From the above remark, in the case of good reduction or multiplicative reduction, the above sequence extends to a short exact sequence of Galois modules:
\begin{equation}
0 \rightarrow E_1(\overline{K}) \rightarrow E_0(\overline{K}) \rightarrow \widetilde E_{ns}(\overline{k}_v) \rightarrow 0,
\end{equation}
where $E_{ns}(\overline{k}_v)\approx \overline{k}_v^*$ in the case of multiplicative reduction.

Moreover, there is an isomorphism 
\begin{equation}
F_v({\mathcal M}_v)\approx E_1(K),
\end{equation} 
where $F_v$ is the formal group of $E$ over $\mathcal{O}_v$ and ${\mathcal M}_v$ is the maximal ideal of 
$\mathcal{O}_{v}$ \cite[p. 191]{silverman2009}. 

Now, let $\ell$ be a prime number (possibly different from $p$) and consider a finite extension $L/K$ with
valuation $w$.
Based on the above facts, we obtain an exact sequence of Abelian groups: 
\begin{equation}
0 \rightarrow W_1 \approx E_1[\ell] \rightarrow E_{0}[\ell] \rightarrow \widetilde E_{ns}[\ell].
\end{equation} 
Here, $W_1$ is the group of $\ell$-torsion points of $F_w({\mathcal M}_w)$, and $E_{1}[\ell]$, $E_{0}[\ell]$ 
and $\widetilde E_{ns}[\ell]$ are the groups of $\ell$ torsion points of $E_1(L)$, $E_0(L)$ and 
$\widetilde E_{ns}(k_w)$, respectively. Also, reduction is with respect to a minimal Weierstrass equation for $E$ over $L$ (not necessarily the same one as over $K$ in the case of additive reduction over $K$). 

We also have an exact sequence of Abelian groups:
\begin{equation}
0 \rightarrow E_{0}[\ell] \rightarrow E[\ell] \approx \mathbb{Z}/\ell\mathbb{Z} \times \mathbb{Z}/\ell\mathbb{Z}  
\rightarrow E[\ell]/E_{0}[\ell] \rightarrow 0.
\end{equation}
The Kodaira-N\'eron Theorem \cite[Theorem 6.1, p. 200]{silverman2009} states that the Abelian group $E[\ell]/E_{0}[\ell]$ has order at most $4$ except possibly in the case of split multiplicative reduction, in which case $E[\ell]/E_{0}[\ell]$ is cyclic of order $v(\Delta)=-v(j)$.

%%%%%%%%%%%%%%

 \begin{lem}
\label{thm:Lemma1}
Let $E$ be an elliptic curve defined over a local field $K$, with bad reduction. Let $\ell>3$ be a prime number different from the characteristic $p$ of the residue field $k_v$ of $K$. If $E$ has potential good reduction, then $E$ has good reduction over $L=K(E[\ell])$. If $E$ has potential multiplicative reduction, then $E$ has multiplicative reduction over $L$. 
\end{lem}
\begin{proof}
By way of contradiction, assume that $E$ has additive reduction over $L$. Then, there is a short exact sequence of Abelian groups:
\begin{equation}
0 \rightarrow E_1(L) \rightarrow E_0(L) \rightarrow \widetilde E_{ns}(k_w)\approx k_w^+ \rightarrow 0,
\end{equation}
where $k_w$ denotes the residue field of $L$. This yields a short exact sequence:
\begin{equation}
0 \rightarrow E_1[\ell] \rightarrow E_0[\ell] \rightarrow k_w^+[\ell],
\end{equation}
since $E_1[\ell],E_0[\ell] \subseteq E[\ell]$, as $L=K(E[\ell])$. 
But since $\ell\not = p$, it follows that $E_1[\ell]=0$ and that $k_w^+[\ell]=0$. Therefore, one obtains that $E_0[\ell]=0$. 
Now, let $o$ be the order of $E(L)/E_0(L)$. Then, an element $P$ of $E[\ell]$ satisfies both conditions $[\ell] P=O\in E_0(L)$ 
and $[o] P\in E_0(L)$. Therefore, since $(o,\ell)=1$, as $\ell>3$, one concludes that $P\in E_0[\ell]$. This means that $E[\ell]\subseteq E_0[\ell]=0$. However, this conclusion contradicts the fact that 
$E[\ell]\approx \mathbb{Z}/\ell \mathbb{Z} \oplus \mathbb{Z}/\ell \mathbb{Z}$,
 as $L=K(E[\ell])$. 

So, if $E$ has bad potentially good reduction ({\em i.e.}, potential good reduction), then $E$ must have good reduction over $L$. 
If $E$ has bad and no potential good reduction ({\em i.e.}, potential multiplicative reduction), then $E$ must have multiplicative reduction over $L$.
\end{proof}

%%%%%%%%%%%%%%

%%%%%%%%%%%%%%

 \begin{lem}
\label{thm:Lemma2}
Let $E$ be an elliptic curve defined over a local field $K$, with good reduction. Let $m$ be a positive integer coprime with
 the characteristic $p$ of the residue field $k_v$ of $K$. Then, one has an isomorphism:
\begin{equation}
E[m](K) \xrightarrow{\approx} \widetilde E[m](k_v). 
\end{equation}
\end{lem}
\begin{proof}
Firstly, there is a well-defined map $E[m](K) \rightarrow \widetilde E[m](k_v)$ obtained by restriction of the 
reduction map $E(K) \rightarrow \widetilde E(k_v)$. Next, this map is one-to-one, having assumed that $(m,p)=1$ and 
that $\widetilde E$ is non-singular \cite[Proposition 3.1, p. 192]{silverman2009}. 
Lastly, the Criterion of N\'eron-Ogg-Shafarevich implies that the extension 
$K(E[m])/K$ is unramified, having assumed good reduction and $(m,p)=1$ \cite[Theorem 7.1, p. 201]{silverman2009}. 
It follows that, for any $Q\in E[m](\overline{K})$, 
the degree of the extension $K(Q)/K$ is equal to its residue degree. Thus, a torsion point 
in $\widetilde E[m](k_v)$ can be lifted to a torsion point in $E[m](K)$. 
\end{proof} 

%%%%%%%%%%%%%%

%%%%%%%%%%%%%%%%%%%%%%%%%%%%%%%%%%%%%%%%%%%%%%%%%%%%

\subsection{Elliptic curves over an algebraic number field}
\label{subsection:globalEllipticK}

Let $E$ be an elliptic curve over an algebraic number field $K$. Then, $E$ admits a Weierstrass equation of the form 
$y^2=x^3+Ax+B$, with $A,B\in \mathcal{O}_K$, where $\mathcal{O}_K$ denotes the integer ring of $K$. 
Indeed, $E$ has a Weierstrass equation over $K$ of the form $y^2 = x^3 - 27 c_4 x - 54 c_6$, with $c_4,c_6\in K$ 
\cite[pp. 42--43]{silverman2009}. 
Writing $c_4=C_4/d$ and $c_6=C_6/d$, with $C_4,C_6,d\in \mathcal{O}_K$, one obtains the Weierstrass equation 
$y^2=x^3 - 27  C_4 d^3 x - 54 C_6 d^5$, upon replacing $(x,y)$ by $(x/d^2,y/d^3)$. 
Thus, $E$ has a Weierstrass equation 
of the form $y^2=x^3+Ax+B$, upon taking $A=- 27  C_4 d^3,B=- 54 C_6 d^5\in \mathcal{O}_K$. 

The elliptic curve $E$ admits a global minimal Weierstrass equation with coefficients in the integer ring of the Hilbert class field of $K$ \cite[Corollary 8.3, p. 245]{silverman2009}. 

We denote $\Sigma_E$ the set of places at which $E$ has bad reduction. 
The set $\Sigma_E$ is finite \cite[Remark 1.3, p. 211]{silverman2009}. 
We let $\Sigma_{E,add}$ ($\Sigma_{E,mult}$) denote the (finite) sets of places $v$ such that $E$ has additive 
(respectively, multiplicative) reduction at $v$. We denote $\Sigma_{E,p.g.}$ the set of places at which $E$ has potential good reduction and $\Sigma_{E,p.m.}$ the set of primes at which $E$ has potential multiplicative reduction. 
Thus, there is a decomposition of $\Sigma_E$ into a disjoint union $\Sigma_{E,p.g.} \dotcup \; \Sigma_{E,p.m.}$, with  
$\Sigma_{E,p.g.} \subseteq \Sigma_{E,add}$ and $\Sigma_{E,mult} \subseteq \Sigma_{E,p.m.}$ (both inclusions are a consequence of the other one). 
A place $v$ of $\Sigma_{E,add}$ is in $\Sigma_{E,p.g.}$ if and only if $v(j(E))\geq 0$.
Here, the place $v$ is identified with the discrete valuation on the completion of $K$ at $v$, $K_{v}$, that maps $K_{v}^*$ onto $\mathbb{Z}$. 

The following cases will be considered in Section \ref{subsection:SelmerL}:

Case A: $v \mid \ell$.

Case B: $v \nmid \ell$ and $v \not \in \Sigma_E$; $E$ has good reduction at $v$ and 
$\ell$  is not equal to the characteristic of the residue field of $K_{v}$.

Case C: $v \nmid \ell$ and $v \in \Sigma_E$, with potential good reduction of $E$ at $v$; {\em i.e.}, 
$v_0\in \Sigma_{E,p.g.}$. 
Then, $v \in \Sigma_{E,add}$ and $v(j(E))\geq 0$.

Case D: $v \nmid \ell$ and $v \in \Sigma_E$, with no potential good reduction of $E$ at $v$; {\em i.e.}, 
$v\in \Sigma_{E,p.m.}$. 
Then, $v \in \Sigma_{E,add} \cup \Sigma_{E,mult}$ and $v(j(E))< 0$.

In case A, $E$ has good reduction at $v\mid \ell$ for all but finitely many primes $\ell$. Then, $E$ has either supersingular or ordinary good reduction at $v \mid \ell$, according to whether the reduced elliptic curve $\widetilde E$ is supersingular or ordinary. 

%%%%%%%%%%%%%%%%%%%%%%%%%%%%%%%%%%%%%%%%%%%%%%%%%%%%

\subsection{Elliptic curves over $\mathbb{Q}$}
\label{subsection:globalEllipticQ}

Let $E$ be an elliptic curve over $\mathbb{Q}$. Then, $E$ admits a global minimal Weierstrass equation (with coefficients in $\mathbb{Z}$) \cite[Corollary 8.3, p. 245]{silverman2009}. It is also convenient to consider a Weierstrass equation of the form 
$y^2=x^3+Ax+B$, with $A,B\in \mathbb{Z}$; for instance, see \cite[Corollary 7.2, p. 240]{silverman2009}.

One says \cite[IV-2.1]{serre1968} that $E/K$ has CM if for some finite extension $F_{\textrm{\tiny CM}}/K$, 
the endomorphism ring $\End_{F_{\textrm{\tiny CM}}}(E)$ is an order $\mathcal{O}$ of an imaginary quadratic extension 
$K_{\textrm{\tiny CM}}/\mathbb{Q}$ \cite[Section 5]{rubin1999}. 
One may assume that $F_{\textrm{\tiny CM}} \supseteq K_{\textrm{\tiny CM}}$. 
Indeed, if $F^\prime$ is a subfield of $F$, then $\End_{F^\prime}(E) \subseteq \End_{F}(E)$, 
so that one may replace $F_{\textrm{\tiny CM}}$ with the compositum $K_{\textrm{\tiny CM}} F_{\textrm{\tiny CM}}$, if necessary.  
In the case of a curve without CM, the endomorphism ring $\End_{\mathbb{C}}(E)$ is minimal; 
{\em i.e.}, it is isomorphic to $\mathbb{Z}$ \cite[Corollary 9.4, p. 102]{silverman2009}. 

Let $E$ be an elliptic curve over $\mathbb{Q}$ with CM. 
Then, $\End_{F_{\textrm{\tiny CM}}}(E)$ is of the form $\mathcal{O} = \mathbb{Z} + c\, \mathcal{O}_{\textrm{\tiny CM}}$ 
over some finite base field extension $F_{\textrm{\tiny CM}}/K_{\textrm{\tiny CM}}$,  
where $\mathcal{O}_{\textrm{\tiny CM}}$ is the integer ring of the imaginary quadratic field $K_{\textrm{\tiny CM}}$, 
and $c = 1$, $2$, or $3$ \cite{serre1967b}. 

Let $\ell$ be a prime number, and set $L=\mathbb{Q}(E[\ell])$. 
Based on \cite[Corollary 5.13]{rubin1999}, there exists an elliptic curve $E^\prime$ defined over $K_{\textrm{\tiny CM}}$, 
such that $\End_{K_{\textrm{\tiny CM}}}(E^\prime) = \mathcal{O}_{\textrm{\tiny CM}}$.
From \cite[Proposition 5.3]{rubin1999}, one has 
$\End_{F_{\textrm{\tiny CM}}}(E^\prime)=\mathcal{O}_{\textrm{\tiny CM}}$ and $E[\ell]\xrightarrow{\approx} E^\prime[\ell]$ 
{\em as Galois modules} for any prime $\ell$ coprime with $c$, a condition satisfied if $\ell>3$. 
This relation is obtained from a short exact sequence based on an isogeny:
\begin{equation}
0 \rightarrow E[c] \rightarrow E(\overline{K}_{\textrm{\tiny CM}}) \rightarrow E^\prime(\overline{K}_{\textrm{\tiny CM}}) \rightarrow 0,
\end{equation}
where $E[c]$ denotes the group of $c$-torsion points of $E$.
In particular, one deduces the identities 
$L F_{\textrm{\tiny CM}} =F_{\textrm{\tiny CM}}(E[\ell]) = F_{\textrm{\tiny CM}}(E^\prime[\ell])$.
Then, using \cite[Corollary 5.5]{rubin1999}, one obtains an embedding of groups:
\begin{equation}
\label{eq:eq40final6}
\varphi: \Gal(L F_{\textrm{\tiny CM}}/F_{\textrm{\tiny CM}})  = 
\Gal(F_{\textrm{\tiny CM}}(E^\prime[\ell])/F_{\textrm{\tiny CM}}) 
\hookrightarrow \left ( \mathcal{O}_{\textrm{\tiny CM}}/(\ell) \right )^*,
\end{equation}

From (\ref{eq:eq40final6}), it follows that $\ell \nmid \vert \Gal(\mathbb{Q}(E[\ell])/\mathbb{Q}) \vert$ in the CM case, 
unless possibly if $\ell$ ramifies in $K_{\textrm{\tiny CM}}$ or $\ell$ divides $[F_{\textrm{\tiny CM}}:\mathbb{Q}]$. 
In particular, $\Gal(\mathbb{Q}(E[\ell])/\mathbb{Q})$ is not the full 
linear group for almost all primes $\ell$, as $\vert {\bf GL}_2(\mathbb{F}_\ell) \vert$ is divisible by $\ell$.

In contrast, in the non-CM case, the Galois group $\Gal(\mathbb{Q}(E[\ell])/\mathbb{Q})$ is the full linear group for almost all primes $\ell$ \cite[Th\'eor\`eme 2, p. 294]{serre1972}; 
{\em i.e.}, the representation $\widetilde \rho_\ell:\mathcal{G}\rightarrow {\bf GL}_2(\mathbb{F}_\ell)$ obtained by Galois action on $\ell$-torsion points is surjective for almost all primes $\ell$.
Moreover, sufficient conditions for the isomorphism $\widetilde \rho_\ell:\Gal(\mathbb{Q}(E[\ell])/\mathbb{Q}) \xrightarrow{\approx} {\bf GL}_2(\mathbb{F}_\ell)$ to hold  
at a specific prime $\ell$ are presented in Serre's work, together with several examples, in the case
of semi-stable curves \cite[\S 5.4 and 5.5, p. 305--311]{serre1972}, as well as  
non semi-stable curves \cite[\S 5.6 to 5.10, p. 311--323]{serre1972}.  
See also \cite{mazur1978} for further results. 

Note that from \cite[Proposition, IV-19]{serre1968} and \cite[Th\'eor\`eme 2, p. 294]{serre1972}, one has for
almost all primes $\ell$, an isomorphism $\Gal(\mathbb{Q}(E[\ell^n])/\mathbb{Q}) \xrightarrow{\approx} {\bf GL}_2(\mathbb{Z}/\ell^n\mathbb{Z})$ induced by $\rho_\ell$, for any $n \geq 1$. 
Indeed, from \cite[p. IV-18]{serre1968}, the representation $\rho_\ell:\mathcal{G}\rightarrow {\bf GL}_2(\mathbb{Z}_\ell)$, obtained 
by Galois action on the Tate module $T_\ell$,  
composed with the determinant map yields the cyclotomic character $\psi_\ell$, whose image is $\mathbb{Z}_\ell^*$ 
(since the base field is $\mathbb{Q}$). Setting $X={\bf SL}_2(\mathbb{Z}_\ell) \cap \Image(\rho_\ell)$, one obtains a closed subgroup of 
${\bf SL}_2(\mathbb{Z}_\ell)$. 
Then, assuming that the image of $X$ into ${\bf SL}_2(\mathbb{F}_\ell)$ is equal to ${\bf SL}_2(\mathbb{F}_\ell)$, one concludes that 
$\Image(\rho_\ell)={\bf SL}_2(\mathbb{Z}_\ell)$ whenever $\ell\geq 5$ \cite[Lemma 3, p. IV-23]{serre1968}.  
Altogether, one has:
\begin{equation}
\label{eq:eq41final7}
\ell\geq 5 \textrm{ and } \Image(\widetilde \rho_\ell)={\bf GL}_2(\mathbb{F}_\ell) \Longrightarrow
\Image(\rho_\ell)={\bf GL}_2(\mathbb{Z}_\ell).
\end{equation}
See \cite[pp. 299--301]{serre1972}. 

We denote $\Sigma_E$ the set of primes at which $E$ has bad reduction. 
The set $\Sigma_E$ is finite \cite[Remark 1.3, p. 211]{silverman2009} and is non-empty \cite[Exerc. 8.15, p. 264]{silverman2009}.
In the case of curves without CM, Serre  proved that the set of primes $\ell$ at which $E$ has ordinary good reduction has density $1$ 
(c.f. \cite[Corollaire 1, p. 189]{serre1981}, using \cite[Exerc. 5.10, p. 154]{silverman2009}). 
See also \cite[Exerc. 5.11, p. 154]{silverman2009} for a weaker statement in the case of an arbitrary elliptic curve over $\mathbb{Q}$. 
On the other hand, Elkies proved that any elliptic curve over $\mathbb{Q}$ has infinitely many primes $\ell$ at which $E$ has supersingular good reduction \cite{elkies1987}. 

If $p$ is a prime of $\mathbb{Z}$, $\ord_p$ denotes the valuation on $\mathbb{Q}_p$ such that $\ord_p(p)=1$. 

%%%%%%%%%%%%%%%%%%%%%%%%%%%%%%%%%%%%%%%%%%%%%%%%%%%%
%%%%%%%%%%%%%%%%%%%%%%%%%%%%%%%%%%%%%%%%%%%%%%%%%%%%

\section{Liftings of points on reduced elliptic curves to points with coordinates in specific algebraic number fields}
\label{section:liftingsPoints}

We collect in this section results on torsion points of elliptic curves that will be useful in the sequel.

%%%%%%%%%%%%%%%%%%%%%%%%%%%%%%%%%%%%%%%%%%%%%%%%%%%%

\subsection{Torsion points over algebraic number fields}
\label{subsection:torsionPointsK}

First of all, the following result on torsion points of elliptic curves over algebraic number fields 
expresses \cite[Theorem 7.1, p. 240]{silverman2009} in a context relevant to this work. 
Equation (\ref{eq:eq42final7}) follows from a result of Cassels; see \cite[Theorem 3.4, p. 193]{silverman2009}. 
Part b) is due to Lutz and Nagell independently 
in the case where $K=\mathbb{Q}$; see \cite[Corollary 7.2, p. 240]{silverman2009}.

%%%%%%%%%%%%%%

\begin{lem}
\label{thm:Lemma3}
Let $E$ be an elliptic over an algebraic number field $K$, with Weierstrass equation of the form $y^2=x^3+Ax+B$, 
where $A,B\in \mathcal{O}_K$. 
Let $\ell$ be a prime number. 

a) Assume that $\ell >3$. Then, any non-trivial $\ell^n$-torsion point $P$ of $E$ over $\overline{\mathbb{Q}}$, where $n\geq 1$, satisfies 
the conditions:
\begin{eqnarray}
\label{eq:eq42final7}
&& x(P), y(P) \in (\ell)^{-1};\\
&&\bigl ( \ell y(P) \bigr )^2 \mid \Delta^\prime \ell^5,
\end{eqnarray} 
where $(x(P),y(P))$ are the affine coordinates of $P$, $\Delta^\prime=4 A^3+27 B^2$, so that $\Delta=-16 \Delta^\prime$ is the discriminant of the Weierstrass equation, 
and the divisibility condition holds in the integer ring $\mathcal{O}_n$ of $L_n:=K(E[\ell^n])$. 

b) If $\ell$ is unramified in $K$, any non-trivial $\ell^n$-torsion point $P$ of $E$ over $K$, where $n\geq 1$, satisfies the conditions:
\begin{eqnarray}
&& x(P), y(P) \in \mathcal{O}_K;\\
&&y(P)^2 \mid \Delta^\prime,
\end{eqnarray}
where the divisibility condition holds in the integer ring $\mathcal{O}_K$ of $K$. 
\end{lem}

\begin{proof} 
Part a). The proof follows closely \cite[pp. 240-241]{silverman2009}, but with some modifications. 

From \cite[Theorem 7.1, p. 240]{silverman2009}, $x(P)$ and $y(P)$ are ${v}$-integral for any place $v \nmid \ell$ of $L_n$. 
Moreover, from that result, if $v \mid \ell$, then one has $v(y(P))\geq -3 v(\ell)/(\ell-1) > - v(\ell)$ and $v(x(P))\geq -2 v(\ell)/(\ell-1) > - v(\ell)$, having assumed that $\ell \geq 5$. 
Therefore, $x(P),y(P)\in (\ell)^{-1}$.

Next, from \cite[Sublemma 4.3, p. 222]{silverman2009}, one deduces the identity:
\begin{equation}
f(x(P)) \phi(x(P)) - g(x(P)) \psi(x(P)) = \Delta^\prime,
\end{equation} 
where
\begin{equation}
\begin{cases}
f(X) = 3 X^2 + 4 A;\\
\phi(X) = X^4 - 2 A X^2 - 8 B X + A^2;\\
g(X) = 3 X^3 - 5 AX - 27 B;\\
\psi(X) = X^3 + AX + B;\\
\Delta^\prime = 4 A^3 + 27 B^2.
\end{cases}
\end{equation}
Note here that $-16 \Delta^\prime$ is the discriminant of the Weierstrass equation $y^2=x^3+Ax+B$ \cite[p. 45]{silverman2009}. 
But, one has the duplication identity, which can be deduced from \cite[p. 54]{silverman2009}, as $\ell\not=2$:
\begin{equation}
\label{eq:eq48final7}
x([2]P) = \frac{\phi(x(P))}{4\psi(x(P))}.
\end{equation}
It follows that:
\begin{equation}
\label{eq:eq49final7}
y(P)^2 \bigl [ 4 f(x(P)) x([2]P) - g(x(P)) \bigr ] = \Delta^\prime,
\end{equation} 
since $y(P)^2=\psi(x(P))$. 
Now, multiplying by $\ell^5$, one obtains:
 \begin{equation}
\label{eq:eq50final7}
\bigl ( \ell y(P) \bigr )^2 \ell^3 \bigl [ 4 f(x(P)) x([2]P) - g(x(P)) \bigr ] = \Delta^\prime \ell^5.
\end{equation} 
But, $\ell^3 \bigl [ 4 f(x(P)) x([2]P) - g(x(P)) \bigr ]$ is an integral element of $L_n$ because both 
$\ell x(P)$ and $\ell x([2]P)$ are integral elements. This proves part a). 

Part b). If $\ell$ is unramified in $K$, then $v(\ell)=1$. Thus, the above conditions 
$v(y(P))\geq -3 v(\ell)/(\ell-1) > - v(\ell)$ and $v(x(P))\geq -2 v(\ell)/(\ell-1) > - v(\ell)$ mean that 
$x(P),y(P)\in \mathcal{O}_K$. This yields:
 \begin{equation}
\bigl ( y(P) \bigr )^2 \bigl [ 4 f(x(P)) x([2]P) - g(x(P)) \bigr ] = \Delta^\prime,
\end{equation}
where $\bigl [ 4 f(x(P)) x([2]P) - g(x(P)) \bigr ] \in \mathcal{O}_K$. 
This proves part b).
\end{proof}

%%%%%%%%%%%%%%

%%%%%%%%%%%%%%

\begin{lem}
\label{thm:Lemma4}
Let $E$ be an elliptic curve over an algebraic number field $K$. 
Given a prime number $\ell$ and a positive integer $n$, let $L_n$ denote the field extension over $K$ obtained 
by adjoining the affine coordinates of all $\ell^n$-torsion points of $E$. 

a) Then, there is a group embedding:
\begin{equation}
0 \rightarrow \Gal (L_n/K) \rightarrow {\bf GL}_2 \left ( \mathbb{Z}/\ell^n\mathbb{Z}\right).
\end{equation}

b) The latter group has order dividing $\ell(\ell-1)^2(\ell+1)\ell^{4(n-1)}$. In particular, $[L_n:K]$ divides
$\ell(\ell-1)^2(\ell+1)\ell^{4(n-1)}$.

c) Let $\ell_i$, $i=1,...,\nu$, be distinct prime numbers, and let $n_i$ be positive integers, $i=1,...,\nu$. Set 
$m=\prod_{i=1}^\nu \ell_i^{n_i}$. 
Let $K(E[m])$ be the field obtained by adjoining over $K$ the affine coordinates of all 
$m$-torsion points of $E$ (so, $L_n$ means $K(E[\ell^n])$). Then, there is a group embedding:
\begin{equation}
0 \rightarrow \Gal (K(E[m])/K) \rightarrow \prod_{i=1}^{\nu} {\bf GL}_2 \left ( \mathbb{Z}/\ell_i^{n_i}\mathbb{Z}\right).
\end{equation}
In particular, $[K(E[m]):K]$ divides the integer:
\begin{equation}
\prod_{i=1}^{\nu} \ell_i(\ell_i-1)^2(\ell_i+1)\ell_i^{4(n_i-1)}.
\end{equation}
\end{lem}

\begin{proof}
Part a). Since $\Gal (L_n/K)$ acts faithfully on the finite group of $\ell^n$-torsion points $E[\ell^n]$, one has a group embedding:
\begin{equation}
0 \rightarrow \Gal (L_n/K) \rightarrow {\bf Aut} \left ( E[\ell^n] \right).
\end{equation}
Since $E[\ell^n]$ is isomorphic to $\mathbb{Z}/\ell^n\mathbb{Z} \oplus \mathbb{Z}/\ell^n\mathbb{Z}$ 
\cite[Corollary 6.4, p. 86]{silverman2009}, it follows that
\begin{equation}
{\bf Aut} \left ( E[\ell^n] \right) \xrightarrow{\approx} {\bf GL}_2 \left ( \mathbb{Z}/\ell^n\mathbb{Z}\right).
\end{equation}

Part b). One has an exact sequence of groups:
\begin{equation}
0 \rightarrow I + \ell {\bf Mat}_2 \left ( \mathbb{Z}/\ell^n\mathbb{Z}\right) \rightarrow   
{\bf GL}_2 \left ( \mathbb{Z}/\ell^n\mathbb{Z}\right) \xrightarrow{\pi_*} {\bf GL}_2 \left ( \mathbb{F}_\ell\right),
\end{equation}
where $I$ denotes the $2 \times 2$ identity matrix over $\mathbb{Z}/\ell^n\mathbb{Z}$, and the map $\pi_*$
is induced by the projection of rings $\pi: \mathbb{Z}/\ell^n\mathbb{Z} \rightarrow \mathbb{Z}/\ell\mathbb{Z}\approx \mathbb{F}_\ell$. 
But the rightmost factor ${\bf GL}_2 \left ( \mathbb{F}_\ell\right)$ has order $(\ell^2-1)(\ell^2-\ell)=\ell(\ell-1)^2(\ell+1)$, 
whereas the left most factor $I + \ell {\bf Mat}_2 \left ( \mathbb{Z}/\ell^n\mathbb{Z}\right)$ has order $\left ( \ell^{n-1} \right )^4$. 
Thus, the order of ${\bf GL}_2 \left ( \mathbb{Z}/\ell^n\mathbb{Z}\right)$ divides $\ell(\ell-1)^2(\ell+1)\ell^{4(n-1)}$. 
Lastly, part a) implies that $[L_n:K]$ divides the order of ${\bf GL}_2 \left ( \mathbb{Z}/\ell^n\mathbb{Z}\right)$.

Part c). Generalizing the proof of part a), one has a group embedding:
\begin{equation}
0 \rightarrow \Gal (K(E[m])/K) \rightarrow {\bf Aut} \left ( E[m] \right) \xrightarrow{\approx} 
{\bf GL}_2 \left ( \mathbb{Z}/m\mathbb{Z}\right),
\end{equation}
since $\Gal (K(E[m])/K)$ acts faithfully on $E[m]$, 
which is isomorphic to $\mathbb{Z}/m\mathbb{Z} \oplus \mathbb{Z}/m\mathbb{Z}$ 
\cite[Corollary 6.4, p. 86]{silverman2009}. But then, the isomorphism of rings 
$\mathbb{Z}/m\mathbb{Z} \xrightarrow{\approx} \prod_{i=1}^{\nu} \mathbb{Z}/\ell_i^{n_i}\mathbb{Z}$ 
(from the Chinese Remainder Theorem) yields an isomorphism:
\begin{equation}
{\bf GL}_2 \left ( \mathbb{Z}/m\mathbb{Z}\right) \xrightarrow{\approx} 
\prod_{i=1}^{\nu} {\bf GL}_2 \left ( \mathbb{Z}/\ell_i^{n_i}\mathbb{Z}\right).
\end{equation}
Now, use part b) on each factor of the right-hand side of this equation. 
\end{proof}

\noindent {\bf Remark 2.}
\label{remark2}
Recall from \cite[\S 4.2, pp. 151--152]{serre1981} that the Galois group of the infinite extension $L_\infty/\mathbb{Q}$ 
obtained by adjoining over $\mathbb{Q}$ the affine coordinates of all $\ell^n$-torsion points of $E$, where $n\geq 1$, is an $\ell$-adic Lie group.  
Indeed, one has an embedding $\rho_\ell:\Gal(L_\infty/\mathbb{Q}) \hookrightarrow {\bf GL}_2(\mathbb{Z}_\ell)$ obtained by Galois action on the Tate module $\varprojlim E[\ell^n]$. 
Its dimension $N$ is at least $2$ and at most $4$, 
since $E[\ell^n] \approx \mathbb{Z}/\ell^n \mathbb{Z} \oplus \mathbb{Z}/\ell^n$ for all $n\geq 1$. 
Then, Lemma \ref{thm:Lemma4} shows that $[L_n:\mathbb{Q}]$ is of the form $b \ell^{nN}$, for some integer 
$b$ dividing $(\ell-1)^2(\ell+1)\ell^{\beta}$, with $\beta\geq 0$. 
See Appendix \ref{section:appendixA} for an expression of the different of the extension 
$L_n/\mathbb{Q}$ based on a theorem of Sen \cite{sen1972} that was conjectured by Serre \cite{serre1967c}. 
\\ 

\subsection{Multiplication by positive integers in elliptic curves}
\label{subsection:Multiplication}

We consider an elliptic curve $E$ over a field $K$, with Weierstrass equation $y^2=x^3+Ax+B$. 

We consider multiplication by a positive integer $n$ in $E(K)$, where $K$ is a field.
For this purpose, we recall from \cite{lang1978,schoof1985} the polynomials over $\mathbb{Z}[A,B]$ 
(note that in \cite{schoof1985}, these polynomials are considered over a finite field):
\begin{equation}
\begin{cases}
\Psi_{-1}(X,Y) = -1;\\
\Psi_{0}(X,Y) = 0;\\
\Psi_{1}(X,Y) = 1;\\
\Psi_{2}(X,Y) = 2Y;\\
\Psi_{3}(X,Y) = 3X^4+6 AX^2+12 B X - A^2;\\
\Psi_{4}(X,Y) = 4Y(X^6 + 5AX^4 + 20BX^3 - 5A^2 X^2 - 4AB X - 8 B^2 -A^3).
\end{cases}
\end{equation}
Then, one has the recursion formulae for $n\geq 1$:
\begin{equation}
\begin{cases}
\Psi_{2n}(X,Y) = \Psi_n(X,Y) \left ( \Psi_{n+2}(X,Y)\Psi_{n-1}^2(X,Y) - \Psi_{n-2}(X,Y)\Psi_{n+1}^2(X,Y) \right)/2Y;\\
\Psi_{2n+1}(X,Y) = \Psi_{n+2}(X,Y)\Psi_n^3(X,Y) - \Psi_{n+1}^3(X,Y)\Psi_{n-1}(X,Y).                                                  
\end{cases}
\end{equation}

As in \cite{schoof1985}, we denote $\Psi^\prime_n(X,Y)$ the polynomial obtained from $\Psi_n(X,Y)$ by replacing $Y^2$ with $X^3+AX+B$. 
Then, it turns out that:
\begin{equation}
\begin{cases}
f_{n}(X) = \Psi^\prime_n(X,Y), \quad n \textrm { odd};\\
f_{n}(X) = \Psi^\prime_n(X,Y)/Y, \quad n \textrm { even},                                                  
\end{cases}
\end{equation}
are polynomials in $X$. From \cite[pp. 37--38]{lang1978}, $f_n(X) \in 2\mathbb{Z}[X]$ for $n$ even. 
Moreover, from \cite[p. 486]{schoof1985}, it follows that for any $n\geq 1$, one has:
\begin{equation} 
\label{eq:eq63final7}
\begin{cases}
f_{n}(X) = c X^{(n^2-1)/2} + \cdot \cdot \cdot , \quad n \textrm { odd};\\
f_{n}(X) = c X^{(n^2-4)/2} + \cdot \cdot \cdot , \quad n \textrm { even},                                                  
\end{cases}
\end{equation}
for some element $c\not=0$ in $K$. 

One can also show, with $P=(x,y)$, that:
\begin{equation}
[n] P = O \Longleftrightarrow f_n(x)=0,
\end{equation}
assuming that $[2]P\not = O$ \cite[Proposition (2.1)]{schoof1985}. One also has \cite[Chapter II]{lang1978}:
\begin{equation}
[n] P = O \Longleftrightarrow (\Psi^\prime_n(x,y))^2=0,
\end{equation}
for $P\not = O$.

Then, one has:
\begin{equation}
\label{eq:eq66final7}
x([n]P) = x  - \frac{\Psi^\prime_{n-1}(x,y)\Psi^\prime_{n+1}(x,y)}{\left (\Psi^\prime_{n}(x,y)\right)^2};\\
\end{equation}
and
\begin{equation}
\label{eq:eq67final7}
y([n]P) = \frac{\Psi^\prime_{n+2}(x,y)\left ( \Psi^\prime_{n-1}(x,y) \right)^2 - \Psi^\prime_{n-2}(x,y)\left (\Psi^\prime_{n+1}(x,y)\right)^2}{4y\left ( \Psi^\prime_{n}(x,y) \right)^3},                                                  
\end{equation}
where $P=(x,y)$, assuming that $[n]P\not = O$; see \cite[Theorem 2.1, p. 38]{lang1978} and \cite[Proposition (2.2)]{schoof1985}. 

The following result refines Eq. (\ref{eq:eq63final7}) ({\em i.e.}, \cite[p. 486]{schoof1985}) and \cite[Theorem 2.1, p. 38]{lang1978}.

\begin{lem}
\label{thm:Lemma5}
For any $n\geq 1$, one has:
\begin{equation}
\begin{cases}
f_{n}(X) = n X^{(n^2-1)/2} + c X^{(n^2-1)/2-2} + \cdot \cdot \cdot , \quad n \textrm { odd};\\
f_{n}(X) = n X^{(n^2-4)/2}  + c X^{(n^2-4)/2-2}+ \cdot \cdot \cdot , \quad n \textrm { even}.                                                  
\end{cases}
\end{equation}
for some $c\in \mathbb{Z}[A,B]$ depending on $n$.
\end{lem} 

\begin{proof}
The proof is by induction on $n\geq 1$. 

The result is obviously true for $n=1,2$, since $f_1(X)=1$ and $f_2(X)$ is equal to 
$2$. It is also true for $n=3$, since $f_3(X)=3X^4+6 AX^2+12 B X - A^2$. For $n=4$, one has: 
\begin{equation}
f_4(X) = 4 ( X^6 + 5AX^4 + 20BX^3 - 5A^2 X^2 - 4AB X - 8 B^2 -A^3),
\end{equation}
so that the result is true.

Assume by induction hypothesis that the result is true for any $1\leq n^\prime<n$, for some integer $n\geq 5$. 

{\em Case 1}: $n=2m$, with $m\geq 4$ even. Then, one computes:
\begin{eqnarray}
f_n(X) &=& \Psi^\prime_{2m}(X,Y)/Y\nonumber\\
&=&  \Psi^\prime_m(X,Y) \left ( \Psi^\prime_{m+2}(X,Y)(\Psi^\prime_{m-1}(X,Y))^2 - \Psi^\prime_{m-2}(X,Y)(\Psi^\prime_{m+1}(X,Y))^2 \right)/2Y^2\nonumber\\
&=& f_m(X)Y \left ( f_{m+2}(X)Yf_{m-1}^2(X) - f_{m-2}(X)Yf_{m+1}^2(X) \right)/2Y^2\nonumber\\
&=& f_m(X)\left ( f_{m+2}(X)f_{m-1}^2(X) - f_{m-2}(X)f_{m+1}^2(X) \right)/2.
\end{eqnarray}
We have:
\begin{equation}
\begin{cases}
\deg (f_{m+2}(X)f_{m-1}^2(X)) = \frac{1}{2}((m+2)^2-4) + ((m-1)^2-1) = \frac{3}{2} m^2;\\
\lc (f_{m+2}(X)f_{m-1}^2(X)) = (m+2)(m-1)^2 = m^3 - 3m +2,
\end{cases} 
\end{equation}
and
\begin{equation}
\begin{cases}
\deg (f_{m-2}(X)f_{m+1}^2(X)) = \frac{1}{2}((m-2)^2-4) + ((m+1)^2-1) = \frac{3}{2} m^2;\\
\lc (f_{m-2}(X)f_{m+1}^2(X)) = (m-2)(m+1)^2 = m^3 - 3m - 2,
\end{cases} 
\end{equation}
where $\lc(f(X))$ denotes here the leading coefficient of polynomial $f(X)$. We also let $\lc_{-}(f(X))$ denote the next coefficient. 
So, if $\deg(f(X))=d$, one has $f(X)=\lc(f(X))X^d + \lc_{-}(f(X)) X^{d-1}+ \cdot \cdot \cdot$.
Thus, one has $\lc(f_{n}(X))=2m=n$, $\lc_{-}(f_n(X))=0$, and $\deg(f_n(X))=(n^2-4)/2$. 

{\em Case 2}: $n=2m$, with $m\geq 3$ odd. Then, one computes:
\begin{eqnarray}
f_n(X) &=& \Psi^\prime_{2m}(X,Y)/Y\nonumber\\
&=&  \Psi^\prime_m(X,Y) \left ( \Psi^\prime_{m+2}(X,Y)(\Psi^\prime_{m-1}(X,Y))^2 - \Psi^\prime_{m-2}(X,Y)(\Psi^\prime_{m+1}(X,Y))^2 \right)/2Y^2\nonumber\\
&=& f_m(X) \left ( f_{m+2}(X)f_{m-1}^2(X)Y^2 - f_{m-2}(X)f_{m+1}^2(X)Y^2 \right)/2Y^2\nonumber\\
&=& f_m(X)\left ( f_{m+2}(X)f_{m-1}^2(X) - f_{m-2}(X)f_{m+1}^2(X) \right)/2.
\end{eqnarray}
We have:
\begin{equation}
\deg (f_{m+2}(X)f_{m-1}^2(X)) = \frac{1}{2}((m+2)^2-1) + ((m-1)^2-4) = \frac{3}{2} (m^2 - 1),
\end{equation}
and
\begin{equation}
\deg (f_{m-2}(X)f_{m+1}^2(X)) = \frac{1}{2}((m-2)^2-1) + ((m+1)^2-4) = \frac{3}{2} (m^2-1),
\end{equation}
with leading coefficients as in Case 1.
Thus, $\lc(f_{n}(X))=2m=n$, $\lc_{-}(f_n(X))=0$, and $\deg(f_n(X))=(n^2-4)/2$. 

{\em Case 3}: $n=2m+1$, with $m\geq 2$ even. Then, one computes:
\begin{eqnarray}
f_n(X) &=& \Psi^\prime_{2m+1}(X,Y)\nonumber\\
&=&  \Psi_{m+2}^\prime(X,Y)(\Psi^\prime_m(X,Y))^3 - (\Psi^\prime_{m+1}(X,Y))^3\Psi^\prime_{m-1}(X,Y)\nonumber\\
&=& f_{m+2}(X)Y(f_m(X)Y)^3 - f_{m+1}^3(X)f_{m-1}(X)\nonumber\\
&=& f_{m+2}(X)f^3_m(X)(X^3+AX+B)^2 - f_{m+1}^3(X)f_{m-1}(X).
\end{eqnarray}
We have:
\begin{equation}
\begin{cases}
\deg (f_{m+2}(X)f^3_m(X)(X^3+AX+B)^2) = \frac{1}{2}((m+2)^2-4) + \frac{3}{2}(m^2-4) + 6\\
\qquad \qquad \qquad \qquad \qquad \qquad \qquad \qquad = 2m^2+2m;\\
\lc (f_{m+2}(X)f^3_m(X)(X^3+AX+B)^2) = (m+2)m^3 = m^4+2m^3,
\end{cases} 
\end{equation}
and
\begin{equation}
\begin{cases}
\deg (f_{m+1}^3(X)f_{m-1}(X)) = \frac{3}{2}((m+1)^2-1) + \frac{1}{2}((m-1)^2-1) = 2m^2+2m;\\
\lc (f_{m+1}^3(X)f_{m-1}(X)) = (m+1)^3(m-1) = m^4 + 2m^3 -2m - 1.
\end{cases} 
\end{equation}
So, one has $\lc(f_{n}(X))=2m+1=n$, $\lc_{-}(f_n(X))=0$, and $\deg(f_n(X))=(n^2-1)/2$. 

{\em Case 4}: $n=2m+1$, with $m\geq 3$ odd. Then, one computes:
\begin{eqnarray}
f_n(X) &=& \Psi^\prime_{2m+1}(X,Y)\nonumber\\
&=&  \Psi_{m+2}^\prime(X,Y)(\Psi^\prime_m(X,Y))^3 - (\Psi^\prime_{m+1}(X,Y))^3\Psi^\prime_{m-1}(X,Y)\nonumber\\
&=& f_{m+2}(X)f^3_m(X) - (f_{m+1}(X) Y)^3(X)f_{m-1}(X)Y\nonumber\\
&=& f_{m+2}(X)f^3_m(X) - f_{m+1}^3(X)f_{m-1}(X)(X^3+AX+B)^2.
\end{eqnarray}
We have:
\begin{equation}
\deg (f_{m+2}(X)f^3_m(X) = \frac{1}{2}((m+2)^2-1) + \frac{3}{2}(m^2-1) = 2m^2+2m,
\end{equation}
and
\begin{eqnarray}
&&\deg (f_{m+1}^3(X)f_{m-1}(X)(X^3+AX+B)^2)\nonumber\\
 && = \frac{3}{2}((m+1)^2-4) + \frac{1}{2}((m-1)^2-4) + 6 = 2m^2+2m,
\end{eqnarray}
with same leading coefficients as in Case 3. 
So, one has $\lc(f_{n}(X))=2m+1=n$, $\lc_{-}(f_n(X))=0$, and $\deg(f_n(X))=(n^2-1)/2$. 
\end{proof}

%%%%%%%%%%%%%%%%%%%

We obtain the following refinement of \cite[Theorem 2.1, ii, p. 38]{lang1978}.

\begin{cor}
\label{thm:Corollary3}
Let $E$ be an elliptic curve with Weierstrass equation of the form $y^2=x^3+Ax+B$. 
For any $n\geq 1$, one has in $\mathbb{Z}[X,A,B]$:
\begin{equation}
(\Psi^\prime_{n}(X,Y))^2 = n^2 X^{n^2-1} + c X^{n^2-3} + \cdot \cdot \cdot.                                                 
\end{equation}
for some $c\in \mathbb{Z}[A,B]$ depending on $n$.
\end{cor} 

\begin{proof}
For $n$ odd, one has directly from Lemma \ref{thm:Lemma5}:
\begin{equation}
(\Psi^\prime_{n}(X,Y))^2 = f_n^2(X) = n^2 X^{n^2-1} + c X^{n^2-3} + \cdot \cdot \cdot,                                                 
\end{equation}
For $n$ even, one has:
\begin{eqnarray}
(\Psi^\prime_{n}(X,Y))^2 &=& f_n^2(X) Y^2 = (n^2 X^{n^2-4} + c^\prime X^{n^2-6} + \cdot \cdot \cdot)(X^3+AX+B)\nonumber\\
&=& n^2 X^{n^2-1} + (c^\prime+n^2A) X^{n^2-3} + \cdot \cdot \cdot,                                                 
\end{eqnarray}
as was to be shown, taking $c=c^\prime + n^2 A$.
\end{proof}

\begin{cor}
\label{thm:Corollary4}
For any prime $\ell>2$ and integer $n >1$, one has in $\mathbb{Z}[X,A,B]$:
\begin{equation}
\frac{(\Psi^\prime_{\ell^n}(X,Y))^2}{(\Psi^\prime_{\ell^{n-1}}(X,Y))^2} = \ell^2 X^{\ell^{2n-2}(\ell^2-1)} + c X^{\ell^{2n-2}(\ell^2-1)-2} + \cdot \cdot \cdot,                                                 
\end{equation}
for some $c\in \mathbb{Z}[A,B]$ depending on $\ell^n$. 
\end{cor} 

\begin{proof}
Any $\ell^{n-1}$-torsion point is also an $\ell^{n}$-torsion point. Therefore, 
the polynomial $(\Psi^\prime_{\ell^{n-1}}(X,Y))^2$ divides $(\Psi^\prime_{\ell^{n}}(X,Y))^2$. 
From \cite[Theorem 2.2-iii, p. 39]{lang1978}, the quotient of these two polynomials is actually in $\mathbb{Z}[X,A,B]$. 
One then computes directly from Corollary \ref{thm:Corollary3}:
\begin{eqnarray}
\frac{(\Psi^\prime_{\ell^n}(X,Y))^2}{(\Psi^\prime_{\ell^{n-1}}(X,Y))^2} &=& \frac{\ell^{2n} X^{\ell^{2n}-1} + c_1 X^{\ell^{2n}-3} + \cdot \cdot \cdot}{\ell^{2n-2} X^{\ell^{2n-2}-1} + c_2 X^{\ell^{2n-2}-3} + \cdot \cdot \cdot}\nonumber\\
&=& \ell^2 X^{\ell^{2n-2}(\ell^2-1)} + c X^{\ell^{2n-2}(\ell^2-1)-2}  +  \cdot \cdot \cdot,         
\end{eqnarray}
as was to be shown.
\end{proof} 

\begin{cor}
\label{thm:Corollary5}
For any $m\geq 1$ and any element $\lambda$ in a field $K$ containing $A$ and $B$, one has in $K[X]$:
\begin{eqnarray}
\Phi_{m}(X,\lambda)&:=&(X-\lambda)(\Psi^\prime_{m}(X,Y))^2 - \Psi^\prime_{m-1}(X,Y)\Psi^\prime_{m+1}(X,Y)\nonumber\\
& = & X^{m^2} - \lambda m^2 X^{m^2-1} + \cdot \cdot \cdot.                                                 
\end{eqnarray}
\end{cor} 

\begin{proof}
For $m$ odd, one has:
\begin{eqnarray}
&&(X-\lambda)(\Psi^\prime_{m}(X,Y))^2 - \Psi^\prime_{m-1}(X,Y)\Psi^\prime_{m+1}(X,Y)\nonumber\\
&& = (X-\lambda)f_{m}^2(X) - f_{m-1}(X)Yf_{m+1}(X)Y\nonumber\\
 && = (X-\lambda)f_{m}^2(X) - f_{m-1}(X)f_{m+1}(X)(X^3+AX+B),                                                
\end{eqnarray}
which shows the result using Lemma \ref{thm:Lemma5}.
For $m$ even, one has:
\begin{eqnarray}
&&(X-\lambda)(\Psi^\prime_{m}(X,Y))^2 - \Psi^\prime_{m-1}(X,Y)\Psi^\prime_{m+1}(X,Y)\nonumber\\
&& = (X-\lambda)f_{m}^2(X)Y^2 - f_{m-1}(X)f_{m+1}(X)\nonumber\\
 && = (X-\lambda)f_{m}^2(X)(X^3+AX+B) - f_{m-1}(X)f_{m+1}(X),                                                
\end{eqnarray}
which implies the result in that case.
\end{proof}

\subsection{Torsion points over $\mathbb{Q}$ in the non-CM case}
\label{subsection:torsionPointsQ}

In this section, we consider non-CM elliptic curves $E$ over $\mathbb{Q}$, with Weierstrass equation $y^2=x^3+Ax+B$, 
where $A,B\in \mathbb{Z}$. 

Under these assumptions, the next results can be applied to almost all primes $\ell$, based on results of Serre. 
Namely, from \cite[Proposition, IV-19]{serre1968} and \cite[Th\'eor\`eme 2, p. 294]{serre1972}, 
one has for almost all prime numbers $\ell$, the isomorphism $\rho_\ell:\Gal(L_\infty/\mathbb{Q})\xrightarrow{\approx} {\bf GL}_2(\mathbb{Z}_\ell)$. 
The condition $\Gal(L_\infty/\mathbb{Q})\approx {\bf GL}_2(\mathbb{Z}_\ell)$, for a given $\ell$, is clearly equivalent to the condition $\Gal(L_n/\mathbb{Q})\approx {\bf GL}_2(\mathbb{Z}/\ell^n\mathbb{Z})$ for any $n\geq 1$. 
Moreover, the latter condition for a given $n>1$ implies the condition for all $1\leq n^\prime < n$. 

\begin{cor}
\label{thm:Corollary6}
Let $E$ be an elliptic curve over $\mathbb{Q}$ without CM. Let $y^2=x^3+Ax+B$ be a Weierstrass equation for $E$, with $A,B\in\mathbb{Z}$. 
Given a prime number $\ell > 2$ and an integer $n\geq 1$, set $L_n=\mathbb{Q}(E[\ell^n])$ and $\tr_n=\tr_{L_n/\mathbb{Q}}$, 
 the trace map of $L_n$ over $\mathbb{Q}$. 
Assume that $\rho_\ell$ induces an isomorphism $\Gal(L_n/\mathbb{Q})\xrightarrow{\approx} {\bf GL}_2(\mathbb{Z}/\ell^n\mathbb{Z})$.  

Then, one has:
\begin{equation}
\tr_{n}(x_n)=0,                                                 
\end{equation}
for any primitive $\ell^n$-torsion point $P_n$ of $E$, $x_n$ denoting its $x$-coordinate. 
\end{cor} 

\begin{proof}
From the assumption on $\ell$, all primitive $\ell^n$-torsion points of $E$ are conjugates. But there are $\ell^{2n} - \ell^{2(n-1)}$ 
of them, which yields $d=(\ell^{2n} - \ell^{2(n-1)})/2$ distinct $x$-coordinates. 
Now, the polynomial appearing in Corollary \ref{thm:Corollary4} is of the form $\left ( g_{\ell^n}(X) \right)^2$, 
where $g_{\ell^n}(X):=f_{\ell^n}(X)/f_{\ell^{n-1}}(X)$ has degree $d$. 
Thus, the polynomial $g_{\ell^n}(X)$ is the irreducible polynomial of $x_n$. 
As the coefficient of $X^{d-1}$ in this polynomial is equal to $0$ (using Corollary \ref{thm:Corollary4}), 
it follows that $\tr_{K/\mathbb{Q}} (x_n) =0$, where $K$ is the splitting field of $x_n$. 
Then, one computes: 
$\tr_n(x_n) = [L_n:K] \tr_{K/\mathbb{Q}} (x_n) =0$.
\end{proof}

\begin{cor}
\label{thm:Corollary7}
Let $E$ be an elliptic curve over $\mathbb{Q}$ without CM. 
Let $y^2=x^3+Ax+B$ be a Weierstrass equation for $E$, with $A,B\in\mathbb{Z}$.  
Given a prime number $\ell>2$ and an integer $n> 1$, set $L_n=\mathbb{Q}(E[\ell^n])$ and $\tr_{n,n-1}=\tr_{L_n/L_{n-1}}$, 
the trace map of $L_n$ over $L_{n-1}$. 
Assume that $\rho_\ell$ induces an isomorphism $\Gal(L_n/\mathbb{Q})\xrightarrow{\approx} {\bf GL}_2(\mathbb{Z}/\ell^n\mathbb{Z})$.  
 
Then, one has $[L_n:L_{n-1}]=\ell^4$, and the following identity holds:
\begin{equation}
\frac{\tr_{n,n-1}(x_n)}{[L_n:L_{n-1}]}=x_{n-1},                                                 
\end{equation}
for any primitive $\ell^n$-torsion point $P_n$ of $E$, $x_n$ and $x_{n-1}$ denoting the $x$-coordinate of 
$P_n$ and $[\ell]P_n$, respectively. In particular, $x_n - x_{n-1} \in \Ker \tr_{n,n-1}$.                                                 
\end{cor} 

\begin{proof} Given $n>1$, let $P_n$ be a primitive $\ell^n$-torsion point of $E$, and set $P_{n-1}:=[\ell]P_n$ and 
$x_{n-1}=x(P_{n-1})$. 
The roots of the polynomial $\Phi_{\ell}(X,x_{n-1})$ appearing in Corollary 
\ref{thm:Corollary5}, where we take $\lambda=x_{n-1}\in L_{n-1}$ and $m=\ell$, are the $x$-coordinates of 
the solutions $P$ to the equation $[\ell]P = \pm P_{n-1}$. Since $P$ and $-P$ have the same $x$-coordinate, we may restrict 
to the solutions of $[\ell]P=P_{n-1}$.   
There are $\ell^2$ solutions to this equation; namely $P=P_n+Q$, where $Q$ is an $\ell$-torsion point. 
This yields $\ell^2$ distinct $x$-coordinates, since $P_n+Q=\pm (P_n+Q^\prime)$ yields $Q=Q^\prime$, or else $[2]P_n=-Q-Q^\prime$, 
which is excluded since $\ell\not=2$ and $n>1$.  
Now, $\ell^2$ is the degree of the polynomial $\Phi_{\ell}(X,x_{n-1})$. 
From the assumption on the Galois group, it follows that the solutions to $[\ell]P=P_{n-1}$ are conjugates, 
and hence that the roots of $\Phi_{\ell}(X,x_{n-1})$ are conjugates. 
Therefore, $\Phi_{\ell}(X,x_{n-1})$ is the minimal polynomial of $x_n$ over $L_{n-1}$.  
We conclude that $\tr_{K/L_{n-1}}(x_n)=\ell^2 x_{n-1}$, where $K$ is the splitting field of $x_n$ over $L_{n-1}$. 
Thus, one obtains $\tr_{n,n-1}(x_n)=[L_n:K] \ell^2 x_{n-1}$. But then, $[L_n:L_{n-1}]=\ell^4$ from the assumption on the Galois group, 
whereas $[K:L_{n-1}]=\ell^2$ from above. Thus, $[L_n:K]=\ell^2$ and the result is shown.
\end{proof}

%%%%%%%%%%%%%%%%%%%

In the next result, we consider a place $v$ of $\mathbb{Q}$, either $p$-adic or Archimedean, and pick the canonical 
norm $\vert \cdot \vert_v$ such that $\vert p \vert_v = p^{-1}$ if $v$ is the non-Archimedean place associated to a prime $p$, 
or $\vert \cdot \vert_v$ is the absolute value of a real number, if $v$ is the Archimedean place. See \cite[pp. 34--35]{lang1986}. 
One then has the product formula \cite[p. 99]{lang1986}:
\begin{equation}
\prod_{v} \vert x \vert_v = 1,                                                 
\end{equation}
for any rational number $x$, where the product covers all canonical places of $\mathbb{Q}$. 

If $K$ is an algebraic field, we consider for each place $w$ lying above a place $v$ of $\mathbb{Q}$, the unique norm that extends 
$\vert \cdot \vert_v$. Namely, if $w$ is non-Archimedean, one defines:
\begin{equation}
\vert x \vert_w = p^{- w(x)/e_w},                                              
\end{equation}
where $w$ is viewed as the discrete valuation associated to the place $w$, $p$ is the prime number lying below $w$,  
and $e_w$ denotes the absolute ramification index of $p$ in $K_w$. If $w$ is Archimedean, $\vert \cdot \vert_w$ denotes the absolute 
value if $K_w=\mathbb{R}$, or the complex modulus if $K_w=\mathbb{C}$. See \cite[p. 35 and p. 99]{lang1986} for the alternative norm 
$\Vert x \Vert_w=\vert x \vert_w^{[K_w:\mathbb{Q}_v]}$ and the corresponding product formula, which is not used here.

We now introduce a topology on $E(\mathbb{C})$ as follows. 
Adapting \cite[Exerc. 7.6, pp. 203--204]{silverman2009}, we consider the Euclidean topology on $\mathbb{C}$ 
defined by the complex modulus $\vert \cdot \vert$. 
Then, we consider the product topology on the affine space $\mathbb{A}^2(\mathbb{C})$. 
Next, for each $0\leq i \leq 2$, there is an inclusion 
$\phi_i: \mathbb{A}^2(\mathbb{C}) \rightarrow \mathbb{P}^2(\mathbb{C})$ \cite[p. 9]{silverman2009}. 
This allows gluing together the product spaces $\phi_i(\mathbb{A}^2(\mathbb{C}))$, $i=0,1,2$. 
In this manner, we obtain a topology naturally defined on $\mathbb{P}^2(\mathbb{C})$, and hence on $E(\mathbb{C}) \hookrightarrow \mathbb{P}^2(\mathbb{C})$ 
based on a homogeneous equation $y^2z=x^3+Axz^2+Bz^3$ for $E$. 
Note that $E(\mathbb{C})$ is a Hausdorff space, so that a sequence of points in this topological space has at most one limit.

In the following result, 
we use the fact that the polynomial $X^3+AX+B$ has three distinct roots, 
since an elliptic curve is non-singular; equivalently, since its discriminant $\Delta$ does not vanish.

\begin{prop}
\label{thm:Proposition2}
Let $E$ be an elliptic curve over $\mathbb{Q}$, having Weierstrass equation $y^2=x^3+Ax+B$, with $A,B\in \mathbb{Z}$. 
Given a prime number $\ell$, set $L_n=\mathbb{Q}(E[\ell^n])$ and $\tr_{n}=\tr_{L_n/\mathbb{Q}}$. 

a) Then, for all prime numbers $\ell >3$ not dividing $\Delta^\prime$, and for any integer $n \geq 1$, one has:
\begin{equation}
\label{eq:eq94final7}
\Bigl \vert \frac{\tr_{n}(\alpha_n)}{[L_n:\mathbb{Q}]} \Bigr \vert_{v} \leq C_{*,v},                                                 
\end{equation}
for any primitive $\ell^n$-torsion point $P_n$ of $E$, where $\alpha_n=\ell^3 \Delta^\prime/y^2(P_n)$, 
and any place $v\not=\ell$ of $\mathbb{Q}$, 
for some constant $C_{*,v}>0$ depending only on $A$ and $B$, the prime $\ell$ and the place $v$.
Namely, one has explicitly:
\begin{eqnarray}
&&i):\quad C_{*,q} = \vert (\ell-1)^{-2} (\ell+1)^{-1} \vert_{q}, \textrm{ if } q\not=\ell;\\
&&ii):\quad C_{*,\infty} = \vert \Delta^\prime \vert \ell^3 \max \left ( 2, 1/\delta^{3} \right ),                                               
\end{eqnarray}
where the constant $\delta>0$ is the minimal distance between $x$-coordinates $x_n$ and $e_1$ of any primitive $\ell^n$- and $2$-torsion points,
respectively, of $E(\mathbb{C})$, such that $\vert x_n \vert < \sqrt{2(\vert A \vert + \vert B \vert)}$.  

b) Let $E$ be a non-CM curve, and assume that $\rho_\ell$ induces an isomorphism $\Gal(\mathbb{Q}(E[\ell^n])/\mathbb{Q})\xrightarrow{\approx} {\bf GL}_2(\mathbb{Z}/\ell^n\mathbb{Z})$ for any $n\geq 1$, where $\ell>3$ is a prime number.  
Then, ${\tr_{n}(\alpha_n)}/{[L_n:\mathbb{Q}]}$ has only finitely many values for all $n\geq 1$ 
($\ell$ being fixed). Moreover, (\ref{eq:eq94final7}) then also holds for the $\ell$-adic norm, taking:
\begin{equation}
iii):\quad C_{*,\ell} = \max_{n=1,2;\, \alpha_n} \Bigl \vert \frac{\tr_n(\alpha_n)}{[L_n:\mathbb{Q}]} \Bigr \vert_{\ell} < \infty.\\                                           
\end{equation}
Similar estimates also hold for the other norms. 
\end{prop}  

\begin{proof} {\em Step 1.}
Let us fix the prime $\ell > 3$, the integer $n\geq 1$, and the primitive $\ell^n$-torsion point $P_n$ with $x$ and $y$-coordinates 
$x_n=x(P_n)$ and $y_n=y(P_n)$, respectively. We set $\alpha_n=\ell^3 \Delta^\prime/y_n^2$. We assume that $\ell \nmid \Delta^\prime$.

{\em Step 2.} We consider first the case where $v$ is a non-Archimedean place of $\mathbb{Q}$ corresponding to a prime number $q\not = \ell$.
Since $\alpha_n$ is a divisor of $\ell^5 \Delta^\prime$, it follows that $\tr_n (\alpha_n )$ is an integer. Therefore, one obtains:
\begin{equation}
\Bigl \vert \frac{\tr_{n}(\alpha_n)}{[L_n:\mathbb{Q}]} \Bigr \vert_{q} \leq \vert (\ell-1)^{-2} (\ell+1)^{-1} \vert_{q},
\end{equation}
using Lemma \ref{thm:Lemma4}, which proves inequality i).

{\em Step 3.} Next, fix a place $w$ of $L_n$ lying above the Archimedean place $v=\infty$ of $\mathbb{Q}$; 
equivalently, $w$ corresponds to an embedding $\xi:\overline{\mathbb{Q}} \hookrightarrow \mathbb{C}$, with $\vert x \vert_{w} = \vert \xi(x) \vert$. 
We set $C_1:=\max \left ( 1, \sqrt{2(\vert A \vert_w + \vert B \vert_w )}\right ) = \sqrt{2(\vert A \vert + \vert B \vert)}$. 
Then, one has for $\vert x_n \vert_w \geq 1$:
\begin{equation}
\vert x_n^3+Ax_n+B \vert  \geq  \vert x \vert_w^3 - \vert A \vert_w \vert x_n \vert_w - \vert B \vert_w 
\geq \vert x_n \vert_w^3 - \left ( \vert A \vert_w + \vert B \vert_w \right ) \vert x_n \vert_w.
\end{equation}
Assuming also that $\vert x_n \vert_w \geq \sqrt{2(\vert A \vert_w + \vert B \vert_w )}$, 
one obtains $\vert x_n \vert_w^2/2 \geq (\vert A \vert_w + \vert B \vert_w )$, which yields:
\begin{equation}
\vert x_n \vert_w^3 - \left ( \vert A \vert_w + \vert B \vert_w \right ) \vert x_n \vert_w \geq \vert x_n \vert_w^3/2.
\end{equation}
Altogether, the lower bound $\vert x_n \vert_w \geq C_1$ implies the inequality:
\begin{equation}
\Bigl \vert \frac{\Delta^\prime \ell^3}{\left ( x_n^3+Ax_n+B \right )} \Bigr \vert_w 
\leq \frac{\vert \Delta^\prime \vert \ell^3}{\vert x_n \vert_w^3/2} \leq 2 \vert \Delta^\prime \vert \ell^3.
\end{equation}

{\em Step 4.} We are left with the case where the norm $\vert x_n \vert_w$ corresponding to an Archimedean place $w$ of $L_n$, is within the bound $C_1$. 
Let $S$ denote the set of elements $x_n$ such that $\vert x_n \vert_w < C_1$. 
We then need a lower bound for $\vert x_n^3+Ax_n+B \vert_w$, whenever $x_n \in S$.

{\em Step 5.} 
Assume that the set $S$ of step 4 is finite. 

Then, if any element $x_n$ of $S$ is equal to a root of $X^3+AX+B$, 
the point $(x_n,0)$ is a $2$-torsion point of $E(\overline{\mathbb{Q}})$, 
which contradicts the assumption that $x_n$ is the $x$-coordinate of a primitive $\ell^n$-torsion point, as $\ell \not =2$. Therefore, these elements $x_n$ are away from the roots of $X^3+AX+B$, say by a distance $\delta>0$, which implies that $1/(x_n^3+Ax_n+B)$ remains bounded on $S$. 
Therefore, one obtains:
\begin{equation}
\Bigl \vert \frac{\Delta^\prime \ell^3}{\left ( x_n^3+Ax_n+B \right )} \Bigr \vert_w \leq \vert \Delta^\prime \vert \ell^3/\delta^{3},
\end{equation}
for all elements $x_n \in S$. 
In that case, one concludes that:
\begin{equation}
\label{eq:eq103final7}
\Bigl \vert \frac{\tr_n(\alpha_n)}{[L_n:\mathbb{Q}]} \Bigr \vert \leq \vert \Delta^\prime \vert \ell^3 \max \left ( 2, 1/\delta^{3} \right ),
\end{equation}
making use of step 3.

{\em Step 6.} Next, we consider the case where the set $S$ of step 4 is infinite. 
At this point, we make use of the topology defined above on $E(\mathbb{C}) \hookrightarrow \mathbb{P}^2(\mathbb{C})$, and we identify 
$E(\overline{\mathbb{Q}})$ with its image under the embedding $\xi: E(\overline{\mathbb{Q}}) \hookrightarrow E(\mathbb{C})$.

As $S$ is bounded, there exists an accumulation point $x_{**} \in \mathbb{C}$ of $S$ such that 
$\vert x_{**} \vert \leq C_1$. Thus, $\lim_{k\rightarrow \infty} \xi(x_{n_k}) = x_{**}$ for some sequence $\{x_{n_k}\}$. 

If the point $x_{**}$ is of the form $\xi(x_*)$, where $x_*:=e_1$ is a root of $X^3+AX+B=(X-e_1)(X-e_2)(X-e_3)$, then $P_*=(x_*,0)$ is a $2$-torsion point of 
$E(\overline{\mathbb{Q}})$. Since $\ell \not = 2$, one has $[\ell]P_*=P_*$. 
But, from (\ref{eq:eq66final7}), one has: 
\begin{equation}
x([\ell]P_*) = x_* - \frac{\Psi_{\ell-1}^\prime (x_*,0) \Psi_{\ell+1}^\prime (x_*,0)}{\left ( \Psi_\ell^\prime (x_*,0)\right)^2},
\end{equation}
knowing that the denominator 
$\left ( \Psi_\ell^\prime (x_*,0)\right)^2$ does not vanish, since $\ell \not = 2$.
Thus, one concludes that $\Psi_{\ell-1}^\prime (x_*,0) \Psi_{\ell+1}^\prime (x_*,0) = 0$. 

We also have:
\begin{equation}
x([\ell]P_n) = x_n - \frac{\Psi_{\ell-1}^\prime (x_n,y_n) \Psi_{\ell+1}^\prime (x_n,y_n)}{\left ( \Psi_\ell^\prime (x_n,y_n)\right)^2},
\end{equation}
assuming that $n>1$. 
This yields:
\begin{eqnarray}
&&\vert \Psi_{\ell-1}^\prime (x_n,y_n) \Psi_{\ell+1}^\prime (x_n,y_n) \vert_w \nonumber\\
&=& \vert \Psi_{\ell-1}^\prime (x_n,y_n) \Psi_{\ell+1}^\prime (x_n,y_n) - \Psi_{\ell-1}^\prime (x_*,0) \Psi_{\ell+1}^\prime (x_*,0) \vert_w
\nonumber\\
&\leq& C \vert x_n - x_* \vert_w,
\end{eqnarray}
for some constant $C>0$, since $\vert x_n \vert_w \leq C_1$.
We thus conclude that:
\begin{equation}
\vert x_{n_k} - x([\ell]P_{n_k}) \vert_w = \Bigl \vert \frac{\Psi_{\ell-1}^\prime (x_n,y_n) \Psi_{\ell+1}^\prime (x_n,y_n)}{\left ( \Psi_\ell^\prime (x_n,y_n)\right)^2} \Bigr \vert_w \leq C^\prime \vert x_{n_k} - x_* \vert_w,
\end{equation}
for some constant $C^\prime>0$.
In particular, one has $\lim_{k\rightarrow \infty} x([1 - \ell]P_{n_k}) = 0$. Therefore, for some infinite sequence $\{ P_{n_k}\}$ of 
primitive $\ell^{n_k}$-torsion points, one has both $\lim_{k\rightarrow \infty} P_{n_k} = P_*$, and 
$\lim_{k\rightarrow \infty} [1 - \ell]P_{n_k} = (0,\sqrt{B},1)$ or $\lim_{k\rightarrow \infty} [1 - \ell]P_{n_k} = (0,-\sqrt{B},1)$.

We now show that $\lim_{k\rightarrow \infty} [1 - \ell]P_{n_k} = O$, based on the assumption that $\lim_{k\rightarrow \infty} P_{n_k}$ 
is the $2$-torsion point $P_*$, which will yield a contradiction.
For this purpose, we compute the $(x,y)$-coordinates of $Q_n := P_n [-] P_* = P_* [+] P_n$ \cite[pp. 53--54]{silverman2009}. 
Since $P_n \not = P_*$ and $a_1=a_2=a_3=0$, one has:
\begin{equation}
\begin{cases}
P_* = (e_1,0); \quad P_{n_k} = (x_{n_k},y_{n_k});\\
\lambda_{n_k} = \frac{y_{n_k}}{(x_{n_k}-e_1)};\\
\nu_{n_k} = \frac{-y_{n_k} e_1}{(x_{n_k}-e_1)} = -\lambda_{n_k} e_1;\\
x(Q_{n_k}) = \lambda_{n_k}^2 - x_{n_k} - e_1;\\
y(Q_{n_k}) = - \lambda_{n_k} x(Q_{n_k}) - \nu_{n_k} = - \lambda_{n_k} \left ( \lambda_{n_k}^2 - x_{n_k} - 2 e_1 \right ).
\end{cases}
\end{equation}
One develops:
\begin{eqnarray}
\lambda_{n_k}^2 &=& \frac{x_{n_k}^3+A x_{n_k} +B}{(x_{n_k}-e_1)^2} = \frac{(x_{n_k}-e_2)(x_{n_k}-e_3)}{(x_{n_k}-e_1)},
\end{eqnarray}
where $e_2$ and $e_3$ are the two other roots of $X^3+AX+B$. 
Since $e_2,e_3\not = e_1$ by non-singularity of $E$, one obtains $\lim_{k \rightarrow \infty} \vert \lambda_{n_k} \vert_w = \infty$, and henceforth:
\begin{equation}
\begin{cases}
\frac{x(Q_{n_k})}{y(Q_{n_k})} = - \frac{\lambda_{n_k}^2 - x_{n_k} - e_1}{\lambda_{n_k} \left ( \lambda_{n_k}^2 - x_{n_k} - 2 e_1 \right )} \rightarrow 0 \textrm{, as } {k \rightarrow \infty};\\
\frac{1}{y(Q_{n_k})} = - \frac{1}{\lambda_{n_k} \left ( \lambda_{n_k}^2 - x_{n_k} - 2 e_1 \right )} \rightarrow 0 \textrm{, as } {k \rightarrow \infty}.
\end{cases}
\end{equation}
Thus, $\lim_{k\rightarrow \infty} Q_{n_k} = O$, as expected. Here, we have thus used the $(x,z)$-coordinates of $Q_{n_k}$.

Next, based on (\ref{eq:eq66final7}) and (\ref{eq:eq67final7}), one has:
\begin{equation}
\begin{cases}
x([a]Q_{n_k}) = \frac{x(\Psi_{a}^\prime)^2 - \Psi_{a+1}^\prime\Psi_{a-1}^\prime}{(\Psi_{a}^\prime)^2};\\
y([a]Q_{n_k}) = \frac{\Psi_{a+2}^\prime (\Psi_{a-1}^\prime)^2 - \Psi_{a-2}^\prime (\Psi_{a+1}^\prime)^2}{ 4 y (\Psi_{a}^\prime)^3},
\end{cases}
\end{equation}
where $y$ stands for $y(Q_{n_k})$ and $a=\ell-1$, which is an even integer since $\ell\not =2$. 

Now, one computes, based on Lemma \ref{thm:Lemma5}:
\begin{eqnarray}
\left ( \Psi_{a+2}^\prime (\Psi_{a-1}^\prime)^2 - \Psi_{a-2}^\prime (\Psi_{a+1}^\prime)^2 \right ) / (4 y) &=& 
\left ( f_{a+2}(f_{a-1})^2 - f_{a-2}(f_{a+1})^2 \right ) / 4\nonumber\\
&=& \left ( 4x^{3a^2/2} + \textrm{ terms of lower degree} \right) / 4.\nonumber\\
\end{eqnarray}
One also has:
\begin{eqnarray}
(\Psi_{a}^\prime)^3 &=& (f_{a})^3 y^3\nonumber\\
&=& \left ( a^3 x^{3a^2/2 - 6} + \textrm{ terms of lower degree} \right ) y^3.
\end{eqnarray}
This yields:
\begin{eqnarray}
y([a]Q_{n_k}) &=& \frac{\left ( x^{3a^2/2} + \textrm{ terms of lower degree} \right)}{\left ( a^3 x^{3a^2/2 - 6} + \textrm{ terms of lower degree} \right ) y^3}\nonumber\\
&\sim & \frac{1}{a^3} \frac{x^6}{y^3},
\end{eqnarray}
where $x$ stands for $x(Q_{n_k})$. It follows that:
\begin{equation}
\frac{1}{y([a]Q_{n_k})} \sim a^3 \frac{y^3}{x^6} \sim  - a^3 \frac{\lambda_{n_k}^{9}}{\lambda_{n_k}^{12}},
\end{equation}
which yields:
\begin{equation}
\lim_{k \rightarrow \infty} \frac{1}{y([a]Q_{n_k})} = 0,
\end{equation}
since $\lim_{k \rightarrow \infty} \vert \lambda_{n_k} \vert_w = \infty$. 

One shows similarly that:
\begin{equation}
\lim_{k \rightarrow \infty} \frac{x([a]Q_{n_k})}{y([a]Q_{n_k})} = 0.
\end{equation}
Namely, one has from Lemma \ref{thm:Lemma5}:
\begin{eqnarray}
x (\Psi_{a}^\prime)^2 - \Psi_{a+1}^\prime \Psi_{a-1}^\prime &=& 
x (x^3+Ax+B) (f_{a})^2 - f_{a+1} f_{a-1}\nonumber\\
&=& \left ( x^{a^2} + \textrm{ terms of lower degree} \right).\nonumber\\
\end{eqnarray}
One also has:
\begin{eqnarray}
(\Psi_{a}^\prime)^2 &=& (f_{a})^2 y^2\nonumber\\
&=& \left ( a^2 x^{a^2 - 4} + \textrm{ terms of lower degree} \right ) y^2.
\end{eqnarray}
This yields:
\begin{eqnarray}
x([a]Q_{n_k}) &=& \frac{\left ( x^{a^2} + \textrm{ terms of lower degree} \right)}{\left ( a^2 x^{a^2 - 4} + \textrm{ terms of lower degree} \right ) y^2}\nonumber\\
&\sim & \frac{1}{a^2} \frac{x^4}{y^2},
\end{eqnarray}
so that:
\begin{eqnarray}
\frac{x([a]Q_{n_k})}{y([a]Q_{n_k})} &\sim& a \frac{y}{x^2} \sim - a \frac{\lambda_{n_k}^3}{\lambda_{n_k}^4}.
\end{eqnarray}

One concludes that $\lim_{k \rightarrow \infty} [a]Q_{n_k} = O$; {\em i.e.}, $\lim_{k \rightarrow \infty} [a]P_{n_k} = O$, 
since $a$ is an even integer, as $\ell\not =2$, and $P_*$ is a $2$-torsion point. 
Henceforth, one obtains $\lim_{k \rightarrow \infty} [1-\ell]P_{n_k} = O$ 
because $[1-\ell]P_{n_k} = (x([\ell-1]P_{n_k}),-y([\ell-1]P_{n_k}))$ as $a_1=a_3=0$ \cite[p. 53]{silverman2009}. 
But on the other hand, from above, one has $\lim_{k \rightarrow \infty} [1-\ell]P_{n_k}=(0,\pm \sqrt{B}, 1) \not = O$, a contradiction.

Therefore, no root of $X^3+AX+B$ can be an accumulation point of the set $S$. This means that $\vert x_n - e_1 \vert_w \geq \delta$, 
for some constant $\delta>0$, for any $n\geq 1$ and any root $e_1$ of $X^3+AX+B$. This completes the proof of inequality ii), 
since (\ref{eq:eq103final7}) is then valid. 

{\em Step 7.} Part a) being proved, we now show part b) under the assumptions on $E$ and $\ell$ stated in the proposition. 

Let $F/\mathbb{Q}$ be the normal closure of the extension obtained by adjoining the roots $e_i$, $i=1,2,3$, of $X^3+AX+B$. 
Thus, $[F:\mathbb{Q}] \mid 6$. 
Consider $L_n^\prime = L_n F$, and let $\tr_{n,n-1}^\prime$ denote the trace map of the relative extension $L_{n}^\prime/L_{n-1}^\prime$. 
Note that $L_n$ and $L_{n-1}^\prime$ are linearly disjoint over $L_{n-1}$, 
since $[L_n:L_{n-1}]=\ell^4$ and $[L_{n-1}^\prime:L_{n-1}] \mid 6$, as $\ell>3$ by assumption. 

Since $x_n$ is a root of the polynomial $f_{\ell^n}(X)$, whereas any root $e_i$ of $X^3+AX+B$ is not, 
as $\ell\not =2$, it follows that $x_n-e_i \not =0$. 
One then computes:
\begin{eqnarray}
\frac{1}{x_n^3+Ax_n+B} &=& \frac{1}{(x_n-e_1)(x_n-e_2)(x_n-e_3)}\nonumber\\ 
&=&  
\sum_{i=1}^{3}\frac{A_i}{(x_n-e_i)},
\end{eqnarray}
where the coefficients $A_i=A_i(e_1,e_2,e_3)$, $i=1,2,3$, belong to $F$. 
Concretely, one has:
\begin{equation}
\begin{cases}
A_1=- \frac{1}{(e_1-e_3)(e_2-e_1)};\\ 
A_2=- \frac{1}{(e_3-e_2)(e_3-e_1)};\\
A_3=- \frac{1}{(e_3-e_2)(e_1-e_3)}.\\
\end{cases}
\end{equation}
Thus, one obtains:
\begin{eqnarray}
\tr_{n,n-1}^\prime \left ( \alpha_n \right ) 
&=& \Delta^\prime \ell^3 \Bigl \{ \sum_{i=1}^{3} A_i\tr_{n,n-1}^\prime \Bigl ( \frac{1}{(x_n-e_i)} \Bigr )\Bigr \}.
\end{eqnarray}

{\em Step 8.} 
We consider the polynomial of Corollary \ref{thm:Corollary5}, taking $m=\ell$ and $\lambda=x_{n-1}$:
\begin{eqnarray}
\Phi_{\ell}(X, x_{n-1}) &=& \left ( X - x_{n-1}\right ) f_{\ell}^2(X) - f_{\ell-1}(X) f_{\ell+1}(X) \left ( X^3 +AX +B \right )\nonumber\\
&:=& \sum_{j=0}^{\ell^2} a_j X^{j}. 
\end{eqnarray}

We observe that $\Phi_{\ell}(X + e_i, x_{n-1})$ is the minimal polynomial of $x_n-e_i$ over $L_{n-1}^\prime$. 
Furthermore, dividing $\Phi_{\ell}(X + e_i, x_{n-1})$ by $X^{\ell^2}$ and making the change of variable $Y:=1/X$, one obtains 
the polynomial:
\begin{eqnarray}
\sum_{j=0}^{\ell^2} a_j (X + e_i)^{j} X^{-j} X^{-(\ell^2-j)} &=& \sum_{j=0}^{\ell^2} a_j (1 + e_i Y)^{j} Y^{\ell^2-j}, 
\end{eqnarray} 
which is the minimal polynomial of $1/(x_n-e_i)$  over $L_{n-1}^\prime$.

We then obtain the trace $\tr_{n,n-1}^\prime$ of $1/(x_n-e_i)$ as the coefficient of the monomial $-Y^{\ell^2-1}$, which is:
\begin{eqnarray}
-\sum_{j=1}^{\ell^2} a_j j e_i^{j-1} &=&  -\frac{d}{dX}\Phi_{\ell}^\prime(X,x_{n-1})\Bigr \vert_{X=e_i}.
\end{eqnarray}
But then, one computes:
\begin{eqnarray}
&& \frac{d}{dX}\Phi_{\ell}^\prime(X,x_{n-1})\Bigr \vert_{X=e_i} \nonumber\\
&&  \quad = f_{\ell}^2(e_i) + (e_i-x_{n-1})2 f_{\ell}(e_i) \frac{d}{dX}f_{\ell}(X)\Bigr \vert_{X=e_i}
- f_{\ell-1}(e_i) f_{\ell+1}(e_i) (3e_i^2+A), \nonumber\\
\end{eqnarray}
since $e_i^3+Ae_i+B=0$. 
Therefore, one obtains:
\begin{eqnarray}
&& \tr_{n,n-1}^\prime \left ( \alpha_n \right ) \nonumber\\
&& \quad = - \Delta^\prime \ell^3 \Bigl \{ \sum_{i=1}^{3} A_i\Bigl ( f_{\ell}^2(e_i) + 2 e_i f_{\ell}(e_i) \frac{d}{dX}f_{\ell}(X)\Bigr \vert_{X=e_i}
- f_{\ell-1}(e_i) f_{\ell+1}(e_i) (3e_i^2+A) \Bigr )\Bigr \} \nonumber\\
&& \quad + \Delta^\prime \ell^3 \Bigl \{ \sum_{i=1}^{3} A_i\Bigl ( 2 x_{n-1}f_{\ell}(e_i) \frac{d}{dX}f_{\ell}(X)\Bigr \vert_{X=e_i}
\Bigr )\Bigr \}
\end{eqnarray}

Applying Corollary \ref{thm:Corollary7} iteratively, one then computes:
\begin{eqnarray}
&& \frac{\tr_{L_n^\prime/L_1^\prime} \left ( \alpha_n \right )}{[L_{n}^\prime:L_1^\prime]} \nonumber\\
&& \quad = - \frac{\Delta^\prime}{\ell} \Bigl \{ \sum_{i=1}^{3} A_i \Bigl ( f_{\ell}^2(e_i) + 2 e_i f_{\ell}(e_i) \frac{d}{dX}f_{\ell}(X)\Bigr \vert_{X=e_i}
- f_{\ell-1}(e_i) f_{\ell+1}(e_i) (3e_i^2+A) \Bigr )\Bigr \} \nonumber\\
&& \quad + \frac{\Delta^\prime}{\ell} \Bigl \{ \sum_{i=1}^{3} A_i  2 x_{1}f_{\ell}(e_i) \frac{d}{dX}f_{\ell}(X)\Bigr \vert_{X=e_i} \Bigr \}.
\end{eqnarray}

Altogether, we have reached the identity:
\begin{eqnarray}
&& \frac{\tr_{n} \left ( \alpha_n \right )}{[L_n:\mathbb{Q}]} = \frac{\tr_{n}^\prime \left ( \alpha_n \right )}{[L_n^\prime:\mathbb{Q}]} \nonumber\\
&& \quad = - \frac{\Delta^\prime}{\ell [F:\mathbb{Q}]} \tr_{F/\mathbb{Q}} \Bigl \{ \sum_{i=1}^{3} A_i \Bigl ( f_{\ell}^2(e_i) + 2 e_i f_{\ell}(e_i) \frac{d}{dX}f_{\ell}(X)\Bigr \vert_{X=e_i} \Bigr )\Bigr \} \nonumber\\
&& \quad + \frac{\Delta^\prime}{\ell [F:\mathbb{Q}]} \tr_{F/\mathbb{Q}} \Bigl \{ \sum_{i=1}^{3} A_i f_{\ell-1}(e_i) f_{\ell+1}(e_i) (3e_i^2+A) \Bigr \} \nonumber\\
&& \quad + \frac{\Delta^\prime}{\ell [L_1^\prime:\mathbb{Q}]} \tr_{L_1^\prime/\mathbb{Q}} \Bigl \{ \sum_{i=1}^{3} A_i 2 x_{1}f_{\ell}(e_i) \frac{d}{dX}f_{\ell}(X)\Bigr \vert_{X=e_i} \Bigr \}.
\end{eqnarray}

Therefore, the rational number ${\tr_{n} \left ( \alpha_n \right )}/{[L_n:\mathbb{Q}]}$ has only finitely many values 
for $n>1$, namely the values reached when taking $n=2$. As for $n=1$, there are only finitely many values for $\alpha_1$.   
This completes the proof of part b). 
\end{proof}

%%%%%%%%%%%%%%%%%%%%%%%%%%%%%%%%%%%%%%%%%%%%%%%%%%%%

\subsection{Liftings of points on reduced elliptic curves}
\label{subsection:liftingsReduced}

In this section, we continue to specialize results to the case where $E$ is an elliptic curve over $\mathbb{Q}$ without CM.

%%%%%%%%%%%%%%

\begin{thm}
\label{thm:Theorem6}
Let $E$ be an elliptic over $\mathbb{Q}$ without CM, with Weierstrass equation $y^2=x^3+Ax+B$, 
where $A,B\in \mathbb{Z}$. 
Set $\Delta^\prime = 4 A^3 + 27 B^2$. 
Let $\ell>3$ be a fixed prime number that does not divide $\Delta^\prime$. 
Assume that the isomorphism $\rho_\ell:\Gal(L_\infty/\mathbb{Q})\xrightarrow{\approx} {\bf GL}_2(\mathbb{Z}_\ell)$ holds. 

Let $p$ be a prime number such that $p \nmid \Delta^\prime \ell (\ell-1) (\ell+1)$.
Let $\widetilde P$ be any non-trivial point of the $\ell$-component of the reduced curve $\widetilde E(\mathbb{F}_{p})$. 
Let the prime power $\ell^n$, with $n\geq 1$, be the order of the point $\widetilde P$. 

Then, the point $\widetilde P$ can be lifted to a point $P$ of $E(\overline{\mathbb{Q}})$  
with affine $y$-coordinate satisfying: 
\begin{equation}
\label{eq:eq132final8}
y(P) = \Bigl ( \frac{a}{b} \Bigr )^{1/2},
\end{equation}
where $a$ and $b$ are integers such that $1\leq \vert a \vert, \vert b \vert \leq C$, 
for some constant $C$ independent of $n$, and where both $a$ and $b$ are coprime with $p$. 
\end{thm}

\begin{proof}  
{\em Step 1.} Let $\ell$ be a fixed prime number other than $2$ and $3$. 
Let $p\nmid \Delta^\prime \ell (\ell-1)(\ell+1)$ be a prime number. 
In particular, since $\ord_p(\Delta)=0$, the reduced curve $\widetilde E(\mathbb{F}_p)$ is non-singular.  
 
Let $\widetilde P$ be a non-trivial point of the $\ell$-component of the reduced curve $\widetilde E(\mathbb{F}_{p})$. 
Let $(\bar{x},\bar{y})$ be the affine coordinates of $\widetilde P$ and set 
$\ell^n = \ord_{\widetilde E}(\widetilde P)$ with $n\geq 1$.

Let $\xi: \overline{\mathbb{Q}} \hookrightarrow \overline{\mathbb{Q}}_p$ be a fixed embedding of fields.
From Lemma \ref{thm:Lemma2} applied to $m=\ell^n$, $\widetilde P$ can be lifted to a point $P^\prime$ in $E(\mathbb{Q}_p)$,  
since $(\ell^n,p)=1$. Since $\xi$ induces an isomorphism 
$E[\ell^n](\overline{\mathbb{Q}}) \xrightarrow{\approx} E[\ell^n](\overline{\mathbb{Q}}_p )$ 
\cite[Corollary 6.4, part b), p. 86]{silverman2009}, 
it follows that there is an element $P_n \in E[\ell^n](\overline{\mathbb{Q}})$ such that $\xi(P_n) = P^\prime$.

{\em Step 2.} We denote $\mathbb{Q}(E[\ell^n])$ by $L_n$, and its integer ring by $\mathcal{O}_n$. 
From part a) of Lemma \ref{thm:Lemma3}, one has $\ell x(P_n), \ell y(P_n)\in \mathcal{O}_n$, 
and  $\left ( \ell y(P_n) \right )^2 \alpha_n = \Delta^\prime \ell^5$, where $\alpha_n$ is given by (\ref{eq:eq50final7}):
\begin{equation}
\label{eq:eq133final8}
\alpha_n = \ell^3 [ 4 f(x_n) x_n^\prime -g(x_n) ] = \ell^3 \left ( 12x_n^2 x_n^\prime + 16 A x_n^\prime - 3x_n^3 + 5A x_n + 27 B \right ),
\end{equation}
with $x_n=x(P_n)$ and $x_n^\prime=x([2]P_n)$. 
From step 1, there exists $y\in \mathbb{Z}$ such that:
\begin{equation}
y \equiv \ell y(P_n) \mod \mathfrak{P},
\end{equation}
where $\mathfrak{P}\mid p$ is the prime ideal of $L_n$ that 
induces the embedding $\xi: \bigl ( L_n \bigr )_{\mathfrak{P}} \hookrightarrow \overline{\mathbb{Q}}_p$.

Taking $y\in \mathbb{Z}$ as above, we then have:
\begin{eqnarray}
\ell^2 y^2(P_n) \alpha_n = \Delta^\prime \ell^5 &\Rightarrow& y^2 \alpha_n = \Delta^\prime \ell^5 + \alpha_0 \nonumber\\
&\Rightarrow& y^2  \tr_n \left ( \alpha_n \right) = \Delta^\prime \ell^5 [L_n:\mathbb{Q}] + \tr_n \left ( \alpha_0 \right) \in \mathbb{Z},
\end{eqnarray}
for some $\alpha_0 \in \mathfrak{P}$, where $\tr_n$ denotes the trace map from $L_n$ onto $\mathbb{Q}$ and 
$[L_n:\mathbb{Q}]$ denotes the degree of $L_n/\mathbb{Q}$.  
 
Now, $\tr_n \left ( \alpha_0 \right) \in (p)$, since $\alpha_0 \in \mathfrak{P}$.
Indeed, one has a commutative diagram:
\begin{equation}
\begin{CD}
L_n \otimes_{\mathbb{Q}} \mathbb{Q}_p @>>> \prod_{\mathfrak{P}^\prime\mid p} \left ( L_n \right)_{\mathfrak{P}^\prime}\\
@VV{\tr_n}V @VV{\sum_{\mathfrak{P}^\prime\mid p} \tr_{\left ( L_n \right)_{\mathfrak{P}^\prime}/\mathbb{Q}_p}}V\\
\mathbb{Q}_p @= \mathbb{Q}_p.\\
\end{CD}
\end{equation}
Moreover, from \cite[Corollary 1, p. 142]{weil1974}, one has $\tr_{\left ( L_n \right)_{\mathfrak{P}^\prime}/\mathbb{Q}_p}(\alpha_0)\in (p)$, for each 
$\mathfrak{P}^\prime\mid p$. 
  
This yields:
\begin{equation}
\label{eq:eq137final8}
y^{2} \tr_n \left ( \alpha_n \right ) \equiv \Delta^\prime \ell^5 [L_n:\mathbb{Q}] \mod p.
\end{equation}
But $[L_n:\mathbb{Q}]$ divides $\ell(\ell-1)^2(\ell+1)\ell^{4(n-1)}$, 
as follows from Lemma \ref{thm:Lemma4} applied to $m=\ell^n$. 
Henceforth, having assumed that $p \nmid \Delta^\prime \ell(\ell-1)(\ell+1)$, one obtains:
\begin{equation}
\label{eq:eq138final8}
y^{2} \equiv \frac{\Delta^\prime \ell^5}{\tr_n \left ( \alpha_n \right )/[L_n:\mathbb{Q}]} \mod p.
\end{equation}
Indeed, we see that $\tr_n \left ( \alpha_n \right ) \not \equiv 0 \mod p$, for otherwise  
we obtain the contradiction $\Delta^\prime \ell^5 [L_n:\mathbb{Q}] \equiv 0 \mod p$, using (\ref{eq:eq137final8}). 
A similar argument shows that $y\not \equiv 0 \mod p$.

{\em Step 3.} It then follows that:
\begin{equation}
y(P_n) \equiv y/\ell \mod \mathfrak{P},
\end{equation}
and 
\begin{equation}
y/\ell \equiv y(P) \mod \mathfrak{p},
\end{equation}
where
\begin{equation}
\label{eq:eq141final8}
y(P):=\left ( \frac{\Delta^\prime \ell^3}{\tr_n \left ( \alpha_n \right )/[L_n:\mathbb{Q}]} \right )^{1/2},
\end{equation}
and $\mathfrak{p}$ is the prime ideal of $K:=\mathbb{Q}(y(P))$ that induces the embedding 
$\xi:K_{\mathfrak{p}} \hookrightarrow \overline{\mathbb{Q}}_p$.

{\em Step 4.} At this point, we consider the dependency of $\alpha_n$ on $n$. Under the assumptions 
stated in the theorem, Proposition \ref{thm:Proposition2} implies that:
\begin{equation}
\frac{\tr_n(\alpha_n)}{[L_n:\mathbb{Q}]} = \frac{b}{c}, 
\end{equation}
where $c\not=0$ and $b$ are integers that are bounded independently of $n$, and which we may assume relatively prime.  
Moreover, we observe that $b$ has to be different from $0$ and actually coprime with $p$, since $\tr_n(\alpha_n) \not \equiv 0 \mod p$ from step 2. Furthermore, $c$ is coprime with $p$ since $[L_n:\mathbb{Q}]\not \equiv 0 \mod p$. 
The theorem follows using (\ref{eq:eq141final8}), upon setting $a=\Delta^\prime \ell^3 c$.
\end{proof}

%%%%%%%%%%%%%%

%%%%%%%%%%%%%%

\noindent {\bf Remark 3.}
\label{remark3}
The lifting $P$ considered in Theorem \ref{thm:Theorem6} might not be a $\ell^n$-torsion point of $E$, but it projects down to a $\ell^n$-torsion point $\widetilde P$ of the reduced curve $\widetilde E(\mathbb{F}_p)$.
\\

%%%%%%%%%%%%%%

%%%%%%%%%%%%%%

\noindent {\bf Remark 4.}
\label{remark4}
One observes that there are 
only finitely many admissible possibilities for the $y$-coordinate 
$y(P)=\Bigl ( \frac{a}{b} \Bigr )^{1/2}$ of the lifting $P$ considered in the statement of Theorem \ref{thm:Theorem6}. Note that both $a$ and $b$ are units of $\mathbb{Z}_p$ in the statement of Theorem \ref{thm:Theorem6}. If one chooses to express $b^{-1}$ as an integer modulo $p$, there might be infinitely many 
such resulting integers as $p$ varies. But this is unnecessary. 
The point here, is to have finitely many possibilities for $y^2(P)$, rather than expressing $y^2(P)$ as an integer modulo $p$. 
\\

%%%%%%%%%%%%%%

To appreciate Theorem \ref{thm:Theorem6}, let us observe that, from \cite[Corollary 2]{kohel2000},  
it follows directly that the reduced curve $\widetilde E(\mathbb{F}_p)$ of an elliptic curve $E$ over $\mathbb{Q}$ admits a set of generators $\widetilde P_i$ with $y$-coordinates 
satisfying the condition:
\begin{equation}
0 \leq y(\widetilde P_i) \leq \lceil 20 (1 + \log p) p^{1/2} \rceil,
\end{equation}
upon taking the rational function $f=y$ of degree $3$ in this result. This implies in turn that any $\ell^n$-torsion point 
$\widetilde P$ of the reduced curve $\widetilde E$ can be lifted to a point $P$ of $E$ with $x$-coordinate belonging to the compositum of all field extensions over $\mathbb{Q}$ generated by the roots of cubic equations of the form 
\begin{equation}
X^3+AX+B=y^2,
\end{equation}
for some $y \in \mathbb{Z}$ such that $0 \leq y \leq \lceil 20 (1 + \log p) p^{1/2} \rceil$. 
In particular, there are infinitely many such extensions to consider as prime $p$ varies.

On the other hand, Theorem \ref{thm:Theorem6} restricts to $\ell^n$-torsion points  
$\widetilde P$ of the reduced curve $\widetilde E$, and states that 
a lifting $P$ of $\widetilde P$ can be chosen so as to satisfy (\ref{eq:eq132final8}). 
Most importantly, there are only finitely many possibilities for the right-hand side of this equation.\\ 

%%%%%%%%%%%%%%%%%%%%%%%%%%%%%%%%%%%%%%%%%%%%%%%%%%%%

\subsection{Properties of the specific algebraic number fields used for liftings}
\label{subsection:liftingsField}

It will be convenient to define a field $K^\prime$ step by step as follows. 

First, we consider the cyclotomic field:
\begin{equation}
\label{eq:eq145final8}
K_1=\mathbb{Q}(\mu_{4})=\mathbb{Q}(\sqrt{-1});
\end{equation}
then, the field obtained by adjoining radicals:
\begin{equation}
\label{eq:eq146final8}
K_2: = K_1(p_1^{1/2},...,p_\nu^{1/2}),
\end{equation}
where $p_1,...,p_\nu$ are the $\nu$ distinct prime numbers other than $\ell$ that are bounded by the constant $C$ appearing in Theorem \ref{thm:Theorem6}; 
$C$ is a positive constant depending on $E$ and $\ell$. 
Next, we consider the field obtained by adjoining the remaining radical:
\begin{equation}
\label{eq:eq147final8}
K_3: = K_2(\ell^{1/2}).
\end{equation}
Lastly, one adjoins over $K_3$ the roots $x_1,x_2,x_3$ of cubic equations of the form:
\begin{equation}
\label{eq:eq148final8}
X^3+AX+B=y^2,
\end{equation} 
as $y^{2}$ covers the set $\mathcal{U}$ 
of rational numbers of the form $\frac{a}{b}$, 
where $a$ and $b$ are integers satisfying $1 \leq \vert a \vert, \vert b \vert \leq C$. 
The resulting field extension is denoted $K_3^{(y)}$.

This yields the compositum of fields:
\begin{equation}
\label{eq:eq149final8}
K^\prime: = \prod_{y^{2} \in \mathcal{U}} K_3^{(y)}.
\end{equation}

%%%%%%%%%%%%%%

\begin{prop}
\label{thm:Proposition3}
The extension $K^\prime/\mathbb{Q}$ defined in (\ref{eq:eq149final8}) is a normal extension that contains all affine coordinates of 
liftings $P$ appearing in Theorem \ref{thm:Theorem6}.

Moreover, the degree $[K^\prime:\mathbb{Q}]$ is of the form $2^{s} 3^{t}$, 
for some non-negative integers $s$ and $t$. 
In particular, $[K^\prime:\mathbb{Q}]$ is coprime with $\ell>3$.
\end{prop}

\begin{proof} We proceed step by step as follows.

Firstly, the cyclotomic field $K_1:=\mathbb{Q}(\mu_{4})$ has degree $2$ over $\mathbb{Q}$. 

Next, consider a Kummer extension of the form $K_1( p^{1/2} )/K_1$, 
where $p$ is the prime $\ell$ or one of the prime numbers bounded by $C$. 
This Kummer extension has relative degree dividing $2$. 
Moreover, since the prime $p$ is fixed under Galois action of  
$\Gal(K_1/\mathbb{Q})$, it follows that $K_1( p^{1/2} )$ is normal over $\mathbb{Q}$. 
Hence, both $K_2/\mathbb{Q}$ and $K_3/\mathbb{Q}$ are normal extensions. 

It is clear that a Kummer extension of the form $K_1(y)/K_1$, 
where $y^2 = \frac{a}{b}$ with 
$1 \leq \vert a \vert, \vert b \vert \leq C$, is contained in $K_3$.  

Next, fixing a rational number $y$ as above, one obtains a cubic equation:
\begin{equation}
X^3+AX+B = y^2,
\end{equation}
whose roots are in $K^\prime$, by construction. 
The relative normal closure of this equation over $K_3$ has relative degree dividing $6$. 

Lastly, Galois action on the roots of such a cubic equation yields roots of another such cubic equation. 
Thus, $K^\prime$ is normal over $\mathbb{Q}$.
\end{proof}

%%%%%%%%%%%%%%

%%%%%%%%%%%%%%

\begin{prop} 
\label{thm:Proposition4}
Assume that $p \not \equiv 1 \mod \ell$. 
Then, the $\ell$-component of the reduced curve $\widetilde E(\mathbb{F}_p)$ is cyclic. 
\end{prop}

\begin{proof} 
If $\widetilde E(\mathbb{F}_p)_{\ell}$ is not cyclic, then it contains $\widetilde E[\ell]$, from which it follows that $\mu_\ell \subset \mathbb{F}_p$ \cite[Corollary 8.1.1, p. 96]{silverman2009} (consequence of the Weil pairing). 
Thus, one would have $\ell \mid (p-1)$, contrary to the assumption that $p \not \equiv 1 \mod \ell$.
\end{proof}

%%%%%%%%%%%%%%

The following results will be crucial in the proof of Lemma \ref{thm:Lemma9}.

%%%%%%%%%%%%%%

\begin{prop}
\label{thm:Proposition5}
Let $\ell$ be a prime number. Let $n$ be a positive integer coprime with $\ell$. 
Define $K_1=\mathbb{Q}(\mu_n)$, and consider a field of the form 
\begin{equation}
K_2 = K_1( p_1^{1/n},...,p_\nu^{1/n}),
\end{equation}
where $p_1,..,p_\nu$ are $\nu$ distinct prime numbers, each coprime with $\ell$. Then, $\ell$ is unramified in $K_2$.
\end{prop}

\begin{proof} Firstly, $\ell$ is unramified in the cyclotomic field $K_1$ \cite[Theorem 2, p. 74]{lang1986}. 

Next, let us consider the extension $K_{(p_1)}:=K_1(p_1^{1/n})$. 
This is a Kummer extension over $K_1$ of degree $d_1$ dividing $n$.
Then, setting $y_1=p_1^{1/n}$, one deduces from Kummer theory that $y_1^{d_1}$ belongs to $K_1$.
In particular, $f_1(X)=X^{d_1}-y_1^{d_1}$ is the minimal polynomial of $y_1$ over $K_1$.
It follows that $\ell$ is coprime with the discriminant of the relative Kummer extension $K_{(p_1)}/K_1$, 
because the different of this extension divides the ideal generated by 
$f_1^\prime(y_1)=d_1y_1^{d_1-1}$ \cite[Corollary 2, p. 56]{serre1979}, 
as $y_1$ belongs to the integer ring of $K_1(y_1)$, and both
$d_1$ and $y_1$ are coprime with $\ell$. 

Proceeding by induction, one considers the Kummer extension $K_{(p_1,...,p_r,p_{r+1})}$ over 
$K_{(p_1,...,p_r)}$, defined as $K_{(p_1,...,p_r)}(p_{r+1}^{1/n})$,  where $1\leq r < \nu$. The same argument as above shows that the relative different of this extension 
divides the ideal generated by $f_{r+1}(y_{r+1})=d_{r+1}y_{r+1}^{d_{r+1}-1}$, 
where $y_{r+1}=p_{r+1}^{1/n}$, $d_{r+1}$ is the degree 
of the relative extension, and $f_{r+1}(X)=X^{d_{r+1}} - y_{r+1}^{d_{r+1}}$ is the minimal polynomial of $y_{r+1}$ over 
$K_{(p_1,...,p_r)}$.

One concludes that $\ell$ is coprime with the different of the relative extension $K_2/K_1$, from the transitivity property 
\cite[Proposition 8, p. 51]{serre1979}. 
Altogether, it follows that $\ell$ is unramified in $K_2$, applying \cite[Proposition 6, p. 50]{serre1979} and 
\cite[Corollary 1, p. 53]{serre1979}. 
\end{proof}

%%%%%%%%%%%%%%

%%%%%%%%%%%%%%

\begin{cor}
\label{thm:Corollary8}
Let $\ell>2$ be a prime number. 
Let $K^\prime$ be the field defined in (\ref{eq:eq149final8}). Then,

a) $\ell$ is unramified in $K_2$, and hence, $\mathbb{Q}(\mu_\ell) \cap K_2 = \mathbb{Q}$;

b) $\mathbb{Q}(\mu_\ell) \cap K_3$ is equal to the unique quadratic subfield $K_0$ of $\mathbb{Q}(\mu_\ell)$ 
({\em i.e.}, $\mathbb{Q}(\sqrt{\ell})$ if $\ell \equiv 1 \mod 4$, or $\mathbb{Q}(\sqrt{-\ell})$ if  
$\ell \equiv 2,3 \mod 4$);

c) the extension $(\mathbb{Q}(\mu_\ell) \cap K^\prime)/\mathbb{Q}$ has degree dividing $12$.
\end{cor}

\begin{proof} Part a). The first statement follows from Proposition \ref{thm:Proposition5} applied to $n=4$ and the distinct prime factors
$p_1,...,p_\nu$ other than $\ell$ that are bounded by $C$, having assumed that $\ell>2$. 

Thus, the prime $\ell$ is unramified in $K_2 \subset K_1( p_1^{1/4},...,p_\nu^{1/4})$. See (\ref{eq:eq146final8}). 
On the other hand, $\ell$ is totally ramified in $\mathbb{Q}(\mu_\ell)$ \cite[Proposition 17, p. 78]{serre1979}. 
This yields the equality $\mathbb{Q}(\mu_\ell) \cap K_{2} = \mathbb{Q}$. 

Part b). The extension $K_3/K_2$ has degree dividing $2$, and hence the same property holds true for 
$(\mathbb{Q}(\mu_\ell)\cap K_3)/(\mathbb{Q}(\mu_\ell)\cap K_2)$. But $\mathbb{Q}(\mu_\ell)\cap K_2 = \mathbb{Q}$ from part a).
Furthermore, $\mathbb{Q}(\mu_\ell)\cap K_3$ contains both quadratic fields $\mathbb{Q}(\sqrt{\ell})$ and $\mathbb{Q}(\sqrt{-\ell})$. 
Thus, it contains the unique quadratic subfield of $\mathbb{Q}(\mu_\ell)$. 

Part c). The field $K^\prime$ is obtained by adjoining to $K_3$ various roots of cubic equations. 
For each of these cubic equations, the splitting field has Galois group a quotient of the permutation group $\mathfrak{S}_3$, 
so that its Galois group has cardinality dividing $6$. Therefore, the normal extension $K^\prime/K_3$ has Galois group of exponent dividing $6$.

Now, we have an isomorphism of groups:
\begin{equation} 
\Gal((\mathbb{Q}(\mu_\ell) \cap K^\prime) \cdot K_3/K_3) \xrightarrow{\approx} 
\Gal((\mathbb{Q}(\mu_\ell) \cap K^\prime)/(\mathbb{Q}(\mu_\ell) \cap K^\prime) \cap K_3).
\end{equation}
But since, $K_3 \subseteq K^\prime$, it follows that 
\begin{equation} 
\Gal((\mathbb{Q}(\mu_\ell) \cap K^\prime) \cdot K_3/K_3) \xrightarrow{\approx} 
\Gal((\mathbb{Q}(\mu_\ell) \cap K^\prime)/(\mathbb{Q}(\mu_\ell) \cap K_3)).
\end{equation}
Then, since $(\mathbb{Q}(\mu_\ell) \cap K^\prime) \cdot K_3 \subseteq K^\prime$, we obtain a surjective group homomorphism:
\begin{equation} 
\Gal(K^\prime/K_3) \twoheadrightarrow
\Gal((\mathbb{Q}(\mu_\ell) \cap K^\prime)/(\mathbb{Q}(\mu_\ell) \cap K_3)).
\end{equation}

Therefore, $\Gal((\mathbb{Q}(\mu_\ell) \cap K^\prime)/(\mathbb{Q}(\mu_\ell) \cap K_3))$ is annihilated by $6$, and henceforth, 
$\Gal((\mathbb{Q}(\mu_\ell) \cap K^\prime)/\mathbb{Q})$ has exponent dividing $12$. 
Since $\Gal(\mathbb{Q}(\mu_\ell)/\mathbb{Q})$ is a cyclic group, 
we conclude that $(\mathbb{Q}(\mu_\ell) \cap K^\prime)/\mathbb{Q}$ has degree dividing $12$. 
\end{proof}

%%%%%%%%%%%%%%

%%%%%%%%%%%%%%%%%%%%%%%%%%%%%%%%%%%%%%%%%%%%%%%%%%%%
%%%%%%%%%%%%%%%%%%%%%%%%%%%%%%%%%%%%%%%%%%%%%%%%%%%%

\section{Proof of the Main Theorem}             
\label{section:ProofMainTheorems}

Let $E$ be an elliptic curve over an algebraic number field $K$ and $\ell$ be a prime number. 
We denote the group of $\ell$-torsion points $E[\ell](\overline{\mathbb{Q}})$ by $E[\ell]$. 
We set $G=\Gal(L/K)$, where $L=K(E[\ell])$. 
The absolute Galois groups $\Gal(\overline{\mathbb{Q}}/K)$ and $\Gal(\overline{\mathbb{Q}}/L)$ 
are denoted $\mathcal{G}$ and $\mathcal{H}$, respectively.

For each places $v_0\mid p$ and $v\mid v_0$ of $K$ and $L$, respectively, where $p$ is a prime number, 
we set $\mathcal{G}_{v_0}=\Gal(\overline{\mathbb{Q}}_p/K_{v_0})$ and 
$\mathcal{H}_v=\Gal(\overline{\mathbb{Q}}_p/L_v)$. We consider the $[\ell]$-Selmer groups:
\begin{eqnarray}
S^{[\ell]}(E/L) &=&\Ker \Bigl \{ \Hom(\mathcal{H},E[\ell]) \xrightarrow{\oplus_v Res^{\mathcal{H}}_{\mathcal{H}_v}} \oplus_{v} 
\Homol^1(\mathcal{H}_v,E(\overline{\mathbb{Q}}_p)) \Bigr \};\\
S^{[\ell]}(E/K) &=& \Ker \Bigl \{ \Hom(\mathcal{G},E[\ell]) \xrightarrow{\oplus_{v_0} Res^{\mathcal{G}}_{\mathcal{G}_{v_0}}} \oplus_{v_0} 
\Homol^1(\mathcal{G}_{v_0},E(\overline{\mathbb{Q}}_p)) \Bigr \}.
\end{eqnarray}
Here, $v$ and $v_0$ cover all places of $L$ and $K$, respectively.

We are interested in computing the groups $S^{[\ell]}(E/K)$ and $S^{[\ell]}(E/L)$. Note that the $[\ell]$-Selmer groups are finite \cite[Theorem 4.2, part b), p. 333]{silverman2009} and, in fact, are finite vector spaces over $\mathbb{F}_\ell$. 
In Sections \ref{subsection:caseQ} and \ref{subsection:proofTheorem5}, we specialize to the case where $K=\mathbb{Q}$. 

%%%%%%%%%%%%%%%%%%%%%%%%%%%%%%%%%%%%%%%%%%%%%%%%%%%%

\subsection{Map from $S^{[\ell]}(E/K)$ into $S^{[\ell]}(E/L)$}
\label{subsection:MapFrom}

We recall the exact inflation-restriction sequence of Galois cohomology.

%%%%%%%%%%%%%%

\begin{lem}[Inflation-restriction exact sequence]
\label{thm:Lemma6}
Let $G$ be a (possibly infinite) profinite group with closed normal subgroup $N$. Let $A$ be a $G$-module. Then, there is an exact (inflation-restriction) sequence:
\begin{equation}
0 \rightarrow \Homol^1(G/N,A^N) \xrightarrow{Inf^{G}_{G/N}} \Homol^1(G,A) \xrightarrow{Res^{G}_{N}} \Homol^1(N,A)^{G/N},
\end{equation}
where $g\in G$ acts on a $1$-cocycle $f:N \rightarrow A$ as $(g\cdot f)(n)=\,^gf(g^{-1}n g)$ for $n\in N$.
\end{lem}

\begin{proof}
As in \cite[p. 420]{silverman2009}, one has an exact sequence:
\begin{equation}
0 \rightarrow \Homol^1(G/N,A^N) \xrightarrow{Inf^{G}_{G/N}} \Homol^1(G,A) \xrightarrow{Res^{G}_{N}} \Homol^1(N,A).
\end{equation}
Then, one computes directly, assuming that $f$ is defined at $g$:
\begin{eqnarray}
\,^{g}f(g^{-1} n g) &=& \,^{g}\{f(g^{-1}) + \,^{g^{-1}}\left ( f(n) + \,^{n} f(g) \right )\}\nonumber\\
&=& \,^{g}f(g^{-1}) + f(n) + \,^{n} f(g) = -f(g) + f(n) + \,^{n} f(g)\nonumber\\
& =& f(n) + (n-1) f(g). 
\end{eqnarray}
So, taking $g\in N$, one sees that the action of $G$ on $\Homol^1(N,A)$ factors through $G/N$. Moreover, if $f$ is defined on all of $G$, one deduces that $Res^{G}_{N}(f)\in \Homol^1(N,A)^{G/N}$. 
\end{proof}

%%%%%%%%%%%%%%

%%%%%%%%%%%%%%

\begin{prop}
\label{thm:Proposition6} 
Let $\ell$ be a prime number and $E$ an elliptic curve over an algebraic number field $K$. 
Set $L=K(E[\ell])$.   
Then, one has a homomorphism of groups:
\begin{equation}
Res^{\mathcal{G}}_{\mathcal{H}}: S^{[\ell]}(E/K) \rightarrow S^{[\ell]}(E/L)^{\mathcal{G}/\mathcal{H}},
\end{equation}
where $\mathcal{G}$ and $\mathcal{H}$ denote the absolute Galois groups of $K$ and $L$, respectively.
\end{prop}

\begin{proof} 
From the inflation-restriction sequence, we obtain a commutative diagram for any places $v_0\mid p$ of $K$ and 
$v\mid v_0$ of $L$:
\begin{equation}
\begin{CD}
\Homol^1(\mathcal{G},E[\ell]) @>{Res^{\mathcal{G}}_{\mathcal{H}}}>> \Homol^1(\mathcal{H},E[\ell])^{\mathcal{G}/\mathcal{H}}\\
@VV{Res^{\mathcal{G}}_{\mathcal{G}_{v_0}}}V  @VV{Res^{\mathcal{H}}_{\mathcal{H}_v}}V \\
\Homol^1(\mathcal{G}_{v_0},E(\overline{\mathbb{Q}}_p))[\ell] @>{Res^{\mathcal{G}_{v_0}}_{\mathcal{H}_v}}>> \Homol^1(\mathcal{H}_v,E(\overline{\mathbb{Q}}_p))[\ell]. 
\end{CD}
\end{equation}
where $\mathcal{G}_{v_0}$ and $\mathcal{H}_v$ are the absolute Galois groups of $K_{v_0}$ and $L_v$, respectively.
The result now follows from an easy diagram chasing.
\end{proof} 

%%%%%%%%%%%%%%

%%%%%%%%%%%%%%%%%%%%%%%%%%%%%%%%%%%%%%%%%%%%%%%%%%%%

\subsection{The group $S^{[\ell]}(E/L)$}
\label{subsection:SelmerL}

In this section, $E$ is an elliptic curve over an algebraic number field $K$. 

Motivated by Proposition \ref{thm:Proposition6}, we consider the $[\ell]$-Selmer group of $E$ over $L=K(E[\ell])$.
 
Now, by construction, the Galois group $\mathcal{H}$ acts trivially on $E[\ell]$, so that one has:
\begin{equation}
\Homol^1(\mathcal{H},E[\ell]) = \Hom(\mathcal{H},E[\ell])\approx \oplus_{i=1}^{2} 
\Hom(\mathcal{H},\mathbb{F}_\ell).
\end{equation}

Moreover, $L$ contains necessarily $\mu_\ell$ because $\Lambda^2 T_\ell(E) \approx T_\ell(\mu)$, where $\mu$ denotes here the multiplicative group. See \cite[p. 99]{silverman2009}.

Thus, any non-trivial character $\chi$ in $\Hom(\mathcal{H},\mathbb{F}_\ell)$ factors through the Galois group of a Kummer extension of degree $\ell$. In particular, one has an isomorphism:
\begin{equation}
L^*/(L^*)^{\ell} \xrightarrow{\approx} \Hom(\mathcal{H},\mathbb{F}_\ell) = \Hom(\mathcal{H}/\mathcal{H}^\ell,\mathbb{F}_\ell).
\end{equation}
Since the group $L^*/(L^*)^{\ell}$ is infinite, we want to specify those group homomorphisms in $\Hom(\mathcal{H},E[\ell])$ that belong to the $[\ell]$-Selmer group of $E$ over $L$.

Thus, we consider a group homomorphism $\Psi \in \Hom(\mathcal{H},E[\ell])$ such that
the following condition holds:
\begin{equation}
\label{eq:eq164final8}
(*) \qquad Res^{\mathcal{H}}_{\mathcal{H}_v}(\Psi) \mapsto  0 \in \Homol^1(\mathcal{H}_v,E(\overline{\mathbb{Q}}_p)),
\end{equation}
for a given finite place $v\mid p$ of $L$. 

Let $L^\prime=\overline{\mathbb{Q}}^{\Ker \Psi}$, $N=\Gal(L^\prime/L)$ and $N_{v^\prime}=N_v$ ($N$ is Abelian) be its decomposition group at $v^\prime\mid v$. Thus, $N$ can be viewed as a subgroup of 
$\mathbb{F}_\ell \oplus \mathbb{F}_\ell$. 
We have a commutative diagram with exact rows:
\begin{equation}
\begin{CD}
0 @>>> \Hom(N,E[\ell]) @>{Inf^{\mathcal{H}}_{N}}>> \Hom(\mathcal{H},E[\ell])\\
@. @VV{Res^{N}_{N_v}}V  @VV{Res^{\mathcal{H}}_{\mathcal{H}_v}}V \\
0 @>>> \Homol^1(N_v,E(L^\prime_{v^\prime})) @>{Inf^{\mathcal{H}_v}_{N_v}}>> \Homol^1(\mathcal{H}_v,E(\overline{\mathbb{Q}}_p)).
\end{CD}
\end{equation}
Therefore, given $\Psi^\prime \in \Hom(N,E[\ell])$, $Res^{N}_{N_v}(\Psi^\prime)$ splits in $E(L^\prime_{v^\prime})$ if and only if 
$Res^{\mathcal{H}}_{\mathcal{H}_v} \circ Inf^{\mathcal{H}}_{N} (\Psi^\prime)$ does in $E(\overline{\mathbb{Q}}_p)$. 
Henceforth, condition (*) -- applied to $\Psi=Inf^{\mathcal{H}}_{N} (\Psi^\prime)$ -- 
is equivalent to:
\begin{equation}
\label{eq:eq166final8}
(**) \qquad Res^{N}_{N_v}(\Psi^\prime) \mapsto 0 \in \Homol^1(N_v,E(L^\prime_{v^\prime})).
\end{equation}
In particular, if $N_v=0$, condition (**) holds trivially. Note also that $\Psi^\prime$ is not identically $0$ on any non-trivial subgroup of $N$. 

%%%%%%%%%%%%%%

\begin{prop}
\label{thm:Proposition7} 
Let $\ell>3$ be a prime number. 

a) In cases B and C, condition (**) holds at a place $v$ of $L$ if only if 

$v$ is unramified in $L^\prime_{v^\prime}$ and $Res^{N}_{N_v}(\Psi^\prime)$ maps to $0\in \Homol^1(N_v,\widetilde E_v(k_{v^\prime}))$, where $k_{v^\prime}$ denotes the residue field of $L^\prime_{v^\prime}$.

b) In Case D, $v$ is at most tamely ramified in $L^\prime_{v^\prime}$.
\end{prop} 

\begin{proof} Part a). 
Case B: $v \mid v_0 \mid p \not = \ell$ and $v_0 \not \in \Sigma_E$. Assume that condition (**) holds. 
Let $I_v$ be the inertia subgroup of $L^\prime_{v^\prime}/L_v$.
Then, there is an exact sequence:
\begin{equation}
0 \rightarrow E_1(L^\prime_{v^\prime}) \rightarrow E(L^\prime_{v^\prime}) \rightarrow \widetilde E_v (k_{v^\prime}) \rightarrow 0, 
\end{equation}
where $k_{v^\prime}$ denotes the residue field of $L^\prime_{v^\prime}$.  
Condition (**) then implies that 
$Res^{N}_{N_v}(\Psi^\prime)$ splits in $\widetilde E_v (k_{v^\prime})$. Since $I_v$ acts trivially on $\widetilde E_v (k_{v^\prime})$ and $E[\ell] \approx \widetilde E_v [\ell] \subset \widetilde E_v (k_{v^\prime})$, this implies that $\Psi^\prime$ is trivial on $I_v$. 
Therefore, $I_v=0$, which means that $v$ is unramified in $L^\prime_{v^\prime}$. 
Furthermore, we obviously have that $Res^{N}_{N_v}(\Psi^\prime)$ maps to $0\in \Homol^1(N_v,\widetilde E_v (k_{v^\prime}))$.

For the converse, it is sufficient to prove that
\begin{equation}
\Homol^1(N_v, E_1(L^\prime_{v^\prime}))=0,
\end{equation}
whenever $L^\prime_{v^\prime}/L_v$ is unramified of degree a power of $\ell$, for then one obtains a commutative diagram:
\begin{equation}
\begin{CD}
@. \Hom(N,E[\ell]) @>\approx>> \Hom(N,\widetilde E_v[\ell])\\
@. @VV{Res^{N}_{N_v}}V  @VV{Res^{N}_{N_v}}V \\
0 @>>>  \Homol^1(N_v,E(L^\prime_{v^\prime})) @>{reduction}>> \Homol^1(N_v,\widetilde E_v(k_{v^\prime})).
\end{CD}
\end{equation}
Now, $E_1(L^\prime_{v^\prime})\approx F_v(\mathcal{M}_{v^\prime})$, where $F_v$ is the formal group of $E$ over 
$\mathcal{O}_{v}$, the integer ring of $L_v$, and $\mathcal{M}_{v^\prime}$ is the maximal ideal of the integer ring $\mathcal{O}_{v^\prime}$ of $L^\prime_{v^\prime}$.  
Moreover, from \cite[Proposition 6.3, p. 200]{silverman2009}, one has a short exact sequence:
\begin{equation}
0 \rightarrow \pi_{v}^r \mathcal{O}_{v^\prime}  \rightarrow 
F_v(\mathcal{M}_{v^\prime}) \rightarrow M \rightarrow 0 
\end{equation}
for some positive integer $r$, where $\pi_v$ is a uniformizer of $L_v$, and 
$M$ is a finite $\mathbb{Z}_p$-representation of $N_v$ that is annihilated by a power of $p$ (recall that $v\mid p\not = \ell$). Since $N_v$ has order $\ell$ (recall that $N_v$ is a cyclic subgroup of $\mathbb{F}_\ell \oplus \mathbb{F}_\ell$), it follows that 
$\Homol^1(N_v,M)=0$. Thus, it is sufficient to show that:
\begin{equation}
\Homol^1(N_v, \mathcal{O}_{v^\prime})=0,
\end{equation}
whenever $L^\prime_{v^\prime}/L_v$ is unramified. This result is well-known and can be proved by using the facts that 
$\Homol^1(N_v,k_{v^\prime})=0$ and  that $\mathcal{O}_{v^\prime}$ is complete. 

Case C: $v \mid v_0 \mid p \not=\ell$ and $p\in \Sigma_E$, with additive reduction of $E$ at $v_0$ and $v_0(j(E))\geq 0$. 
From Lemma \ref{thm:Lemma1} (by assumption, $\ell>3$), it follows that $E$ has good reduction over $L$. Therefore, the argument in case B applies.

Part b). Case D: This is clear since $\ell$ is coprime with the characteristic of the residue field of $L_v$.
\end{proof}

%%%%%%%%%%%%%%

At this point, we make use of class field theory. 
One may consult reference \cite{lang1986} for the classical approach, close to Takagi-Artin's treatment. 
One may also consult \cite{artin1967,tate1967} for global class field theory, and \cite{serre1967,serre1979} for local class field theory, with a treatment based extensively on homological algebra, including the notion of invariant class. 
A treatment based on the notion of henselian valuation with respect to a degree map can be found in \cite{neukirch1986}. 
Reference \cite{fesenko2002} should be consulted for a development of local class field theory, including explicit reciprocity laws, 
that does not rely on homological algebra. 
In this work, we found convenient to refer to \cite{lang1986,neukirch1986,serre1979}.

We now recall the notion of conductor of a finite Abelian extension of local fields $F/L_v$ \cite[p. 44]{neukirch1986}. 
Let $\pi_{v}$ be a uniformizer of $L_v$ and denote the prime ideal $(\pi_{v})$ by $\mathfrak{p}_v$. 
Set $U_{v}^{(0)}=U_{v}$ the group of units of the ring of integers $\mathcal{O}_v$ of $L_v$, and $U_{v}^{(n)}=1+\mathfrak{p}_v^n$ for $n\geq 1$. The conductor of $F/L_v$ is equal to $\mathfrak{f}_{v}=\mathfrak{p}_v^n$, where $n$ is the smallest integer such that 
$U_{v}^{(n)}\subset N_{F/L_v}(F^*)$.

Next, assume that $\mu_\ell\subset L_v^*$. 
Let $(\;\; ,\,L_v^{ab}/L_v): L_v^* \rightarrow \Gal(L_v^{ab}/L_v)$ be the local reciprocity map, where $L_v^{ab}$ denotes the maximal Abelian extension over $L_v$ \cite[pp. 168--171]{serre1979}. Let $\chi_a$ be the Kummer character associated to an element $a$ of $L_v^*$; {\em i.e.}, $\chi_a(\sigma)=\sigma(a^{1/\ell})/a^{1/\ell}$. Then, Hilbert's local symbol is defined as $(a,b)_{v,\ell}=\chi_a((b,\,L_v^{ab}/L_v))\in \mu_{\ell}$ \cite[p. 205--206]{serre1979}. One has: $(a,b)_{v,\ell}=1$ if and only if $b$ is a norm from the extension $L_v(a^{1/\ell})/L_v$ \cite[Proposition 4, p. 206]{serre1979}. It follows that the conductor of the Kummer extension $L_v(a^{1/\ell})/L_v$ is given by the smallest integer $n$ such that $(a,b)_{v,\ell}=1$ for all $b\in U_{v}^{(n)}$.

%%%%%%%%%%%%%%

\begin{prop}
\label{thm:Proposition8} 
Let $\ell>3$. 
Assume that condition (**) holds at a place $v$ of $L$. Let $\mathfrak{f}_v$ 
be the conductor of $L^\prime_{v^\prime}/L_v$. Then,

a) Case A: $\mathfrak{f}_v\mid \mathfrak{p}_v^{1+e_v\ell/(\ell-1)}$, where $e_v$ is the absolute ramification index of $L_v$.

b) Cases B and C: $\mathfrak{f}_v=1$.

c) Case D: $\mathfrak{f}_v\mid \mathfrak{p}_v$.
\end{prop} 

\begin{proof} Part a). 
Case A: The prime number $\ell$ is equal to the characteristic of the residue field of $L_v$. Moreover, one has $\mu_\ell \subset L_v^*$. It is sufficient to consider the case where $F^\prime/L_v$ is a cyclic sub-extension of $L^\prime_{v^\prime}/L_v$ of degree $\ell$. Indeed, $L^\prime_{v^\prime}/L_v$ is either the trivial extension, a cyclic extension of degree $\ell$ or the compositum of two cyclic extensions $F^\prime$ and $F''$ of degree $\ell$. In the first case, there is nothing to prove. In the third case, the inclusions $1+\mathfrak{p}_v^m \subset N_{F^\prime/L_v}((F^\prime)^*)$ and $1+\mathfrak{p}_v^n \subset N_{F''/L_v}((F'')^*)$ imply that 
$1+\mathfrak{p}_v^{\max(m,n)} \subset N_{F^\prime/L_v}((F^\prime)^*) \cap N_{F''/L_v}((F'')^*)=N_{F^\prime F''/L_v}((F^\prime F'')^*)$. Thus, 
we consider the case where $F^\prime/L_v$ is a Kummer extension of degree $\ell$, say $F^\prime=L_v(x^{1/\ell})$, that might be wildly ramified. 
From \cite[p. 186]{lang1986}, one has the inclusion $U_{v}^{(1+e_v\ell/(\ell-1))}\subset U_{v}^{\ell}$. 
But, one obviously has $(x,b)_{v,\ell}=1$ for all $b\in U_{v}^{\ell}$. 
Thus, $\mathfrak{f}_v\mid \mathfrak{p}_v^{1+e_v\ell/(\ell-1)}$.

Part b). Cases B and C: From local class field theory, one has $\mathfrak{f}_{v}=1$ if and only if 
$L^\prime_{v^\prime}/L_v$ is unramified \cite[Proposition (3.4), p. 44]{neukirch1986}, which holds from Proposition \ref{thm:Proposition7} (since $\ell>3$).

Part c). Case D: The prime number $\ell$ is coprime with the characteristic $p$ of the residue field of $L_v$. Moreover, one has $\mu_\ell \subset L_v^*$. As above, it is sufficient to consider the case where $F^\prime/L_v$ is a cyclic sub-extension of 
$L^\prime_{v^\prime}/L_v$ of degree $\ell$. Then, $F^\prime/L_v$ is tamely ramified and it is sufficient to consider the case where it is totally tamely ramified. Thus, we consider the case where $F^\prime/L_v$ is a Kummer extension of the form $L_v(\pi_{v}^{1/\ell})/L_v$ for some uniformizer $\pi_{v}$ of $L_v$ \cite[Proposition 12, p. 52]{lang1986}.   
In the case where $\ell$ is coprime with $p$, Hilbert's local symbol can be computed explicitly as in \cite[pp. 210--211]{serre1979}. Namely, let $(a)=(\pi_{v})^{\alpha}$ and  $(b)=(\pi_{v})^{\beta}$. Set $c=(-1)^{\alpha \beta} a^{\beta}/b^{\alpha}$. Then, $(a,b)_{v,\ell}=\overline{c}^{(q-1)/\ell}$, where $\overline{c}$ is the image of $c$ in the residue field of $L_v$ and $q$ is the cardinality of the residue field. In our case, $a=\pi_{v}$, so that $\alpha=1$. Now, let $b\in U_{v}$, so that $\beta=0$. Then, $c=1/b$, so that $(a,b)_{v,\ell}=(\overline{b})^{-(q-1)/\ell}$. It follows that $(\pi_{v},b)_{v,\ell}=1$ for all $b\in U_{v}^{(1)}$ ({\em i.e.}, the group of units that map to $1$ in the residue field). On the other hand, the extension $L_v(\pi_v^{1/\ell})/L_v$ is ramified, so that 
$\mathfrak{f}_{v}\not = 1$. Therefore, in the totally tamely ramified case, one concludes that $\mathfrak{f}_{v}=\mathfrak{p}_v$. 
\end{proof}

%%%%%%%%%%%%%%

Let $\mathfrak{m}$ be the cycle \cite[p. 97]{neukirch1986} defined as:
\begin{equation}
\label{eq:eq172final8}
\mathfrak{m} = \prod_{v\mid \ell} \mathfrak{p}_v^{1+e_{v} \ell/(\ell-1)} \prod_{v\mid v_0 \in \Sigma_{E,p.m.}} \mathfrak{p}_v,
\end{equation}  
where $e_{v}$ is the absolute ramification index of $L_v$. 
One considers the subgroup $I^{\mathfrak{m}}_L$ of the idele group $I_L$ \cite[p. 98]{neukirch1986}:
\begin{equation}
I^{\mathfrak{m}}_L = \prod_{v\mid \ell} U_v^{(1+e_{v} \ell/(\ell-1))} \times \prod_{v\mid v_0 \in \Sigma_{E,p.m.}} U_v^{(1)} 
\times \prod_{v \not \in T} U_v,
\end{equation}
where $T=\{v: v\mid \ell\} \cup \{v: v\mid v_0 \in \Sigma_{E,p.m.}\}$. 
Here, $\{v\}$ includes the set $S_\infty$ of infinite places of $L$. Since $\mu_{\ell}\subset L^*$, assuming that $\ell\not=2$, it follows that the infinite places of $L$ are all complex. In that case, one sets $U_v=\mathbb{C}^*$. 
From global class field theory \cite[Chapter IV, \S 7]{neukirch1986}, there exists a unique finite Abelian extension $L^{\mathfrak{m}}/L$ such that:
 \begin{equation}
(\;\; ,\,L^{\mathfrak{m}}/L)\::\: C_L/ C^{\mathfrak{m}}_L \xrightarrow{\approx} \Gal(L^{\mathfrak{m}}/L),
\end{equation}                      
where $C_L = \left (L^* \cdot I_L \right )/L^*$ and $C^{\mathfrak{m}}_L = \left (L^* \cdot I^{\mathfrak{m}}_L \right )/L^*$. 

In the next result, we let $I_L^S$ denote $\prod_{v\in S} L_v^* \times \prod_{v\not \in S} U_v$ \cite[p. 76]{neukirch1986}. 
Now, let $T$ be any finite set of prime ideals of an algebraic number field 
$L$. Then, there exists a finite set of primes $S$, disjoint from $T$, such that the
classes of the elements of $S$ generate the ideal class group of $L$ \cite[pp. 124--125]{lang1986}. 
Next, let $S$ be any finite set of primes of an algebraic number field $L$,
such that: 1) $S$ includes the set $S_\infty$ of infinite places of $L$; 2) the classes of
the elements of $S \setminus S_\infty$ generate the ideal class group of $L$. Then, $L^* \cdot I_L^S = L^* \cdot I_L$ 
\cite[pp. 77--78]{neukirch1986}. 

Thus, given a finite set $T$ of non-Archimedean places, there exists a finite set of places $S \supseteq S_\infty$ disjoint from $T$ such that $L^* \cdot I_L^S = L^* \cdot I_L$.

%%%%%%%%%%%%%%

\begin{lem} 
\label{thm:Lemma7}
Let $E$ be an elliptic curve over $K$ and $\ell>2$ be a prime number. Set $L=K(E[\ell])$. 
Let $T=\{v: v\mid \ell\} \cup \{v: v\mid v_0 \in \Sigma_{E,p.m.}\}$. Let $S\supseteq S_\infty$ be a finite set of places of $L$, disjoint from $T$, such that $L^* \cdot I_L^S = L^* \cdot I_L$.  Take $S$ sufficiently large, so that $S$ is closed under Galois action of $\Gal(L/K)$. 
Let $\mathcal{N}$ be the subgroup of $I_L$ defined as:
\begin{equation}
\mathcal{N} = 
\prod_{v\in T} U_v^\ell \times
\prod_{v\in S} U_v \cdot (L_v^*)^\ell \times 
\prod_{v\not \in S \cup T} U_v.
\end{equation} 

Then, $\overline{\mathcal{N}} = \left ( L^* \cdot \mathcal{N} \right )/L^*$ is the class group of the maximal sub-extension $\widetilde L/L$ of $L^{\mathfrak{m}}/L$ whose Galois group is annihilated by $\ell$. 
In particular, $\Gal(\widetilde L/L)$ is a direct product of cyclic groups of order $\ell$ such that:
\begin{equation}
\Hom(\Gal(\widetilde L/L),E[\ell]) = \Hom(\Gal(L^{\mathfrak{m}}/L),E[\ell]).
\end{equation}
Moreover, $\widetilde L/K$ is a Galois extension. 
\end{lem} 

\begin{proof}
Let $\widetilde L/L$ be the maximal sub-extension of 
$L^{\mathfrak{m}}/L$ with Galois group annihilated by $\ell$. Then, one has:
\begin{equation}
\Hom(\Gal(\widetilde L/L),E[\ell]) = \Hom(\Gal(L^{\mathfrak{m}}/L),E[\ell]),
\end{equation}
and $\Gal(\widetilde L/L)$ is a direct product of cyclic groups of order $\ell$. 
We show that $\overline{\mathcal{N}}$ is its class group; 
{\em i.e.}, the class field $L_{\overline{\mathcal{N}}}$ is equal to $\widetilde L$.

First, we observe that $C_L/\overline{\mathcal{N}}\approx (L^* \cdot I_L )/(L^* \cdot \mathcal{N})$ is annihilated by $\ell$, 
since $L^* \cdot I_L = L^* \cdot I_L^S$ and $\left ( I_L^S \right)^{\ell} \subseteq \mathcal{N}$. 
Thus, $L_{\overline{\mathcal{N}}} \subseteq \widetilde L$.

Conversely, the Galois group $\Gal(\widetilde L/L)$ is the maximal quotient group of 
$\Gal(L^{\mathfrak{m}}/L) \approx C_L^{\mathfrak{m}} \approx (L^* \cdot I_L^S)/(L^* \cdot I_L^{\mathfrak{m}})$ that is annihilated by $\ell$. 
But $\prod_{v\in T} U_v^{\ell} \times \prod_{v \in S} \left ( L_v^*\right)^{\ell} \subset \left (I_L^S \right)^{\ell}$, 
and $\prod_{v \in S} U_v \times \prod_{v \not \in S \cup T} U_v \subset I_L^{\mathfrak{m}}$. 
Since both $\left (I_L^S \right)^{\ell}$ and $I_L^{\mathfrak{m}}$ map to $0$ under the natural projection 
$I_L \rightarrow C_L^{\mathfrak{m}}/\left ( C_L^{\mathfrak{m}} \right)^{\ell}$, 
$\overline{\mathcal{N}}$ is contained in the class group of $\widetilde L/L$, 
which means that $\widetilde L \subseteq L_{\overline{\mathcal{N}}}$.

Since $S$ is taken closed under Galois action of $\Gal(L/K)$, the same holds true for the class group $\overline{\mathcal{N}}$. It follows that
$L_{\overline{\mathcal{N}}}$ is closed under any element of $\mathcal{G}=\Gal(\overline{K}/K)$. 
\end{proof}

%%%%%%%%%%%%%% 

Having assumed that $\ell>2$, the group $U_v \cdot (L_v^*)^\ell$ is actually equal to $U_v=L_v^*=\mathbb{C}^*$, for any 
$v\in S_\infty$.  

Combining Propositions \ref{thm:Proposition7} and \ref{thm:Proposition8},  
and Lemma \ref{thm:Lemma7}, we have reached the following result.

%%%%%%%%%%%%%%

\begin{prop} 
\label{thm:Proposition9}
Let $\ell>3$ be a prime number. 
Let $\overline{\mathcal{N}}$ be the class group defined in Lemma \ref{thm:Lemma7}.
Let $\widetilde L = L_{\overline{\mathcal{N}}}$ be the corresponding class field. Set $\widetilde H=\Gal(\widetilde L/L)$. 
Then, one has:
\begin{eqnarray}
S^{[\ell]}(E/L) = 
Inf^{\mathcal{H}}_{\widetilde H} \Ker \Bigl \{ \Hom(\widetilde H,E[\ell]) \rightarrow 
\oplus_{w} \Homol^{1}(\widetilde H_w,E(\widetilde L_w)) \Bigr \},
\end{eqnarray}
where $w$ covers all places of $\widetilde L$. 
\end{prop}

%%%%%%%%%%%%%%

%%%%%%%%%%%%%%%%%%%%%%%%%%%%%%%%%%%%%%%%%%%%%%%%%%%%

\subsection{Returning to the group $S^{[\ell]}(E/K)$}
\label{subsection:Returning}

In this section, $E$ is an elliptic curve over an algebraic number field $K$.

Given a prime number $\ell>3$, we set $L=K(E[\ell])$.
Let $\overline{\mathcal{N}}$ be the class group defined in Lemma \ref{thm:Lemma7}, and 
let $\widetilde L = L_{\overline{\mathcal{N}}}$ be the corresponding class field. 
We set $\widetilde G=\Gal(\widetilde L/K)$ and $\widetilde H=\Gal(\widetilde L/L)$. 

Combining Propositions \ref{thm:Proposition6} and \ref{thm:Proposition9}, we have obtained the following result.

%%%%%%%%%%%%%%

\begin{cor} 
\label{thm:Corollary9}
Let $E$ be an elliptic curve over $K$. 
Let $\ell>3$ be a prime number. 
Then, one has:
\begin{eqnarray}
S^{[\ell]}(E/K) = 
Inf^{\mathcal{G}}_{\widetilde G} \Ker \Bigl \{ \Homol^1(\widetilde G,E[\ell]) \rightarrow 
\oplus_{w} \Homol^{1}(\widetilde G_{w},E(\widetilde L_w)) \Bigr \},
\end{eqnarray}
where $w$ covers the places of $\widetilde L$ and $\widetilde G_{w}$ is the decomposition group of $w$ in $\widetilde L/K$. 
\end{cor}

\begin{proof} Let $w\mid v \mid v_0 \mid p$ be places of $\widetilde L$, $L$, $K$, and $\mathbb{Q}_p$, respectively. 

We show the inclusion:
\begin{equation}
S^{[\ell]}(E/K) \subseteq 
Inf^{\mathcal{G}}_{\widetilde G} \Ker \Bigl \{ \Homol^1(\widetilde G,E[\ell]) \rightarrow 
\oplus_{w} \Homol^{1}(\widetilde G_{w},E(\widetilde L_w)) \Bigr \}.
\end{equation} 
Let $f\in S^{[\ell]}(E/K)$. Then, 
$Res^{\mathcal{G}}_{\mathcal{H}}(f)=Inf^{\mathcal{H}}_{\widetilde H}(g)$, for some 
group homomorphism $g \in \Hom(\widetilde H,E[\ell])$ satisfying the property of Proposition \ref{thm:Proposition9}. 
Then, for $\sigma_1 \in \mathcal{G}$ and $\sigma_2 \in \Gal(\overline{\mathbb{Q}}/\widetilde L)$, one computes: 
\begin{equation}
f(\sigma_1 \sigma_2)=f(\sigma_1)+\,^{\sigma_1} f(\sigma_2)=f(\sigma_1)+\,^{\sigma_1} g(\sigma_2)
=f(\sigma_1).
\end{equation} 
Thus, $f = Inf^{\mathcal{G}}_{\widetilde G}(\widetilde f)$, for some 
$\widetilde f \in \Homol^1(\widetilde G,E[\ell])$, and satisfies the stated property, 
as follows from the following diagram with exact bottom row:
\begin{equation}
\begin{CD}
@.  \Homol^1(\widetilde G,E[\ell]) @>{Inf^{\mathcal{G}}_{\widetilde G}}>> \Homol^1(\mathcal{G},E[\ell])\\
@. @VV{Res^{\widetilde G}_{\widetilde G_{w}}}V  @VV{Res^{\mathcal{G}}_{\mathcal{G}_{v_0}}}V \\
0 @>>> \Homol^1(\widetilde G_{w},E(\widetilde L_w))[\ell] @>{Inf^{\mathcal{G}_{v_0}}_{\widetilde G_{w}}}>> 
\Homol^1(\mathcal{G}_{v_0},E(\overline{\mathbb{Q}}_p))[\ell],
\end{CD}
\end{equation}
where $\mathcal{G}_{v_0}$ denotes the absolute Galois group of $K_{v_0}$.

The other inclusion is clear.
\end{proof}

%%%%%%%%%%%%%% 

%%%%%%%%%%%%%%%%%%%%%%%%%%%%%%%%%%%%%%%%%%%%%%%%%%%% 

\subsection{The group $S^{[\ell]}(E/\mathbb{Q})$}
\label{subsection:caseQ}   

We now specialize to the case where $E$ is an elliptic curve over $\mathbb{Q}$, of Weierstrass equation of the form
$y^2=x^3+Ax+B$, where $A,B\in \mathbb{Z}$. Let $\Delta^\prime:=4A^3+27B^2$. 

We consider a prime number $\ell>3$, and set $L=\mathbb{Q}(E[\ell])$. 
We let $\widetilde L$ denote the class field corresponding to the subgroup $\overline{\mathcal{N}}$ of $C_{L}$ as in Lemma \ref{thm:Lemma7}, when taking the base field $K=\mathbb{Q}$. We set $\widetilde H=\Gal(\widetilde L/L)$.     

Next, we consider the field $K^\prime$ defined in (\ref{eq:eq149final8}), and we denote:
\begin{equation}
\label{eq:eq183final8}
L^\prime: = L K^\prime.
\end{equation}
In addition to the field $\widetilde L$, we also consider $\widetilde L^\prime$ the class field corresponding to the subgroup $\overline{\mathcal{N}}$ of $C_{L^\prime}$ as in Lemma \ref{thm:Lemma7}, when taking the base field $K=K^\prime$. 

The motivation for Theorem \ref{thm:Theorem6} and Proposition \ref{thm:Proposition3} was to reach the following 
result, which is useful for passing from $p$-adic rational points to algebraic ones.

%%%%%%%%%%%%%%

\begin{prop}
\label{thm:Proposition10}
Let $E$ be an elliptic curve over $\mathbb{Q}$ without CM, 
with Weierstrass equation of the form $y^2=x^3+Ax+B$, where $A,B\in \mathbb{Z}$. 
Set $\Delta^\prime=4A^3+27B^2$.

Let $\ell > 3$ be a prime number such that $\rho_\ell(\mathcal{G}) = {\bf GL}_2(\mathbb{Z}_\ell)$. 
Let $p$ be any prime number such that $p\nmid \Delta^\prime \ell (\ell-1) (\ell+1)$.  
Let $\xi$ be a fixed embedding of $\overline{\mathbb{Q}}$ into $\overline{\mathbb{Q}}_p$, inducing an embedding of Galois groups 
$\mathcal{G}_{p}=\Gal(\overline{\mathbb{Q}}_p/\mathbb{Q}_p) \hookrightarrow \mathcal{G}=\Gal(\overline{\mathbb{Q}}/\mathbb{Q})$. 

Then, one can write any point $P_0 \in E(\mathbb{Q}_{p})$ in the form:
\begin{equation}
P_0 = \xi(P) + [\ell] Q^\prime,
\end{equation}
for some point $P\in E(K^\prime)$, where $K^\prime$ is the field defined in (\ref{eq:eq149final8}), 
and $Q^\prime\in E(\mathbb{Q}_{p})$. 

More precisely, let $\mathfrak{p}$ be the prime ideal of $K^\prime$ lying above $p$ such that $\xi$ induces 
$K^\prime_{\mathfrak{p}} \hookrightarrow \overline{\mathbb{Q}}_p$. Let $F$ be the subfield of $K^\prime$ 
fixed by the decomposition group $\Gal(K^\prime/\mathbb{Q})_{\mathfrak{p}}$. Then, in fact, one may take $P\in E(F)$.  
\end{prop}

\begin{proof}  
Since $\ord_p(\Delta)=0$, the reduced curve $\widetilde E$ of $E$ over the residue field $\mathbb{F}_{p}$ is non-singular.  Let $\# \widetilde E(\mathbb{F}_{p}) = \ell^n m$, where $n\geq 0$ and $(m,\ell)=1$. 

Let $P_0\in E(\mathbb{Q}_{p})$. Then, $[m]P_0$ projects to a point $\widetilde P$ of $\widetilde E[\ell^n]$. 
We first consider the non-trivial case where $n>0$. 
From Theorem \ref{thm:Theorem6}, 
there is a lifting $P^\prime$ of $\widetilde P$ with affine coordinates in 
the field $K^\prime$ that projects to $\widetilde P$. 
If $n=0$, then one may take $P^\prime=O$, so that $P^\prime\in E(\mathbb{Q})\subset K^\prime$ and $P^\prime$ 
is trivially a lifting of $[m]P_0 = O$. 

Then, $[m]P_0 - \xi(P^\prime)\in E_1( K^\prime_{\mathfrak{p}})$, where $\mathfrak{p}$ is the maximal ideal of $K^\prime$ 
lying above $p$ that is compatible with the embedding $\xi$. 
Now, one has a commutative diagram:
\begin{equation}
\begin{CD}
E(K^\prime) @>{\xi}>> E(K^\prime_{\mathfrak{p}})\\
@VV{\tr_{K^\prime/F}}V  @VV{\tr_{K^\prime_{\mathfrak{p}}/\mathbb{Q}_{p}}}V\\
E(F) @>{\xi}>> E(\mathbb{Q}_{p}),\\
\end{CD}
\end{equation}
where $F$ is the fixed field of $K^\prime$ by the decomposition group $\Gal(K^\prime/\mathbb{Q})_{\mathfrak{p}}$.
Here, we have used the fact that $F_{\mathfrak{p}^\prime}=\mathbb{Q}_p$, where $\mathfrak{p}^\prime$ is the prime ideal of $F$ lying below
$\mathfrak{p}$, as well as \cite[Exerc. 1.12 b), p.16]{silverman2009}. 
This yields $[m]\tr_{K^\prime_{\mathfrak{p}}/\mathbb{Q}_{p}}(P_0) - \xi(\tr_{K^\prime/F}(P^\prime))\in E_1(\mathbb{Q}_{p})$. 
We set $P'':=\tr_{K^\prime/F}(P^\prime) \in E(F)$. Moreover, one has $\tr_{K^\prime_{\mathfrak{p}}/\mathbb{Q}_{p}}(P_0)=[m^\prime]P_0$, where 
$m^\prime = \vert \Gal(K^\prime/\mathbb{Q})_{\mathfrak{p}} \vert$ is coprime with $\ell$, based on Proposition \ref{thm:Proposition3}. 
We set $m'':=m \cdot m^\prime$.

Now, $E_1(\mathbb{Q}_{p})\approx F_{p}(\mathcal{M}_{p})$, 
where $F_{p}$ is the formal group of $E$ over $\mathbb{Q}_{p}$. 
From \cite{kolyvagin1980}, $F_{p}$ is necessarily a formal $\mathbb{Z}_p$-module. Since $\ell\in \mathbb{Z}_p^{*}$ by assumption, 
it follows that 
$[m'']P_0 - \xi(P'') = [\ell] Q'' \in [\ell](E_1(\mathbb{Q}_{p}))$, for some point $Q'' \in E_1(\mathbb{Q}_{p})$.  
Writing $1=m'' a+\ell b$, with $a,b\in\mathbb{Z}$, one deduces that 
\begin{eqnarray}
P_0 &=& [a]([m'']P_0) + [\ell]([b]P_0) = [a](\xi(P'') + [\ell]Q'') + [\ell]([b]P_0)\nonumber\\ 
&=& \xi([a]P'') + [\ell] ([a]Q'' + [b]P_0),
\end{eqnarray}
where $P:=[a]P'' \in E(F)$ and $Q^\prime:=[a]Q'' + [b]P_0 \in E(\mathbb{Q}_{p})$. 
\end{proof} 

%%%%%%%%%%%%%%

Proposition \ref{thm:Proposition10} allows proving the following result.

%%%%%%%%%%%%%%
               
\begin{prop}
\label{thm:Proposition11}
Let $E$ be an elliptic curve over $\mathbb{Q}$ without CM, with Weierstrass equation 
$y^2 = x^3 + A x + B$, where $A,B\in \mathbb{Z}$. Set $\Delta^\prime=4A^3+27B^2$. 

Let $\ell>3$ be a prime number such that $\rho_\ell(\mathcal{G}) = {\bf GL}_2(\mathbb{Z}_\ell)$.  
Let $p$ be a prime number such that $p\nmid \Delta^\prime \ell(\ell-1)(\ell+1)$. 
Consider $\widetilde L$, $K^\prime$, and $\widetilde L^\prime$ as above.   
Let $w^\prime$ be a place of $\widetilde L^\prime$ that lies above $p$, 
and consider the place $w$ of $\widetilde L$ that lies below $w^\prime$. 
Assume that the decomposition group $\widetilde H^\prime_{w^\prime}$, 
where $\widetilde H^\prime=\Gal(\widetilde L^\prime/L^\prime)$,  
maps onto the decomposition group 
$\widetilde H_w$, where $\widetilde H=\Gal(\widetilde L/L)$, under the natural projection 
$Res_{\widetilde L}^{\widetilde L^\prime}: \widetilde H^\prime \rightarrow \widetilde H$ (defined by restriction of automorphisms to $\widetilde L$).

Consider $\widetilde f$ in $\Ker \Bigl \{ \Homol^1(\widetilde G,E[\ell]) \xrightarrow{Res^{\widetilde G}_{\widetilde G_w}} 
 \Homol^{1}(\widetilde G_w,E(\widetilde L_w)) \Bigr \}$. 

Then, there exists an element $\widetilde Q \in E(\overline{\mathbb{Q}})$ such that $[\ell]\widetilde Q \in E(K^\prime)$, 
and $\widetilde f(Res_{\widetilde L}^{\widetilde L^\prime}(\sigma)) = [\sigma-1] \widetilde Q$, 
for all $\sigma \in \widetilde H^\prime_{w^\prime}$. 
\end{prop}  

\begin{proof} 
Firstly, consider a place $w^\prime$ of $\widetilde L^\prime$, as in the statement of the proposition. 
Let us fix an embedding 
$\xi:\overline{\mathbb{Q}} \hookrightarrow \overline{\mathbb{Q}}_p$ that yields the embedding 
$\bigl ( \widetilde L^\prime \bigr )_{w^\prime} \hookrightarrow \overline{\mathbb{Q}}_p$. 

Let $\widetilde f:\widetilde G \rightarrow E[\ell]$ be a $1$-cocycle such that 
$\widetilde f(\sigma)=[\sigma-1] Q_0$ for all $\sigma\in \widetilde G_w$, for some $Q_0\in E(\widetilde L_w)$. 
Then, one has $P_0:=[\ell]Q_0 \in E(\widetilde L_w^{\widetilde G_w})=E(\mathbb{Q}_{p})$. 
Now, Proposition  \ref{thm:Proposition10} applies to $P_0$, having assumed that 
$\ell > 3$, and $p\nmid \Delta^\prime \ell(\ell-1)(\ell+1)$.
Thus, one may write $P_0$ in the form $\xi(P) + [\ell]Q^\prime$, where $P \in E(K^\prime)$, and $Q^\prime \in E(\mathbb{Q}_{p})$, 
based on Proposition \ref{thm:Proposition10}.

Let $Q \in E(\overline{\mathbb{Q}})$ such that $[\ell]Q=P$.  
The extension $L^\prime(Q)/L^\prime$ is an Abelian extension with Galois group 
embedded into $E[\ell]$, since $E[\ell] \subset L$. Indeed, the function $\Gal(L^\prime(Q)/L^\prime) \rightarrow E[\ell]$ 
that maps $\sigma \in \Gal(L^\prime (Q)/L^\prime)$ to $[\sigma -1]Q$ 
is a group homomorphism (it is a ``Kummer character'' of the elliptic curve). 
Moreover, it is injective.
Furthermore, one computes in $E(\overline{\mathbb{Q}}_p)$:   
\begin{equation}
[\ell](Q_0-\xi(Q) - Q^\prime)= P_0 - \xi(P) - [\ell] Q^\prime = O.
\end{equation} 
Thus, one has $Q_0 = \xi(Q) + Q^\prime + \xi(Q'')$, with $Q'' \in E[\ell] \subset E(L)$. 
We set 
\begin{equation}
\widetilde Q = Q + Q''.
\end{equation} 

Let us define $\widetilde g(\sigma):=[\sigma-1]\widetilde Q$, for $\sigma\in \mathcal{G}$. 
Let $w$ be the place of $\widetilde L$ lying below $w^\prime$, and consider 
the case where $\widetilde H_{w}=Res_{\widetilde L}^{\widetilde L^\prime}(\widetilde H^\prime_{w^\prime})$, as in the statement of the proposition.
Let $\sigma \in \widetilde H^\prime_{w^\prime}$. 
One then computes: 
\begin{eqnarray}
\xi \left ( \widetilde f(Res_{\widetilde L}^{\widetilde L^\prime}(\sigma)) \right ) &=& [\sigma-1] Q_0  = [\sigma-1](Q_0 - Q^\prime)\nonumber\\
&=& [\sigma -1](\xi(Q + Q'')) = [\sigma-1]\xi(\widetilde Q) = \xi ( [\sigma-1]\widetilde Q )\nonumber\\
&=& \xi \left ( \widetilde g(\sigma) \right ),
\end{eqnarray}
since $Q^\prime \in E(\mathbb{Q}_{p})$, and the embedding $\xi$ was chosen to be compatible with localization of 
$\widetilde L^\prime$ at $w^\prime$. Thus, one obtains $\widetilde f(Res_{\widetilde L}^{\widetilde L^\prime}(\sigma)) = g(\sigma)$ for any $\sigma \in \widetilde H^\prime_{w^\prime}$.
\end{proof}

%%%%%%%%%%%%%%

The hypotheses of Proposition \ref{thm:Proposition11} comprise a condition on the characteristic $p$ of the finite field intervening in the reduced curve, 
as well as condition $Res_{\widetilde L}^{\widetilde L^\prime}(\widetilde H_{w^\prime}^\prime)=\widetilde H_{w}$. 
Using Chebotarev's Density Theorem, one can show that these conditions can be met, in a form relevant to Proposition \ref{thm:Proposition12}. Namely, we have the following result.

%%%%%%%%%%%%%%

\begin{lem}
\label{thm:Lemma8} 
Let $E$ be an elliptic curve over $\mathbb{Q}$, of Weierstrass equation $y^2=x^3+Ax+B$, where $A,B\in \mathbb{Z}$. 
Let $\ell \not = 2,3,5,7,13$ be a prime number such that $\ell\nmid \Delta^\prime:=4A^3+27B^2$.  
Set $L=\mathbb{Q}(E[\ell])$, and let $d_0$ be a prime factor of $(\ell-1)/\gcd(\ell-1,12)$. 

Let $K^\prime$ and $L^\prime$ be the fields 
defined in (\ref{eq:eq149final8}) and (\ref{eq:eq183final8}), respectively. 
Let $H_1$ be a subspace of $\widetilde H^\prime=\Gal(\widetilde L^\prime/L^\prime)$, of dimension $n_1$, that is a normal
subgroup of $\Gal(\widetilde L^\prime/\left ( L^\prime \right )^{\sigma^\prime_0})$, for some element $\sigma^\prime_0$ of 
order $d_0$ in $\Gal(L^\prime/K^\prime)$. 

Then, $H_1$ admits a decomposition $\oplus_{\nu=1}^{n_1} C_\nu$, where each $C_\nu$ is a cyclic group (of order $\ell$) that is closed under conjugation by elements of $\Gal(\widetilde L^\prime/\left ( L^\prime \right )^{\sigma^\prime_0})$. 
In particular, $\langle \sigma^\prime_0 \rangle$ acts on each group $C_\nu$ through some character $\chi_\nu$. 

Then, in the case where $\chi_\nu$ is the trivial character, the group $C_\nu$ is equal to 
the decomposition group $\widetilde H^\prime_{w^\prime}$ of some place $w^\prime$ of $\widetilde L^\prime$ ({\em depending on} $C_\nu$)
that lies above a prime $p$ satisfying the condition $p\nmid \Delta^\prime \ell (\ell-1)(\ell+1)$. 

Moreover, assume that $Res_{\widetilde L}^{\widetilde L^\prime}(C_\nu)\not = 0$, where $Res_{\widetilde L}^{\widetilde L^\prime}$ denotes the natural projection  
$\widetilde H^\prime \rightarrow \widetilde H$, with $\widetilde H=\Gal(\widetilde L/L)$. 
Then, one has $Res_{\widetilde L}^{\widetilde L^\prime}(C_\nu)=\widetilde H_w$, where $w$ is the place of $\widetilde L$ lying below $w^\prime$. 
\end{lem}

\begin{proof} 
{\em Step 1.} We consider the fields $K^\prime$ and $L^\prime$ as in (\ref{eq:eq149final8}) and (\ref{eq:eq183final8}), respectively. 
We view the Galois group $\Gal(L^\prime/K^\prime)$ as the subgroup $\widetilde \rho_\ell(\Gal(L/(L\cap K^\prime)))$ of 
$\widetilde \rho_\ell(\Gal(L/\mathbb{Q})) < {\bf GL}_2(\mathbb{F}_\ell)$, under the Galois embedding $\widetilde \rho_\ell$. 

We denote $Res_{L^\prime}^{\widetilde L^\prime}$ the projection of Galois groups $\Gal(\widetilde L^\prime/K^\prime) \rightarrow \Gal(L^\prime/K^\prime)$.

{\em Step 2.}  
Let $\sigma_0^\prime$ be an element of $\Gal(L^\prime/K^\prime)$ of order $d_0$. 
Thus, one has: i) $\ord(\sigma_0^\prime)=d_0$ is coprime with $\ell$. 

Let then $\sigma_0''$ be any lifting of $\sigma_0^\prime$ in $\Gal(\widetilde L^\prime/K^\prime)$. 
Set $\widetilde \sigma_0^\prime:=(\sigma_0'')^{\ell}$. 
Then, $(\widetilde \sigma_0^\prime)^{d_0}=((\sigma_0'')^{d_0})^{\ell}=1$, since 
$(\sigma_0'')^{d_0} \in \widetilde H^\prime$, as $(\sigma_0^\prime)^{d_0}=1$, and $\widetilde H^\prime$ is a vector space over 
$\mathbb{F}_\ell$. 
Therefore, one has $d'':= \ord(\widetilde \sigma_0^\prime)\mid d_0$. 

Moreover, the projection $Res_{L^\prime}^{\widetilde L^\prime}(\widetilde \sigma_0^\prime)\in \Gal(L^\prime/K^\prime)$ is equal to $(\sigma_0^\prime)^{\ell}$. 
Thus, one has $(\sigma_0^\prime)^{d''\ell}=1$, which combined with 
$(\sigma_0^\prime)^{d_0}=1$, yields $d_0 = \ord(\sigma_0^\prime) \mid \gcd(d''\ell,d_0)=d''$ 
(since $d''\mid d_0$ is coprime with $\ell$). 
Thus, one has ii) $\ord(\widetilde \sigma_0^\prime)=d_0$.  

{\em Step 3.} Assume that $h^\prime$ is a non-trivial element of $\widetilde H^\prime$ that is in the 
centralizer of $\widetilde \sigma_0^\prime$. Consider the cyclic group $\langle h^\prime \widetilde \sigma_0^\prime \rangle$, 
where $\widetilde \sigma_0^\prime$ is as in step 2. Then, this cyclic group has order $d_0 \ell$,
since $d_0$ is coprime with $\ell$. 
It follows that iii) $(h^\prime \widetilde \sigma_0^\prime)^{d_0}=(h^\prime)^{d_0}$ is a generator of 
the cyclic group $\langle h^\prime \rangle$. 

{\em Step 4.} Assume now that $p$ is a prime that does not ramify in $\widetilde L^\prime$,
and that $w^\prime \mid p$ is a place of $\widetilde L^\prime$ such that 
$\Frob_{\widetilde L^\prime/\mathbb{Q}}(w^\prime)=h^\prime \widetilde \sigma_0^\prime$, with $h^\prime$ as in step 3, and 
$\widetilde \sigma_0^\prime$ as in step 2. 

Then, except for finitely many such primes, one may assume that $p\nmid \Delta^\prime \ell (\ell-1)(\ell+1)$. 

{\em Step 5.} Since the element $Res_{L^\prime}^{\widetilde L^\prime}(\widetilde \sigma_0^\prime)$ considered in step 2 has order $d_0$, it follows that 
 a prime $p$ as in step 4 would have residue degree $f_{L^\prime/\mathbb{Q}}$ equal to $d_0$. 
Then, this means that any place $v^\prime$ of $L^\prime$ lying above $p$ would have Frobenius element equal to 
$\Frob_{\widetilde L^\prime/L^\prime}(w^\prime)=\Frob_{\widetilde L^\prime/\mathbb{Q}}(w^\prime)^{f_{L^\prime/\mathbb{Q}}}
=(h^\prime)^{d_0}$, from property iii) of step 3. 
Therefore,  $\Frob_{\widetilde L^\prime/L^\prime}(w^\prime)$ generates the cyclic group $\langle h^\prime \rangle$.

Therefore, all desired properties for $p$ would be met, provided $\sigma_0^\prime$ satisfies condition i) -- step 2 --, and 
$h^\prime$ is in the centralizer of $\widetilde \sigma_0^\prime$  -- step 3.

{\em Step 6.} 
Now, the cyclic group $\langle \sigma^\prime_0 \rangle < \Gal(L^\prime/K^\prime)$ acts by Galois conjugation on the 
$\mathbb{F}_\ell$-vector space $H_1$, which yields a representation $\langle \sigma^\prime_0 \rangle \hookrightarrow \Aut( H_1 )$. 
Since $\# \langle \sigma^\prime_0 \rangle =d_0 \mid (\ell-1)$, the eigenvalues of this representation belong to $\mathbb{F}_\ell^*$. 
Therefore, one obtains a decomposition of representations of the finite group $\langle \sigma^\prime_0 \rangle$ over $\mathbb{F}_\ell$:
\begin{equation}
H_1 = \oplus_{\nu=1}^{n_1} C_\nu,
\end{equation}  
where $n_1:=\dim_{\mathbb{F}_\ell}( H_1 )$, and $C_\nu \approx \mathbb{F}_\ell(\chi_\nu)$, for some character 
$\chi_\nu$ of the group $\langle \sigma^\prime_0 \rangle$.

If ever $\chi_\nu$ is the trivial character, then the element $\sigma_0^\prime$ satisfies 
condition i), {\em and} $C_\nu$ is in the centralizer of $\widetilde \sigma_0^\prime$. 

An application of Chebotarev's Density Theorem \cite{chebotarev1926}
to $\widetilde L^\prime/\mathbb{Q}$ and the conjugacy class of $h^\prime \widetilde \sigma_0^\prime$,  
then yields $\Frob_{\widetilde L^\prime/\mathbb{Q}}(w^\prime_0) = \tau h^\prime \widetilde \sigma_0^\prime \tau^{-1}$,
for some place $w^\prime_0$ of $\widetilde L^\prime$ and element $\tau$ of the Galois group
$\Gal(\widetilde L^\prime/\mathbb{Q})$. 
But then, $\Frob_{\widetilde L^\prime/\mathbb{Q}}(\tau^{-1 }w^\prime_0)=h^\prime \widetilde \sigma_0^\prime$, 
so that one may take $w^\prime=\tau^{-1 }w^\prime_0$. Then, the prime $p$ lying below $w^\prime$
satisfies all the desired properties: 
$p$ is unramified in $\widetilde L^\prime$; 
$p\nmid \Delta^\prime \ell(\ell-1)(\ell+1)$;  
and $\Frob_{\widetilde L^\prime/L^\prime}(w^\prime)$ generates $C_\nu$. 

{\em Step 7.} One has 
$\Frob_{\widetilde L/\mathbb{Q}}(w) = Res^{\widetilde L^\prime}_{\widetilde L} (  \Frob_{\widetilde L^\prime/\mathbb{Q}}(w^\prime) )$, 
where $w$ is the place of $\widetilde L$ lying below $w^\prime$. Let $v$ be the place of $L$ lying below $w$, and 
$v»^\prime$ be the place of $L^\prime$ lying below $w^\prime$.
One computes:
\begin{eqnarray}
\Frob_{L/\mathbb{Q}}(v) &=& Res^{\widetilde L}_{L} (  \Frob_{\widetilde L/\mathbb{Q}}(w) )\nonumber\\
 &=& Res^{\widetilde L}_{L} \circ Res^{\widetilde L^\prime}_{\widetilde L} (  \Frob_{\widetilde L^\prime/\mathbb{Q}}(w^\prime) )\nonumber\\
&=& Res^{L^\prime}_{L} \circ Res^{\widetilde L^\prime}_{L^\prime} (  \Frob_{\widetilde L^\prime/\mathbb{Q}}(w^\prime) )\nonumber\\
&=& Res^{L^\prime}_{L} (  \Frob_{L^\prime/\mathbb{Q}}(v^\prime) ),
\end{eqnarray}
which shows that $f_{L/\mathbb{Q}}=f_{L^\prime/\mathbb{Q}}$, as $\sigma_0=\Frob_{L/\mathbb{Q}}(v)$ has same order as 
$\sigma_0^\prime=\Frob_{L^\prime/\mathbb{Q}}(v^\prime)$ under the isomorphism 
$Res^{L^\prime}_{L}: \Gal(L^\prime/K^\prime) \xrightarrow{\approx} \Gal(L/(L \cap K^\prime))$. 
One then obtains:
\begin{eqnarray}
\Frob_{\widetilde L/L}(w) &=& \Frob_{\widetilde L/\mathbb{Q}}(w)^{f_{L/\mathbb{Q}}}\nonumber\\
 &=& Res^{\widetilde L^\prime}_{\widetilde L} (  \Frob_{\widetilde L^\prime/\mathbb{Q}}(w^\prime) )^{f_{L/\mathbb{Q}}}\nonumber\\
&=& Res^{\widetilde L^\prime}_{\widetilde L} (  \Frob_{\widetilde L^\prime/\mathbb{Q}}(w^\prime)^{f_{L^\prime/\mathbb{Q}}} )\nonumber\\
&=& Res^{\widetilde L^\prime}_{\widetilde L} ( \Frob_{\widetilde L^\prime/L^\prime}(w^\prime) ),
\end{eqnarray}
since $f_{L/\mathbb{Q}}=f_{L^\prime/\mathbb{Q}}$. It then follows that $\Frob_{\widetilde L/L}(w)$ generates 
$Res_{\widetilde L}^{\widetilde L^\prime}(C_v)=\langle Res^{\widetilde L^\prime}_{\widetilde L} ( \Frob_{\widetilde L^\prime/L^\prime}(w^\prime) ) \rangle$. 
One then concludes that $Res_{\widetilde L}^{\widetilde L^\prime}(C_\nu)=\widetilde H_w$ since $p$ is unramified in $\widetilde L$,
and $\widetilde H$ is an Abelian group of exponent $\ell$, which implies that any of its non-trivial cyclic subgroups has order $\ell$.
\end{proof}

%%%%%%%%%%%%%%

Combining Lemma \ref{thm:Lemma8} and Proposition \ref{thm:Proposition11}, we have reached the following result.

%%%%%%%%%%%%%%

\begin{cor}
\label{thm:Corollary10} 
Let $E$ be an elliptic curve over $\mathbb{Q}$, of Weierstrass equation $y^2=x^3+Ax+B$ without CM, where $A,B\in \mathbb{Z}$.    
Let $\ell \not = 2,3,5,7,13$ be a prime number such that i) $\rho_\ell(\mathcal{G}) = {\bf GL}_2(\mathbb{Z}_\ell)$; 
and ii) $\ell\nmid \Delta^\prime:=4A^3+27B^2$.    
Set $L=\mathbb{Q}(E[\ell])$, and let $d_0$ be a prime factor of $(\ell-1)/\gcd(\ell-1,12)$. 

Let $K^\prime$ and $L^\prime$ be the fields 
defined in (\ref{eq:eq149final8}) and (\ref{eq:eq183final8}), respectively (depending on $A$, $B$ and $\ell$). 
Let $H_1$ be a subspace of $\widetilde H^\prime=\Gal(\widetilde L^\prime/L^\prime)$, of dimension $n_1$, that is a normal
subgroup of $\Gal(\widetilde L^\prime/\left ( L^\prime \right )^{\sigma^\prime_0})$, for some element $\sigma^\prime_0$ of 
order $d_0$ in $\Gal(L^\prime/K^\prime)$.  

Let $H_1 = \oplus_{\nu=1}^{n_1} C_\nu$ be the decomposition of representations of $\langle \sigma^\prime_0 \rangle$ over $\mathbb{F}_\ell$, 
as in Lemma \ref{thm:Lemma8}. 

Let $f=Inf^{\mathcal{G}}_{\widetilde G}(\widetilde f)$ be a $1$-cocycle in the Selmer group $S^{[\ell]}(E/\mathbb{Q})$, 
where $\widetilde G=\Gal(\widetilde L/\mathbb{Q})$ and $\widetilde f\in \Homol^1(\widetilde G,E[\ell])$. 

Then, in the case where $\chi_\nu$ is the trivial character and $Res_{\widetilde L}^{\widetilde L^\prime}(C_\nu)\not = 0$, 
one has $\widetilde f(Res_{\widetilde L}^{\widetilde L^\prime}(\sigma)) = [\sigma -1]\widetilde Q$, for any $\sigma\in C_\nu$, for some  
$\widetilde Q \in E(\overline{\mathbb{Q}})$ such that $[\ell]\widetilde Q \in E(K^\prime)$. 
The point $\widetilde Q$ depends on $\widetilde f$ and $C_\nu$. 
\end{cor}

\begin{proof} 
Let $E$ have Weierstrass equation $y^2=x^3+Ax+B$, where $A,B\in \mathbb{Z}$. Set $\Delta^\prime=4A^3+27B^2$.  
Assume that $C_\nu$ has conjugacy action defined by the trivial character of $\langle \sigma^\prime_0 \rangle$.
Then, using Lemma \ref{thm:Lemma8}, one can take a place $w^\prime$ of $\widetilde L^\prime$ such that 
$\widetilde H^\prime_{w^\prime} =  C_\nu$, where $w^\prime$ lies above a prime 
$p$ such that $p\nmid \Delta^\prime \ell (\ell-1)(\ell+1)$.  
The condition $Res_{\widetilde L}^{\widetilde L^\prime}(C_\nu)\not = 0$ implies that $\widetilde H_{w} = Res_{\widetilde L}^{\widetilde L^\prime}(C_\nu)$, where $w$ is the place of $\widetilde L$ lying below $w^\prime$. 
The corollary now follows from Proposition \ref{thm:Proposition11}. 
\end{proof}

%%%%%%%%%%%%%%

%%%%%%%%%%%%%%%%%%%%%%%%%%%%%%%%%%%%%%%%%%%%%%%%%%%% 

\subsection{Proof of Theorem \ref{thm:Theorem5}}
\label{subsection:proofTheorem5}   

The following lemma will be crucial in the proof of the important intermediate result, Proposition \ref{thm:Proposition12}, and  
the proof of Theorem \ref{thm:Theorem5}. 

%%%%%%%%%%%%%%

\begin{lem} 
\label{thm:Lemma9}
Let $E$ be an elliptic curve over $\mathbb{Q}$ without CM, of Weierstrass equation $y^2=x^3+Ax+B$, with $A,B\in \mathbb{Z}$. 
Given a prime number $\ell \not = 2,3,5,7,13$, set $L=\mathbb{Q}(E[\ell])$.
Assume that: i) $\Gal(L/\mathbb{Q}) \approx {\bf GL}_2(\mathbb{F}_\ell)$; and ii) $\ell\nmid \Delta^\prime:=4A^3+27B^2$. 
Let $K^\prime$ and $K_2$ be the fields defined in (\ref{eq:eq149final8}) and (\ref{eq:eq146final8}), respectively
(depending on $A,B$ and $\ell$). Then, one has: 

a) $\mathbb{Q}(\mu_\ell)/\mathbb{Q}$ is the maximal Abelian subextension of $L/\mathbb{Q}$, and 
$\Gal(L/\mathbb{Q}(\mu_\ell))$ corresponds to ${\bf SL}_2(\mathbb{F}_\ell)$ under $\widetilde \rho_\ell$;

b) $L \cap K_2 = \mathbb{Q}$, and $(L \cap K_3)/\mathbb{Q}$ has degree $2$;

c) $L \cap K^\prime \subseteq \mathbb{Q}(\mu_\ell)$;

d) $L \cap K^\prime = \mathbb{Q}(\mu_\ell) \cap K^\prime$ has degree over $\mathbb{Q}$ dividing $12$;

e) let $d_0$ be a prime divisor of $(\ell-1)/\gcd(\ell-1,12)$; 
then, there exists an element $\sigma_0\in \Gal(L/(L \cap K^\prime))$, of order $d_0$, such that 
$\widetilde \rho_\ell(\sigma_0)=\begin{pmatrix} \lambda_0 & 0\\0 & 1\end{pmatrix}$, with $\lambda_0\in \mathbb{F}_\ell^*$ (of order $d_0$); 
in particular, $\sigma_0$ acts on $X_{(\chi)}:=\langle \begin{pmatrix} 1\\0\end{pmatrix} \rangle$ through the character defined by 
$\chi(\sigma_0)=\lambda_0$, and acts on $X_{(1)}:=\langle \begin{pmatrix} 0\\1\end{pmatrix} \rangle$ through the trivial character;

f) the element $\tau \in \Gal(L/\mathbb{Q})$ defined by  
$\widetilde \rho_\ell(\tau)=\begin{pmatrix} 1 & 0\\1 & 1\end{pmatrix}$ belongs to the subgroup $\Gal(L/(L \cap K^\prime))$; 
in particular, $X_{(\chi)}$ is not stable under the action of $\tau$;

g) the element $\beta \in \Gal(L/\mathbb{Q})$ defined by  
$\widetilde \rho_\ell(\beta)=\begin{pmatrix} \lambda & 0\\0 & \lambda\end{pmatrix}$, where $\lambda$ has order $(\ell-1)/\gcd(\ell-1,12)$, 
belongs to the subgroup $\Gal(L/(L \cap K^\prime))$, 
and is non-trivial;
in particular, $\beta$ is in the center of $\Gal(L/(L \cap K^\prime))$, and $\beta-1$ defines an automorphism of $E[\ell]$.
\end{lem}

\begin{proof} {\em Step 1.} Firstly, we claim that the maximal Abelian sub-extension of $L/\mathbb{Q}$ 
is $\mathbb{Q}(\mu_\ell)$. 

Indeed, $\mathbb{Q}(\mu_\ell)/\mathbb{Q}$ is an Abelian sub-extension of $L/\mathbb{Q}$. 
Let $\widetilde \rho_\ell$ be the Galois representation of $\Gal(L/\mathbb{Q})$ on $E[\ell]$.
Having assumed that $\Gal(L/\mathbb{Q}) \approx {\bf GL}_2(\mathbb{F}_\ell)$, 
it follows that ${\bf SL}_2(\mathbb{F}_\ell) = \Ker \det (\widetilde \rho_\ell)$. But from the Weil pairing, one has 
$\det (\widetilde \rho_\ell)=\psi_\ell$, where $\psi_\ell$ is the cyclotomic character \cite[1.2, Example 2, pp. I-3-4]{serre1968}. 
Therefore, $\Gal(L/\mathbb{Q}(\mu_\ell))=\Ker \psi_\ell={\bf SL}_2(\mathbb{F}_\ell)$.

Now, let $K/\mathbb{Q}$ be an Abelian extension, with $L \supseteq K\supseteq \mathbb{Q}(\mu_\ell)$. 
Then, $\Gal(K/\mathbb{Q}(\mu_\ell))$ corresponds to an Abelian quotient of ${\bf SL}_2(\mathbb{F}_\ell)$. 
Since this special linear group is a perfect group for $\ell>3$ \cite[p. 61]{rose1994}, it follows that $\Gal(K/\mathbb{Q}(\mu_\ell))$ is trivial. 
Thus, $K$ is equal to $\mathbb{Q}(\mu_\ell)$, which completes the proof of part a).

{\em Step 2.} Let $K_1:=\mathbb{Q}(\mu_{4})$ as in equation (\ref{eq:eq145final8}). 
Also, as in equation (\ref{eq:eq146final8}), let  $K_2=K_1(p_1^{1/2},...,p_\nu^{1/2})$, where 
$p_1,...,p_{\nu}$ are the distinct prime factors other than $\ell$ that are bounded by the constant $C$ 
of Theorem \ref{thm:Theorem6}.

Then, the extension $L \cap K_1$ is an Abelian sub-extension of 
$L/\mathbb{Q}$, and hence is contained in $\mathbb{Q}(\mu_\ell)$, from step 1. 
From part a) of Corollary \ref{thm:Corollary8}, it follows that $L \cap K_1 \subseteq \mathbb{Q}(\mu_\ell) \cap K_2=\mathbb{Q}$.

{\em Step 3.} Next, the extension $(L \cap K_2)/\mathbb{Q}$ is an Abelian sub-extension of $L$, since $K_2/K_1$ is Abelian, 
and $L \cap K_1 = \mathbb{Q}$ from step 2.
From steps 1 and 2, one must have $L \cap K_2 \subseteq \mathbb{Q}(\mu_\ell) \cap K_2=\mathbb{Q}$, which proves the first statement of part b). 

{\em Step 4.} Since the Kummer extension $K_3/K_2$ has degree dividing $2$, it follows that the extension $(L \cap K_3)/(L\cap K_2)$ 
has degree dividing $2$. But $L\cap K_2 = \mathbb{Q}$ from step 3. Thus, $(L\cap K_3)/\mathbb{Q}$ has degree $1$ or $2$. 
But $L\cap K_3$ contains the unique quadratic subextension of $\mathbb{Q}(\mu_\ell)$. This proves the second statement of part b).

{\em Step 5.} The field $K^\prime$ is obtained by adjoining to $K_3$ various roots of cubic equations, 
as in equation (\ref{eq:eq148final8}). 
Since $K^\prime/\mathbb{Q}$ is a normal extension of degree dividing a power of $6$,  
it follows that the extension $(L \cap K^\prime)/\mathbb{Q}$ has degree dividing a power of $6$. 

Since both $K^\prime/\mathbb{Q}$ and $L/\mathbb{Q}$ are normal extensions, 
the extension $L \cap K^\prime$ is also normal over $\mathbb{Q}$. We consider the compositum 
$K'':=(L\cap K^\prime)\cdot \mathbb{Q}(\mu_\ell) \subseteq L$. 
Then, $K''$ is normal over $\mathbb{Q}$, and hence over $\mathbb{Q}(\mu_\ell)$. 
We set $N:=\Gal(L/K'')$, which is a normal subgroup of ${\bf SL}_2(\mathbb{F}_\ell)$, using part a). 
From \cite[Theorem 2, p. 62]{suprunenko1976}, there are only three cases if $\ell\geq 5$: $N=1$, $N=\{\pm I\}$, or $N={\bf SL}_2(\mathbb{F}_\ell)$, 
where $I$ denotes the $2\times 2$ matrix over $\mathbb{F}_\ell$. 
But the first two cases are ruled out, since the order of ${\bf SL}_2(\mathbb{F}_\ell)/N$ would then be divisible by $\ell$, 
whereas $K''/\mathbb{Q}(\mu_\ell)$ has degree dividing a power of $6$. 
It follows that $K''=\mathbb{Q}(\mu_\ell)$. 
This means that $L \cap K^\prime \subseteq \mathbb{Q}(\mu_\ell)$,  
which proves part c).

{\em Step 6.} From part c) of Corollary \ref{thm:Corollary8}, having assumed that $\ell >2$, the extension $(\mathbb{Q}(\mu_\ell) \cap K^\prime)/\mathbb{Q}$ has degree dividing $12$, which proves part d), making use of part c).  

{\em Step 7.} Part e) is a consequence of part d) and the assumption that $\Gal(L/\mathbb{Q})$ is the full linear group.  
Indeed, consider $\sigma \in \Gal(L/\mathbb{Q})$ such that $\widetilde \rho_\ell(\sigma)=\begin{pmatrix} \lambda & 0\\0 & 1\end{pmatrix}$, 
where $\lambda$ is a generator of $\mathbb{F}_\ell^*$. Let $d_0$ be a divisor of $(\ell-1)/\gcd(\ell-1,12)$. Then, one has
$\sigma_0:=\sigma^{(\ell-1)/d_0} \in \Gal(L/(L \cap K^\prime))$, using part d), since $\gcd(\ell-1,12)\mid (\ell-1)/d_0$. 
Then, the element $\sigma_0$ has order $d_0$, and $\widetilde \rho_\ell(\sigma_0)$ is of the form $\begin{pmatrix} \lambda_0 & 0\\0 & 1\end{pmatrix}$, 
where $\lambda_0$ has same order as $\sigma_0$. 

{\em Step 8.} The element $\tau$ has order $\ell$, so that $\tau^{\ell} = 1 \in \Gal(L/(L \cap K^\prime))$. 
But, from part d), one has $\tau^{12} \in \Gal(L/(L \cap K^\prime))$. As $\gcd(\ell,12)=1$, one concludes that 
$\tau \in \Gal(L/(L \cap K^\prime))$, which proves part f).

{\em Step 9.} Let $\lambda_1$ have order $\ell-1$ in $\mathbb{F}_\ell^*$, 
and set $\beta_1=\begin{pmatrix} \lambda_1 & 0\\0 & \lambda_1\end{pmatrix}$. 
Then, the element $\beta:=\beta_1^{\gcd(\ell-1,12)}$ belongs to $\Gal(L/(L \cap K^\prime))$, using part d). 
Furthermore, since $\ell \not = 2,3,5,7,13$, the element $\beta$ is non-trivial, and  
it then follows that $\beta-1$ is an automorphism of $E[\ell]$.
Lastly, $\widetilde \rho_\ell(\beta)$ is equal to $\begin{pmatrix} \lambda & 0\\0 & \lambda\end{pmatrix}$, 
where $\lambda:=\lambda_1^{\gcd(\ell-1,12)}$ has order $(\ell-1)/{\gcd(\ell-1,12)}$. 
This proves part g).
\end{proof}

%%%%%%%%%%%%%%

Lemma \ref{thm:Lemma9} and Corollary \ref{thm:Corollary10} play an important role in our proof of the following 
intermediate result.

%%%%%%%%%%%%%%

\begin{prop} 
\label{thm:Proposition12}
Let $E$ be an elliptic curve over $\mathbb{Q}$ without CM, of Weierstrass equation $y^2=x^3+Ax+B$, where $A,B\in \mathbb{Z}$. 
Let $\ell \not = 2,3,5,7,13$ be a prime number   
such that: i) $\rho_\ell(\mathcal{G}) = {\bf GL}_2(\mathbb{Z}_\ell)$; and ii) $\ell\nmid \Delta^\prime:=4 A^3 + 27 B^2$. 
Set $L=\mathbb{Q}(E[\ell])$, and let $K^\prime$ and $L^\prime$ be the fields defined in (\ref{eq:eq149final8}) and 
(\ref{eq:eq183final8}), respectively. 

Then, for any element of $S^{[\ell]}(E/\mathbb{Q})$ represented by $1$-cocycle $f$, 
one has a decomposition of the form: 
 \begin{equation}
Res^{\mathcal{G}}_{\mathcal{G}^\prime} (f)  = 
g + 
Inf^{\mathcal{G}^\prime}_{G^\prime}(\widetilde h),
\end{equation}
where $\mathcal{G}^\prime:=\Gal(\overline{\mathbb{Q}}/K^\prime)$ 
and $G^\prime:=\Gal(L^\prime/K^\prime) \approx \Gal(L/(L \cap K^\prime))$, 
$g \in \Homol^1(\mathcal{G}^\prime,E(\overline{\mathbb{Q}}))$ splits in $E(\overline{\mathbb{Q}})$, 
and $\widetilde h\in \Homol^1(G^\prime,E[\ell])$. 
\end{prop}

\begin{proof} 
{\em Step 1.} From Corollary \ref{thm:Corollary9} (assuming that $\ell > 3$),
 $f \in S^{[\ell]}(E/\mathbb{Q})$ is of the form $Inf^{\mathcal{G}}_{\widetilde G}(\widetilde f)$, for some 
$\widetilde f\in \Homol^{1}(\widetilde G,E[\ell])$, where $\widetilde L = L_{\overline{\mathcal{N}}}$ is the class field  
defined in Lemma \ref{thm:Lemma7} (with $K=\mathbb{Q}$) and $\widetilde G=\Gal(\widetilde L/\mathbb{Q})$. 

{\em Step 2.} 
Let $y^2=x^3+Ax +B$ be a Weierstrass equation for $E$, with $A,B\in \mathbb{Z}$.
We assume the non-CM case, with $\rho_\ell(\mathcal{G}) = {\bf GL}_2(\mathbb{Z}_\ell)$, 
$\ell\nmid \Delta^\prime:=4 A^3 + 27 B^2$, and  
$\ell \not = 2,3,5,7,13$. 
We consider the field $K^\prime$ defined in (\ref{eq:eq149final8}). 
From part d) of Lemma \ref{thm:Lemma9}, the extension $( L\cap K^\prime )/\mathbb{Q}$ has degree dividing $12$. 

Having assumed that $\ell \not = 2,3,5,7,13$, let $d_0$ be a prime factor of $(\ell-1)/\gcd(\ell-1,12)$. 
Applying part e) of Lemma \ref{thm:Lemma9}, consider the element $\sigma_0^\prime \in \Gal(L^\prime/K^\prime)\approx \Gal(L/(L \cap K^\prime))$ defined by  
$\widetilde \rho_\ell(\sigma_0^\prime) = \begin{pmatrix} \lambda_0 & 0 \\ 0 & 1 \end{pmatrix}$, 
where $\lambda_0$ is an element of order $d_0$ in the multiplicative group $\mathbb{F}_\ell^{*}$. 
Let $X_{(\chi)} = \langle \begin{pmatrix} 1\\ 0 \end{pmatrix} \rangle$ 
be the one-dimensional $\mathbb{F}_\ell$-subspace of $E[\ell]$ on which the automorphism 
$\sigma^\prime_0 \in \Gal(L^\prime/K^\prime)$, 
acts through the character defined by $\chi(\sigma^\prime_0)=\lambda_0$. 
Let $X_{(1)} = \langle \begin{pmatrix} 0\\ 1 \end{pmatrix} \rangle$ 
be the one-dimensional $\mathbb{F}_\ell$-subspace of $E[\ell]$ on which 
$\sigma^\prime_0$ acts through the trivial character. One has the decomposition $E[\ell] = X_{(\chi)} \oplus X_{(1)}$.  

{\em Step 3.} 
Given $\widetilde Q\in E(\overline{\mathbb{Q}})$ such that $[\ell]\widetilde Q=P\in E(K^\prime)$, 
the $1$-cocycle defined by $\widetilde g(\sigma):=[\sigma -1]\widetilde Q$ 
belongs to $S^{[\ell]}(E/K^\prime)$. Furthermore, by construction, any such 
element of this Selmer group splits in $E(\widetilde L^\prime)$. 
Let $S^{[\ell]}(E/K^\prime)_{split}$ be the subgroup of $S^{[\ell]}(E/K^\prime)$ consisting of such elements.  

We define (with $L^\prime=L K^\prime$):
\begin{equation}
\begin{cases}
\mathcal{G}^\prime = \Gal(\overline{\mathbb{Q}}/K^\prime);\qquad   
\mathcal{H}^\prime = \Gal(\overline{\mathbb{Q}}/L^\prime);\\ 
\widetilde G^\prime = \Gal(\widetilde L^\prime/K^\prime);\qquad 
\widetilde H^\prime = \Gal(\widetilde L^\prime/L^\prime);\\ 
G^\prime = \Gal(L^\prime/K^\prime).              
\end{cases}
\end{equation}
The $\mathbb{F}_\ell \langle \sigma^\prime_0 \rangle$-module $\widetilde H^\prime$ admits a decomposition:
\begin{equation}
\widetilde H^\prime = \oplus_{i=0}^{d_0-1} \widetilde H^\prime_{(\chi^i)},
\end{equation}
where $\langle \sigma^\prime_0 \rangle$ acts on $\widetilde H^\prime_{(\chi^i)}$ through the character $\chi^i$ of the cyclic group $\langle \sigma^\prime_0 \rangle$, 
for $i=0,1,...,d_0-1$.

Let then $\mathcal{I}^\prime$ be the image of the map:
\begin{equation}
\begin{CD}
\varrho^\prime: S^{[\ell]}(E/K^\prime)_{split}  @>{Res^{\mathcal{G}^\prime}_{\mathcal{H}^\prime}}>> S^{[\ell]}(E/L^\prime)^{G^\prime}\\ 
@>>>  \Hom(\widetilde H^\prime, E[\ell]/X_{(\chi)})^{\sigma_0^\prime}\\
@>{\approx}>> \Hom(\widetilde H^\prime/\bigl (\oplus_{i=1}^{d_0-1} \widetilde H^\prime_{(\chi^i)} \bigr ), E[\ell]/X_{(\chi)}),\\
\end{CD}
\end{equation}
where the first map is the one of Proposition \ref{thm:Proposition6} (applied to $K=K^\prime$), and the second map is induced by 
restriction of $G^\prime$ to $\langle \sigma_0^\prime \rangle$ and the projection $E[\ell] \rightarrow E[\ell]/X_{(\chi)}$. 
We have used the isomorphism $E[\ell]/X_{(\chi)}\approx X_{(1)}$ in the third map.

Then, from Pontryagin duality, one has an isomorphism:
\begin{equation}
\mathcal{I}^\prime \approx 
\Hom( \widetilde H^\prime/H_1, E[\ell]/X_{(\chi)}),
\end{equation}
for some closed subspace $H_1 \supseteq \oplus_{i=1}^{d_0-1} \widetilde H^\prime_{(\chi^i)}$ of $\widetilde H^\prime$. 

{\em Step 4.} 
Since $X_{(\chi)}$ is stable under $\sigma_0^\prime$, it follows that $H_1$ is
stable under conjugation by a lifting $\widetilde \sigma_0^\prime$ of $\sigma_0^\prime$ in $\Gal(\widetilde L^\prime/K^\prime)$. 
Indeed, with notation as above, 
the equality $[h - 1]\widetilde Q \in X_{(\chi)}$ for all $\widetilde Q$ 
such that $[\ell]\widetilde Q = \widetilde P \in E(K^\prime)$, 
implies that $[h   - 1] (\widetilde \sigma_0^\prime)^{-1}(\widetilde Q )
\in X_{(\chi)}$ for all such point $\widetilde Q$, because $K^\prime$ is normal over $\mathbb{Q}$.  
But then, this implies that $[\widetilde \sigma_0^\prime h (\widetilde \sigma_0^\prime)^{-1} - 1]\widetilde Q 
=  \widetilde \sigma_0^\prime [h - 1] (\widetilde \sigma_0^\prime)^{-1}(\widetilde Q ) \in 
\widetilde \sigma_0^\prime(X_{(\chi)})=X_{(\chi)}$ for all 
such point $\widetilde Q$. Therefore, $\widetilde \sigma_0^\prime h (\widetilde \sigma_0^\prime)^{-1} \in H_1$.

{\em Step 5.} We define:
\begin{equation}
\begin{cases}
\mathcal{G} = \Gal(\overline{\mathbb{Q}}/\mathbb{Q});\qquad 
\mathcal{H} = \Gal(\overline{\mathbb{Q}}/L);\\ 
\widetilde G = \Gal(\widetilde L/\mathbb{Q});\qquad 
\widetilde H = \Gal(\widetilde L/L);\\ 
G = \Gal(L/\mathbb{Q}).              
\end{cases}
\end{equation}

Let then $\mathcal{I}$ denote the image of the map 
\begin{equation}
\begin{CD}
\varrho: S^{[\ell]}(E/\mathbb{Q}) @>{Res^{\mathcal{G}}_{\mathcal{H}}}>> S^{[\ell]}(E/L)^{G}\\ 
@>>> \Hom(\widetilde H, E[\ell]/X_{(\chi)})^{\sigma_0}\\
@>{\left( Res_{\widetilde L}^{\widetilde L^\prime}\right )_*}>> \Hom(\widetilde H^\prime, E[\ell]/X_{(\chi)})^{\sigma_0^\prime}\\
@>{\approx}>> \Hom(\widetilde H^\prime/\bigl (\oplus_{i=1}^{d_0-1} \widetilde H^\prime_{(\chi^i)} \bigr ), E[\ell]/X_{(\chi)}),\\
\end{CD}
\end{equation}
where the first map is the one of Proposition \ref{thm:Proposition6} (applied to $K=\mathbb{Q}$), and $Res_{\widetilde L}^{\widetilde L^\prime}$ denotes the natural projection $\widetilde H^\prime \rightarrow \widetilde H$ of Galois groups. 
We have used the isomorphism $E[\ell]/X_{(\chi)}\approx X_{(1)}$ in the fourth map.

Then, from Pontryagin duality, one has an isomorphism:
\begin{equation}
\mathcal{I} \approx 
\Hom( \widetilde H^\prime/H_2, E[\ell]/X_{(\chi)}),
\end{equation}
for some closed subspace $H_2 \supseteq \oplus_{i=1}^{d_0-1} \widetilde H^\prime_{(\chi^i)}$ of $\widetilde H^\prime$. 

{\em Step 6.} We claim that $\mathcal{I} \subseteq \mathcal{I}^\prime$, equivalently $H_1 \subseteq H_2$. 

Based on step 4, the subgroup $H_1$ is normal in the Galois group 
$\Gal(\widetilde L^\prime/\left ( L^\prime \right )^{\sigma^\prime_0})$, where $\sigma^\prime_0$ is constructed in step 2.  
From Lemma \ref{thm:Lemma8}, one obtains a decomposition $H_1 = \oplus_{\nu=1}^{n_1} C_\nu$ of representations of 
$\langle \sigma^\prime_0 \rangle$ over $\mathbb{F}_\ell$. 

If the one-dimensional representation $C_\nu$ maps to $0$ under the Galois projection 
$Res_{\widetilde L}^{\widetilde L^\prime}: \widetilde H^\prime \rightarrow \widetilde H$, then $C_\nu$ is contained in $H_2$, and there is nothing to prove.

If the one-dimensional representation $C_\nu$ is defined by the trivial character $\chi_\nu$ of $\langle \sigma^\prime_0 \rangle$, 
and $Res_{\widetilde L}^{\widetilde L^\prime}(C_\nu)\not = 0$, then Corollary \ref{thm:Corollary10} implies that 
$\widetilde f(Res_{\widetilde L}^{\widetilde L^\prime}(\sigma)) = [\sigma -1]\widetilde Q$, for any $\widetilde f \in S^{[\ell]}(E/\mathbb{Q})$  
and any $\sigma\in C_\nu$, for some $\widetilde Q \in E(\overline{\mathbb{Q}})$ such that $[\ell]\widetilde Q \in E(K^\prime)$, depending on $\widetilde f$ and $C_\nu$. 
But then, the inclusion $C_\nu < H_1$ implies that $\widetilde f(Res_{\widetilde L}^{\widetilde L^\prime}(\sigma))=0$, for any $\widetilde f \in S^{[\ell]}(E/\mathbb{Q})$ and any $\sigma\in C_\nu$. This means that $C_\nu < H_2$.

If the one-dimensional representation $C_\nu$ is defined by a non-trivial character of $\langle \sigma^\prime_0 \rangle$, 
then $C_\nu < \oplus_{i=1}^{d_0-1} \widetilde H^\prime_{(\chi^i)}$. 
But from step 5, one has the inclusion $\oplus_{i=1}^{d_0-1} \widetilde H^\prime_{(\chi^i)} \subseteq H_2$, 
which means that $C_\nu < H_2$.

Altogether, we conclude that $H_1 \subseteq H_2$. 

{\em Step 7.} From Step 6, we have $\mathcal{I} \subseteq \mathcal{I}^\prime$. 
So, let $f$ be in the Selmer group $S^{[\ell]}(E/\mathbb{Q})$. 
Then, $\varrho(f)$ satisfies:
\begin{equation}
\varrho(f)(\sigma) = \varrho^\prime (g)(\sigma) \mod X_{(\chi)},
\end{equation}
valid for any $\sigma \in \mathcal{H}^\prime$, for some $1$-coboundary $g$ of the form:
\begin{equation}
\widetilde g(\sigma)  : = [\sigma-1]\widetilde Q,
\end{equation}
where $\widetilde Q\in E(\overline{\mathbb{Q}})$ such that $[\ell]\widetilde Q=P\in E(K^\prime)$.

Define:
\begin{equation}
h:= Res^{\mathcal{G}}_{\mathcal{G}^\prime} (f) - g.
\end{equation} 
Thus, $Res^{\mathcal{G}^\prime}_{\mathcal{H}^\prime}(h)$ maps to $0\in\Hom(\mathcal{H}^\prime,E[\ell]/X_{(\chi)})$. 
This means that $h$ maps the Galois group $\mathcal{H}^\prime$ into $X_{(\chi)}$. 

{\em Step 8.} Now, from Proposition \ref{thm:Proposition6} applied to $K^\prime$, one has:
\begin{equation}
Res^{\mathcal{G}^\prime}_{\mathcal{H}^\prime} ( h ) 
\in \Hom(\mathcal{H}^\prime,E[\ell])^{G^\prime}.
\end{equation} 
Therefore, for any $\sigma \in \mathcal{G}^\prime$ and $\tau\in \mathcal{H}^\prime$, one has:
\begin{equation}
\sigma(h(\tau))  = h(\sigma \tau \sigma^{-1}) \in X_{(\chi)},
\end{equation}
since $\mathcal{H}^\prime$ is a normal subgroup of $\mathcal{G}^\prime$.  
Thus, one must have $h(\tau)=0$ for any $\tau\in \mathcal{H}^\prime$, 
since the one-dimensional subspace $X_{(\chi)}$ is not stable under Galois action of $G^\prime$. 
Indeed, one may take the element $\tau$ corresponding to the matrix $\begin{pmatrix} 1 & 0\\ 1 & 1 \end{pmatrix}$ of order $\ell$, 
as in part f) of Lemma \ref{thm:Lemma9}.

It follows that $Res^{\mathcal{G}}_{\mathcal{G}^\prime} (f) - g = 0$ on $\mathcal{H}^\prime$, so that 
\begin{equation}
Res^{\mathcal{G}}_{\mathcal{G}^\prime} (f) - g = Inf^{\mathcal{G}^\prime}_{G^\prime}(\widetilde h),
\end{equation} 
for some $\widetilde h \in \Homol^1(G^\prime, E[\ell])$, as was to be shown.
\end{proof}

%%%%%%%%%%%%%%

%%%%%%%%%%%%%%%%%%%%%%%%%%%%%%%%%%%

Equipped with Lemma \ref{thm:Lemma9} and Proposition \ref{thm:Proposition12}, 
we are now ready to prove Theorem \ref{thm:Theorem5}.

{\em Proof of Theorem \ref{thm:Theorem5}.} 
We show the inclusion:
\begin{equation}
\label{eq:eq207final8}
S^{[\ell]}(E/\mathbb{Q}) \subseteq \Ker \Bigl \{ \Homol^{1}(\mathcal{G},E[\ell]) \rightarrow 
 \Homol^{1}(\mathcal{G},E(\overline{\mathbb{Q}})) \Bigr \}.
\end{equation} 

Let $\ell \not = 2,3,5,7,13$ be a prime number. 
Assuming that: i) $\rho_\ell(\mathcal{G}) = {\bf GL}_2(\mathbb{Z}_\ell)$, 
as well as the condition: ii) $\ell \nmid \Delta^\prime$, Proposition \ref{thm:Proposition12} applies.  
Set $L:=\mathbb{Q}(E[\ell])$ and $G:=\Gal(L/\mathbb{Q})$. 
Thus, for any $f \in S^{[\ell]}(E/\mathbb{Q})$, $Res^{\mathcal{G}}_{\mathcal{G}^\prime}(f)$ is of the form 
$g + Inf^{\mathcal{G}^\prime}_{G^\prime}(\widetilde h)$ on $\mathcal{G}^\prime$, 
for some $g \in \Homol^1(\mathcal{G}^\prime,E(\overline{\mathbb{Q}}))$ 
that splits in $E(\overline{\mathbb{Q}})$, and $\widetilde h\in \Homol^1(G^\prime,E[\ell])$. 

Now, consider the element $\beta =\begin{pmatrix} \lambda & 0\\0 & \lambda \end{pmatrix}$, where 
$\lambda \in \mathbb{F}_\ell^*$ has order $(\ell-1)/\gcd(\ell-1,12)$. 
Then, from part g) of Lemma \ref{thm:Lemma9}, 
$\beta$ is in the center of $\Gal(L/(L \cap K^\prime)) \approx G^\prime$, 
and $\beta-1$ defines an automorphism of $E[\ell]$. 
Thus, from Sah's Theorem \cite[Theorem 5.1, p. 118]{lang1978}, one has:
\begin{equation}
\label{eq:eq208final8}
\Homol^1(G^\prime,E[\ell])=0.
\end{equation} 
See \cite[Proposition 19, p. 51]{bashmakov1972} in the context of Abelian varieties. 
Note that  \cite[Theorem 11]{lawson2016} on vanishing of Galois cohomology groups defined on torsion points does not apply here, 
since $\mathbb{Q}(\mu_\ell) \cap K^\prime$ might be larger than $\mathbb{Q}$ ({\em c.f.} part d) of Lemma \ref{thm:Lemma9}). 
On the other hand, Coates' result \cite[Lemma 10, p. 179]{coates1970} does apply, since  
$G^\prime=\Gal(L^\prime/K^\prime)$ contains ${\bf SL}_2(\mathbb{F}_\ell)$, as follows from parts a) and d) of Lemma \ref{thm:Lemma9}.   

Thus, one has for any $\sigma \in \mathcal{G}^\prime$:
\begin{equation}
f(\sigma) = [\sigma-1](\widetilde Q + R),
\end{equation} 
where $\widetilde Q\in E(\overline{\mathbb{Q}})$ and $R \in E[\ell]$. 
Lastly, one has for any $\sigma \in \mathcal{G}$:
\begin{equation}
f(\sigma) - [\sigma-1](\widetilde Q + R ) = Inf^{\mathcal{G}}_{\Gal(K^\prime/\mathbb{Q})}(h^\prime),
\end{equation}
for some $h^\prime\in \Homol^{1}(\Gal(K^\prime/\mathbb{Q}),E(K^\prime))$. 
But then, from \cite[Corollary 10.2, p. 84]{brown1982}, $\Homol^{1}(\Gal(K^\prime/\mathbb{Q}),E(K^\prime))$ 
is annihilated by $\vert \Gal(K^\prime/\mathbb{Q}) \vert = n$. 
Therefore, one has for any $\sigma \in \mathcal{G}$:
\begin{equation}
[n] \left ( f(\sigma) - [\sigma-1](\widetilde Q + R ) \right ) = [\sigma-1]R^\prime,
\end{equation}
for some $R^\prime\in E(K^\prime)$. Thus, $[n]f$ maps to $0$ in $H^1(\mathcal{G},E(\overline{\mathbb{Q}}))$. 
Since, on the other hand, $f$ maps into $H^1(\mathcal{G},E(\overline{\mathbb{Q}}))[\ell]$, 
the inclusion (\ref{eq:eq207final8}) is proved, as $\gcd(n,\ell)=1$ by Proposition \ref{thm:Proposition3}.

The other inclusion is clear since, for any place $v_0$ of $\mathbb{Q}$, the homomorphism   
$Res^{\mathcal{G}}_{\mathcal{G}_{v_0}}: \Homol^{1}(\mathcal{G},E[\ell]) \rightarrow \Homol^{1}(\mathcal{G}_{v_0},E(\overline{\mathbb{Q}}_{v_0}))$ is the composition of homomorphisms $\Homol^{1}(\mathcal{G},E[\ell]) \rightarrow \Homol^{1}(\mathcal{G},E(\overline{\mathbb{Q}})) 
\xrightarrow{Res^{\mathcal{G}}_{\mathcal{G}_{v_0}}} \Homol^{1}(\mathcal{G}_{v_0}, E(\overline{\mathbb{Q}}_{v_0}))$.
\hfill $\square$ 

%%%%%%%%%%%%%%%%%%%%%%%%%%%%%%%%%%%%%%%%%%%%%%%%%%%%
%%%%%%%%%%%%%%%%%%%%%%%%%%%%%%%%%%%%%%%%%%%%%%%%%%%%

\section{Examples}
\label{section:examples}

Recall that, given an elliptic curve $E$ over the rationals, there is, for each prime $\ell$, a Galois representation 
$\widetilde \rho_\ell: \mathcal{G} \rightarrow {\bf GL}_2(\mathbb{F}_\ell)$, where $\mathcal{G}$ is the absolute Galois group of $\mathbb{Q}$, through Galois action on the $\ell$-torsion points of $E$. Then, one has an isomorphism 
$\Gal(\mathbb{Q}(E[\ell])/\mathbb{Q}) \approx \widetilde \rho_\ell(\mathcal{G})$. In \cite{serre1972}, the notation is $\phi_\ell(G)$.
\\

%%%%%%%%%%%%%%

\noindent {\bf Example 1.} From \cite[Proposition 1.4, p. 45]{silverman2009}, to each $j\in \mathbb{Q}$ corresponds a unique class of elliptic curves over $\mathbb{Q}$ up to isomorphism
 over $\overline{\mathbb{Q}}$. For each $j_0\in\mathbb{Q}$, the class of elliptic curves having $j_0$ as $j$-invariant is 
in one-to-one correspondence with $\mathbb{Q}^*/(\mathbb{Q}^*)^{n(j)}$, where $n(j)=2,4,6$ according to the cases 
$j\not=0,1738$, $j=1728$, or $j=0$, respectively \cite[Corollary 5.4.1, p. 343]{silverman2009}. 
From \cite[p. 427]{silverman2009}, there are exactly $13$ elliptic curves over $\mathbb{Q}$, up to isomorphism over 
$\overline{\mathbb{Q}}$, having CM. 
Then, for each corresponding $j$-invariant (see \cite[p. 295]{serre1967b} for a list), there are infinitely many elliptic curves over $\mathbb{Q}$ having CM  
({\em i.e.}, over a finite base field extension).
Moreover, except for these $13$ $j$-invariants, to any $j\in \mathbb{Q}$ corresponds a class of elliptic curves without CM. 
\\

%%%%%%%%%%%%%%

\noindent {\bf Example 2.}  Consulting \cite[Table 1]{rubin2002}, we consider the elliptic curve from \cite{penney1975}:
\begin{equation}
E\::\: y^2=x^3 + a x^2 + b x,
\end{equation}
with $a=1,692,602 = 2 \cdot 37 \cdot 89 \cdot 257$ 
and $b=-3 \cdot 5 \cdot 11 \cdot 13 \cdot 17 \cdot 19 \cdot 23 \cdot 29 \cdot 31 \cdot 37$.
This equation corresponds to \cite[p. 42]{silverman2009}:
\begin{equation}
\begin{cases}
a_1=a_3=a_6=0;\\
a_2=a; \quad a_4=b;\\
b_2=4a;\quad b_4=2b; \quad b_6=0;\\
b_8=-b^2;\\
c_4=16a^2-48b;\\
c_6=- 64a^3 + 36 \cdot 8 ab. 
\end{cases}
\end{equation}
One computes the discriminant:
\begin{eqnarray}
\Delta(E)&=&-b_2^2b_8-8b_4^3-27b_6^2+9b_2b_4b_6=16a^2b^2-64b^3\nonumber\\
&=&16 b^2(a^2-4b),
\end{eqnarray} 
and the $j$-invariant:
\begin{equation}
j(E)=c_4^3/\Delta=\frac{(16)^2(a^2-3b)^3}{b^2(a^2-4b)}.
\end{equation}
Furthermore, one has:
\begin{equation}
\begin{cases}
\Delta(E)=2^8 \cdot 3^2 \cdot 5^2 \cdot 11^2 \cdot 13^2 \cdot 17^2 \cdot 19^2 \cdot 23^2 \cdot 29^2 \cdot 31^2 \cdot 37^3 \cdot p^\prime;\\
j(E) = 
\frac{2^4 \cdot 7^3 \cdot 61^3 \cdot 347^3 \cdot (p'')^3}
{3^2 \cdot 5^2 \cdot 11^2 \cdot 13^2 \cdot 17^2 \cdot 19^2 \cdot 23^2 \cdot 29^2 \cdot 31^2 \cdot p^\prime},      
\end{cases}
\end{equation}
where $p^\prime$ and $p''$ are the prime numbers $8,420,798,017$ and $812,633$, respectively. 
In particular, $E$ is not semi-stable.
 
Thus, this elliptic curve has potential multiplicative reduction at 
$p_0=3$. Therefore, $E$ has no CM. Moreover, from \cite{penney1975}, $E$ has rank at least $7$ over $\mathbb{Q}$. 
Therefore, Theorem \ref{thm:Theorem4} does not apply, whereas 
Theorem \ref{thm:Theorem5} does. 

Moreover, $p=7$ is the smallest prime number at which $E$ has good reduction. 
From \cite[Proposition 24, p. 314]{serre1972}, one concludes that 
$\widetilde \rho_\ell(\mathcal{G})={\bf GL}_2(\mathbb{F}_\ell)$, 
whenever $\ell \nmid \Delta(E)$, $\ell \nmid \ord_{p_0}(j(E))= -2$, 
and $\ell > (\sqrt{p}+1)^8$. This means that Theorem \ref{thm:Theorem5} applies to $E$ with $\ell > 31,210$   
other than $p^\prime$.  
  
Pushing further these computations, let us consider   
the corresponding Tate's curve \cite[Theorem 14.1, p. 445]{silverman2009} at $p_0=3$. One obtains an isomorphism 
$\phi: \overline{\mathbb{Q}}_{p_0}/\langle q \rangle \xrightarrow{\approx} E(\overline{\mathbb{Q}}_{p_0})$, 
as $\Gal(\overline{\mathbb{Q}}_{p_0}/K'')$-modules, for some unramified quadratic extension $K''/\mathbb{Q}_{p_0}$. 
Here, $q$ is defined in \cite[p. 444]{silverman2009}. One can easily check that $\ord_{p_0}(q)=2$, 
from the fact that $\ord_{p_0}(\Delta)=2$. 
Therefore, one obtains an isomorphism 
$\phi: (\langle q^{1/\ell}\rangle \times\mu_\ell)/\langle q \rangle \xrightarrow{\approx} E[\ell](\overline{\mathbb{Q}}_{p_0})$, 
as $\Gal(\overline{\mathbb{Q}}_{p_0}/K'')$-modules. 
This implies that $\widetilde \rho_\ell(\mathcal{G})$ has a cyclic subgroup of order $\ell$, for any prime $\ell\not = 2$. 
 
Now, from \cite[Theorem 3]{mazur1978}, if $G=\widetilde \rho_\ell(\mathcal{G})\not = {\bf GL}_2(\mathbb{F}_\ell)$, then either 
$G$ is in the normalizer $N$ of a Cartan subgroup $C$, or $\ell$ is one of the exceptional primes $2,3,5,7,11,13,17,19,37,43,67,163$.
But $C$ has index $2$ in $N$ \cite[\S 2.2, p. 279]{serre1972}, so that $N$ has order coprime with $\ell$. Therefore, one concludes that $\ell$ is one of the exceptional primes. 

Furthermore, note that $E$ has potential good reduction at $2$, 
since $\ord_{2}(\Delta(E))=8$ and $\ord_{2}(j(E))=4$.  
It follows from \cite[a3), p. 312]{serre1972}, that the group $\Phi_{2}$ defined in 
\cite[pp. 311--312]{serre1972}, has cardinality $2$, $3$, $4$, $6$, $8$, or $24$. Recall that $\Phi_2$ is a quotient group 
of the inertia group $I_2$, and that it embeds into ${\bf GL}_2(\mathbb{F}_\ell)$, if $\ell \geq 5$ \cite[pp. 311--312]{serre1972}.
Now, the Weierstrass equation $y^2=x^3+ax^2+bx$ is minimal over $\mathbb{Q}_2$ since $\ord_2(\Delta)=8<12$. 
From Section \ref{subsection:localElliptic}, the curve $E$ has good reduction over a finite extension $K^\prime/\mathbb{Q}_2$ of degree with only $2$ or $3$ as prime factors. 
Moreover, one can make a change of variable of the form $x=u^2x^\prime+r$ and $y=u^3y^\prime+u^2 s x^\prime+t$, 
where $u\in \mathcal{O}_{K^\prime}^*$, $r,s,t \in \mathcal{O}_{K^\prime}$ \cite[Proposition 1.3, part a), p. 186]{silverman2009}, 
and obtain the discriminant $u^{-12}\Delta$ \cite[Remark 1.1, p. 186]{silverman2009}. 
It follows that $v(u^{-12}\Delta)= - 12 v(u) + 8 v(2)  = 0$, where $v$ is the discrete valuation of $K^\prime$.
This in turn implies that $3 \mid v(2)$, so that the ramification index of the extension $K^\prime/\mathbb{Q}_2$ 
is divisible by $3$. 
One concludes that $\Phi_2$ has order $3$, $6$, or $24$ (see \cite[p. 312]{serre1972}). 

Thus, for $\ell\not = 2$, if $\widetilde \rho_\ell(\mathcal{G})$ is not the full linear group, then it is contained in a Borel subgroup of the linear group. Indeed, \cite[Corollaire, p. 277]{serre1972} implies that $\widetilde \rho_\ell(\mathcal{G})$ contains a split Cartan 
semi-subgroup of the form $\begin{pmatrix} * & 0\\ 0 & 1 \end{pmatrix}$. 
Then, the condition $\ell \mid \# \widetilde \rho_\ell(\mathcal{G})$ for $\ell\not =2$ implies that $\widetilde \rho_\ell(\mathcal{G})$ 
is either the full linear group, or else is contained in a Borel subgroup. See \cite[Proposition 17 and remark a), p. 282]{serre1972}. 
But then, assuming $\ell \geq 5$, \cite[Proposition 23, part b), p. 313]{serre1972} implies that the divisor $3$ of $\vert \Phi_{2} \vert$ divides the order of $\left (\mathbb{Z}/2^n\mathbb{Z}\right)^*$, for some $n\geq 1$, which is not the case. 
So, actually, $\widetilde \rho_\ell(\mathcal{G})={\bf GL}_2(\mathbb{F}_\ell)$, for all $\ell\geq 5$. 
See \cite[5.7.1, p. 315]{serre1972} for this argument. 
It follows from (\ref{eq:eq41final7}) that $\rho_\ell(\mathcal{G})={\bf GL}_2(\mathbb{Z}_\ell)$, for all $\ell\geq 5$. 

Next, we consider the Weierstrass equation of $E$:
\begin{eqnarray}
&&E : y^2 = x^3 + A x + B;\nonumber\\
&& \; A = -27 c_4 = - 27 (16a^2-48b);\nonumber\\
&& \; B  = - 54 c_6 = -54 (- 64a^3 + 36 \cdot 8 ab);\nonumber\\
&& \; \Delta^\prime(E) = - 2^{20} \cdot 3^{12} \cdot b^2 (a^2 - 4b).%;\nonumber\\
\end{eqnarray}

We conclude from Theorem \ref{thm:Theorem5} that $\Sha(E/\mathbb{Q})_{\ell}$ vanishes at all primes $\ell$, other than the ones in the set:
\begin{equation}
P=\{2,3,5,7,11,13,17,19,23,29,31,37,p^\prime\}.
\end{equation}
In particular, the smallest prime number for which Corollary \ref{thm:Corollary2} applies is $\ell=41$. Thus, one has:
\begin{eqnarray}
\rank (E/\mathbb{Q}) &=& \dim_{\mathbb{F}_{41}} S^{[41]}(E/\mathbb{Q}),
\end{eqnarray} 
since $\Sha(E/\mathbb{Q})[41]=0$, from Theorem \ref{thm:Theorem5}, and $E[41](\mathbb{Q})=0$ from Mazur's Theorem on torsion points. 
\\   

%%%%%%%%%%%%%%%

\noindent {\bf Remark 5.}
\label{thm:Remark5}
The previous example solves the open problem mentioned in \cite[Problem 2.16, p. 27]{stein2007} in the non-CM case.
\\

%%%%%%%%%%%%%%

%%%%%%%%%%%%%%

\noindent {\bf Example 3.}  The following example was communicated to us by Professor C. Wuthrich. 
The non-CM elliptic curve of rank $0$ defined by the cubic equation \cite{wuthrich1058e1}:
\begin{eqnarray}
&&E : y^2 + xy = x^3 - x^2 - 332,311 x - 73,733,731;\nonumber\\
&& \; \Delta(E) = - 5,302,593,435,347,072 = - 2^7 \cdot 23^{10};\nonumber\\ 
&& \; c_4 = 15,950,937 = 3 \cdot 19 \cdot 23^4;\nonumber\\
&&\; j(E) = - \frac{(3 \cdot 19 \cdot 23^4)^3}{2^7 \cdot 23^{10}} 
= - 2^{-7} \cdot 3^3 \cdot 19^3 \cdot 23^2,
\end{eqnarray}
has Shafarevich-Tate group of {\em analytic} order $25$, which is denoted as $\# \Sha(E/\mathbb{Q})_{an}$ $=25$. 
From \cite{miller2011}, one concludes that $\# \Sha(E/\mathbb{Q}) = 25$, as $E$ has conductor $N=1058<5000$, and rank $r\leq 1$.

Recall that $\# \Sha(E/\mathbb{Q})_{an}$ is based on BSD-2, and is computed as follows: 
\begin{equation}
\# \Sha(E/\mathbb{Q})_{an} = \frac{\lim_{s\rightarrow 1} (s-1)^{-r}L(E,s) (\# E_{tor}(\mathbb{Q}))^{2}}{\Omega \, 2^r R  \prod_p c_p},
\end{equation}
where $L(E,s)$ denotes the $L$-series of $E$, $r$ is the rank of $E/\mathbb{Q}$, $\Omega$ is defined from the invariant differential, 
$R$ is the elliptic regulator of $E(\mathbb{Q})/E_{tor}(\mathbb{Q})$, and $c_p$ denotes $\# E(\mathbb{Q}_p)/E_0(\mathbb{Q}_p)$.
See \cite[pp. 451--452]{silverman2009}. In the example, based on the information available on the Website \cite{wuthrich1058e1}, 
this expression simplifies to:
\begin{equation}
\# \Sha(E/\mathbb{Q})_{an} = \frac{L(E,1)}{\Omega} =  25,
\end{equation}
which is consistent with Cassels' result \cite[Theorem 4.14, p. 341]{silverman2009}.

Since $E$ has multiplicative reduction at $p_0=2$, it follows that $\widetilde \rho_\ell(\mathcal{G})$ 
has a cyclic subgroup of order $\ell$, 
for any prime $\ell \not = 7$. Indeed, the Tate's curve $E_q$, with $\ord_{p_0}(q)=7$, yields an isomorphism 
$\phi: (\langle q^{1/\ell}\rangle \times\mu_\ell)/\langle q \rangle \xrightarrow{\approx} E[\ell](\overline{\mathbb{Q}}_{p_0})$, as 
$\Gal(\overline{\mathbb{Q}}_{p_0}/K'')$-modules 
(where $K''$ is the unramified extension of degree $2$ over $\mathbb{Q}_{p_0}$, since $E$ has non-split multiplicative reduction at 
$p_0=2$). 

Thus, unless $\ell$ is one of the exceptional primes $2,3,5,7,11,13,17,19,37,43,67$, or $163$, 
one has $\widetilde \rho_\ell(\mathcal{G})={\bf GL}_2(\mathbb{F}_\ell)$ \cite[Theorem 3]{mazur1978}.

Pushing further this example, observe that $E$ has potential good reduction at $23$. 
It follows from \cite[a1), p. 312]{serre1972}, that the group $\Phi_{23}$ has cardinality $6$, since $23 >3$ and $\ord_{23}(\Delta)=10$. 
Thus, for $\ell \not = 7$, if $\widetilde \rho_\ell(\mathcal{G})$ is not the full linear group, 
then it is contained in a Borel subgroup of the linear group. 
But then, assuming that $\ell \geq 5$, \cite[Proposition 23, part b), p. 313]{serre1972} implies that $6 = \vert \Phi_{23} \vert$ divides  
the order of $\left (\mathbb{Z}/23^n\mathbb{Z}\right)^*$, for some $n\geq 1$, which is not the case.
So, actually, $\widetilde \rho_\ell(\mathcal{G})={\bf GL}_2(\mathbb{F}_\ell)$, for all $\ell \not = 2,3,7$, and hence, in particular for $\ell=5$. It follows from (\ref{eq:eq41final7}) that $\rho_\ell(\mathcal{G})={\bf GL}_2(\mathbb{Z}_\ell)$, for all $\ell$  
except possibly $2$, $3$ and $7$.
Actually, it is reported that $\widetilde \rho_\ell(\mathcal{G})={\bf GL}_2(\mathbb{F}_\ell)$ holds for any prime $\ell$ 
\cite{wuthrich1058e1}. 

Next, we consider the Weierstrass equation of $E$:
\begin{eqnarray}
&&E : y^2 = x^3 + A x + B;\nonumber\\
&& \; A = -430675299 = - 3^{4} \cdot 19 \cdot 23^{4};\nonumber\\
&& \; B  = - 3,443,997,030,498 = - 2 \cdot 3^6 \cdot 23^5 \cdot 367;\nonumber\\
&& \; \Delta^\prime(E) = 2^{15} \cdot 3^{12} \cdot 23^{10}.
\end{eqnarray}
This yields the following Weierstrass equation, under the change of variable $(x,y)\mapsto (3^2 x^\prime,3^3 y^\prime)$:
\begin{eqnarray}
&&E : y^2 = x^3 + A_1 x + B_1;\nonumber\\
&& \; A_1 = -5,316,979 = - 19 \cdot 23^{4};\nonumber\\
&& \; B_1 = - 4724275762 = - 2 \cdot 23^5 \cdot 367;\nonumber\\
&& \; \Delta_1^\prime(E) = 2^{15} \cdot 23^{10}.
\end{eqnarray}
Thus, one has to discard the prime $\ell=23 \mid \Delta_1^\prime(E)$ in addition to the exceptional primes  
$2,3,5,7,13$ (to avoid the exceptional condition $(\ell-1)/\gcd(\ell-1,12)=1$). 
So, although $\rho_\ell(\mathcal{G})={\bf GL}_2(\mathbb{Z}_\ell)$ at $\ell=5$ in this example, 
Theorem \ref{thm:Theorem5} does not predict the vanishing of $\Sha(E/\mathbb{Q})_{\ell}$ 
due to the exceptional condition $(\ell-1)\mid 12$. 

Altogether, Theorem \ref{thm:Theorem5} predicts that $\Sha(E/\mathbb{Q})_{\ell}$ vanishes at any prime $\ell$ other than 
$2,3,5,7,13,23$. 
In particular, the conclusion is consistent with BSD-2 in this example ({\em i.e.}, $\# \Sha(E/\mathbb{Q})=25$).
\\

\noindent {\bf Example 4.} 
Consider the elliptic curve presented in Example 3:
\begin{eqnarray}
&&E : y^2 = x^3 + A_1 x + B_1;\nonumber\\
&& \; A_1 = -5,316,979 = - 19 \cdot 23^{4};\nonumber\\
&& \; B_1 = - 4724275762 = - 2 \cdot 23^5 \cdot 367;\nonumber\\
&& \; \Delta_1^\prime(E) = 2^{15} \cdot 23^{10}.
\end{eqnarray}
This curve has no rational points over $\mathbb{Q}$ \cite{wuthrich1058e1}. 

Let us consider the prime $p=7$. The reduced curve has Weierstrass equation:
\begin{eqnarray}
&&\widetilde E_{p} : y^2 = x^3 + 4 x + 4,
\end{eqnarray}
since $y^2 = x^3 + A_1 x + B_1$ is a minimal Weierstrass equation for $E$ at $p$. 
The non-trivial points of the reduced curve modulo $p$ are: $(0,\pm 2)$, $(1,\pm 3)$, $(3,\pm 1)$, $(-3,0)$, 
$(-2,\pm 3)$. Thus, $\widetilde E(\mathbb{F}_p)$ has order $2 \cdot 5$. 
We take $\ell=5$, so that $\ell \nmid \Delta^\prime$.  
Since $p \nmid \Delta^\prime(\ell -1)(\ell +1)\ell$, 
it follows that Theorem \ref{thm:Theorem6} applies.

The non-trivial points that have order $\ell=5$ are the ones of the form $[2]P$ with $[2]P\not = O$: 
$(1,\pm 3)$, $(-2,\pm 3)$.  
One may take the generator $(1,3)$ of $\widetilde E(\mathbb{F}_p)_{\ell}$, 
so that $\frac{a}{b} = 9$ in Theorem \ref{thm:Theorem6}. 
We may choose the radical $y=\left ( \frac{a}{b} \right )^{1/2}=3$, in this simple situation.   
Let then $x$ be a root of $X^3+A_1X+B_1=3^2$ in $\overline{\mathbb{Q}}$ that maps to 
$1$ in the residue field $\mathbb{F}_p$, under an embedding $\xi:\mathbb{Q}(x) \rightarrow \overline{\mathbb{Q}}_p$ 
followed by projection into the residue field of $\mathbb{Q}(x)$ at some prime $\mathfrak{p}\mid p$. 
By Hensel's Lemma \cite[p. 43]{lang1986}, one has $\mathbb{Q}(x)_{\mathfrak{p}} = \mathbb{Q}_p$. 
Any point of $\widetilde E[\ell](\mathbb{F}_p)$ can be lifted to a point of 
$\mathbb{Q}(x)$, which is contained in the field $K^\prime$ defined in (\ref{eq:eq149final8}). 

Then, for any point $P_0 \in E(\mathbb{Q}_p)$, the point $[m]P_0$, where $m=2$, projects to a point 
$\widetilde P \in \widetilde E(\mathbb{F}_p)$ that can be lifted to a point $P^\prime$ 
with affine coordinates in $K^\prime$ (in fact, in $\mathbb{Q}(x)$). 
Thus, $[m]P_0-\xi(P^\prime) \in E_1(\mathbb{Q}_p) \subset [\ell]E(\mathbb{Q}_p)$, so that $[m]P_0-\xi(P^\prime)=[\ell]Q''$, 
for some $Q''\in E(\mathbb{Q}_p)$. 
Since $3 m -\ell=1$, one obtains: $P_0=\xi([3]P^\prime) + [\ell]([3]Q'' - P_0)$, and we set $P=[3]P^\prime$ and $Q^\prime=[3]Q'' - P_0$.

Now, assume that $P_0 \not \in [\ell] E(\mathbb{Q}_p)$. Such a point exists since $\# \widetilde E(\mathbb{F}_p) = 2 \cdot \ell$ and 
$E(\mathbb{Q}_p)$ projects onto $\widetilde E(\mathbb{F}_p)$.   
Then, one must have $P\in E(\mathbb{Q}(x)) \setminus E(\mathbb{Q})$. 
Indeed, since $E(\mathbb{Q})=0$ in this example, the case $P\in E(\mathbb{Q})$ would imply
that $P_0 \in [\ell]E(\mathbb{Q}_p)$. 
This issue was pointed out to us by Professor K. Rubin in an early draft of this paper. 
This motivated us to develop the results of Section \ref{section:liftingsPoints}. As the field $K^\prime$ 
is a finite extension over $\mathbb{Q}$, our approach in Section \ref{section:ProofMainTheorems} was then sufficient 
to prove Theorem \ref{thm:Theorem5}. 
\\

%%%%%%%%%%%%%%

%%%%%%%%%%%%%%%%%%%%%%%%%%%%%%%%%%%%%%%%%%%%%%%%%%%%
%%%%%%%%%%%%%%%%%%%%%%%%%%%%%%%%%%%%%%%%%%%%%%%%%%%%

\appendix

%%%%%%%%%%%%%%%%%%%%%%%%%%%%%%%%%%%%%%%%%%%%%%%%%%%%
%%%%%%%%%%%%%%%%%%%%%%%%%%%%%%%%%%%%%%%%%%%%%%%%%%%%

\section{Ramification of the extension $L_\infty/\mathbb{Q}$}
\label{section:appendixA}

The following result is a consequence of a theorem of Sen \cite{sen1972} that was conjectured by Serre \cite{serre1967c}.

\begin{prop}
\label{thm:Proposition13}
Let $E$ be an elliptic curve over the rationals. 
Let $\ell > 3$ be a prime number at which $E$ has good reduction. 
Let $L_n$ be the number field obtained by adjoining the affine coordinates of the $\ell^n$-torsion points of $E$. 
Then, the different $\mathfrak{D}_n$ of $L_n/\mathbb{Q}$ satisfies the estimate:
\begin{equation}
(\ell^{n} a) \subseteq \mathfrak{D}_n \subseteq (\ell^{n} a^{-1}) ,
\end{equation}
for all $n\geq 1$, for some integer $a$. 
\end{prop}

\begin{proof}
We consider the following four cases, in view of \cite[Proposition 10, p. 52]{serre1979}.

{\em Case A:} $p =\ell>3$ (and $p \not \in \Sigma_{E}$). 
Consider the Galois group $G$ of the infinite extension obtained by adjoining over $\mathbb{Q}$ the affine coordinates of all $\ell^n$-torsion points of $E$, where $n\geq 1$, as an $\ell$-adic Lie group.  
Let $\{ G_n \}$ be a Lie filtration on $G$. 
For instance, one may take $G_n:=\rho_\ell^{-1}(I+\ell^n{\bf Mat}_{2\times 2}(\mathbb{Z}_\ell))$. 
On the other hand, let $\{ G(n) \}$ denote the upper numbering filtration on the Galois group $G$ corresponding to 
an embedding $\overline{\mathbb{Q}} \hookrightarrow \overline{\mathbb{Q}}_\ell$. 
Sen's Theorem \cite[Theorem, p. 48]{sen1972} gives the estimate:
\begin{equation}
G(n e-c) \subseteq G_n \subseteq G(n e+c),
\end{equation} 
valid for all $n$, for some constant $c$ (depending only on $G$, and hence on $E$ and $\ell$), where $e$ is the absolute ramification index of the ground field (so, $\mathbb{Q}_\ell$ here, and $e=1$).

Now, let $\mathfrak{P}$ denote a prime ideal of $L_n$ lying above $\ell$. 
Then, using \cite[Proposition 4, p. 64]{serre1979}, one has:
 \begin{eqnarray}
\val_{\mathfrak{P}} \left ( \mathfrak{D}_n \right ) &=& 
\sum_{u=0}^{u_0} \left ( \vert G[u] \vert - 1 \right),
\end{eqnarray}
where $G[u]$ denotes the lower numbering ramification groups, and $u_0$ is the largest integer $u$ such that $G[u]$ is non-trivial.
One has:
 \begin{eqnarray}
\sum_{u=0}^{u_0} \vert G[u] \vert &=& \vert G[0] \vert \int_{u=0}^{u_0+1} \frac{1}{(G[0]:G[u])}\, du\nonumber\\
&=& \vert G[0] \vert \varphi(u_0+1),
\end{eqnarray}
where $\varphi$ denotes Herbrand's function \cite[p. 73]{serre1979}. From Sen's Theorem, one has 
$n-c \leq \varphi(u_0+1) \leq n+c$.

Moreover, from \cite[Exerc. 3 c), pp. 71--72]{serre1979}, one has $u_0\leq e_\mathfrak{P}/(\ell-1)$, and hence 
$u_0 + 1\leq e_\mathfrak{P}\ell/(\ell-1)$. Thus, we obtain:
 \begin{equation}
\prod_{\mathfrak{P}\mid \ell} \mathfrak{P}^{e_{\mathfrak{P}}(n+c)} \subseteq 
\prod_{\mathfrak{P}\mid \ell} \left ( \mathfrak{D}_n \right )_{\mathfrak{P}} \subseteq 
\prod_{\mathfrak{P}\mid \ell} \mathfrak{P}^{e_{\mathfrak{P}}(n-c-\ell/(\ell-1))},
\end{equation}
that is:
 \begin{equation}
(\ell^{n+c} ) \subseteq \prod_{\mathfrak{P}\mid \ell} \left ( \mathfrak{D}_n \right )_{\mathfrak{P}} \subseteq 
(\ell^{n-c -\ell/(\ell-1)} ).
\end{equation}

{\em Case B:} $p \not = \ell$ and $p\not \in \Sigma_{E}$. Then, from the Criterion of N\'eron-Ogg-Shafarevich, 
the extension $L_n/\mathbb{Q}$ is unramified. 
Hence, using \cite[Theorem 1, p. 53]{serre1979}, one obtains:
 \begin{equation}
\prod_{\mathfrak{P}\mid p} \left ( \mathfrak{D}_n \right )_{\mathfrak{P}} = (1), 
\end{equation}
where $\mathfrak{P}$ stands for prime ideals of $L_n$. 

{\em Case C:} $p \not = \ell$ and $p \in \Sigma_{E,p.g.}$. 
Then, from Lemma \ref{thm:Lemma1} (having assumed that $\ell > 3$), $E$ has good reduction over $L_1$. 
From the Criterion of N\'eron-Ogg-Shafarevich, the extension $L_n/L_1$ is unramified. 
Therefore, using \cite[Proposition 8, p. 51]{serre1979}, one has:
\begin{equation}
\prod_{\mathfrak{P}\mid p} \left ( \mathfrak{D}_n \right )_{\mathfrak{P}} = 
 \prod_{\mathfrak{p}\mid p} \left ( \mathfrak{D}_1 \right )_{\mathfrak{p}},
\end{equation}
where $\mathfrak{p}$ denotes prime ideals of $L_1$.

{\em Case D:} $p \not = \ell$ and $p \in \Sigma_{E,p.m.}$. 
From Lemma \ref{thm:Lemma1} (having assumed that $\ell > 3$), $E$ has multiplicative reduction over $L_1$.
Considering Tate's curves, the extension $L_n/L_1$ is at most tamely ramified. 
Therefore, one obtains:
\begin{equation}
\prod_{\mathfrak{p}\mid p} \left ( \mathfrak{D}_1 \right )_{\mathfrak{p}} p \subseteq 
\prod_{\mathfrak{p}\mid p} \left ( \mathfrak{D}_1 \right )_{\mathfrak{p}} 
\prod_{\mathfrak{P}\mid p} \mathfrak{P}^{e_{\mathfrak{P}}(L_n/L_1)-1} \subseteq
\prod_{\mathfrak{P}\mid p} \left ( \mathfrak{D}_n \right )_{\mathfrak{P}} \subseteq
(1),
\end{equation}
using \cite[Proposition 13, p. 58]{serre1979}. Here, $e_{\mathfrak{P}}(L_n/L_1)$ denotes 
the relative ramification index of $\mathfrak{P}$ in $L_n/L_1$. 
Since $e_{\mathfrak{P}}(L_n/L_1)$ is bounded by (in fact, divides)  
the ramification index $e_{\mathfrak{P}}$ of $\mathfrak{P}$ in $L_n/\mathbb{Q}$, 
the first inclusion holds, as $(p)=\prod_{\mathfrak{P}\mid p} \mathfrak{P}^{e_{\mathfrak{P}}}$. 
\end{proof}

%%%%%%%%%%%%%%

\noindent {\bf Remark 6.}
\label{remark6}
Let $L_\infty$ be the infinite Galois extension obtained by adjoining the affine coordinates of all $\ell^n$-torsion points of $E$, 
with $n\geq 1$. Consider the cyclotomic fields $K_n=\mathbb{Q}(\mu_{\ell^n})$, for $n\geq 1$. 
From the Weil pairing, one has the inclusions $K_n \subset L_n$, for $n\geq 1$. 
From \cite[Theorem 3, p. 75]{lang1986}, one has:
\begin{equation}
\mathfrak{D}_{K_n/\mathbb{Q}} = (\ell^n \ell^{-1/(\ell-1)} ).
\end{equation}
This is consistent with the general results of Tate \cite[\S 3.1, pp. 170--172]{tate1967b}. 
Based on Sen's Theorem, one deduces that:
\begin{eqnarray}
(1) &\supseteq& \mathfrak{D}_{L_n/K_n} = \mathfrak{D}_{L_n/\mathbb{Q}} \, \mathfrak{D}_{K_n/\mathbb{Q}}^{-1} 
\supseteq (\ell^{n} a) (\ell^{-n} \ell^{1/(\ell-1)})\nonumber\\
&=& (a \ell^{1/(\ell-1)}) \supseteq (c),
\end{eqnarray}
for all $n\geq 1$, where $c=a \ell$. See also \cite[Remarque, p. 152]{serre1981} for a consequence of Sen's Theorem.

Now, consider $\alpha_n$ as in Proposition \ref{thm:Proposition2}, an element of the integer ring $\mathcal{O}_n$ of $L_n=\mathbb{Q}(E[\ell^n])$, with 
$n\geq 1$. Then, using Proposition \ref{thm:Proposition2} and Remark 2, one has:
\begin{eqnarray}
\Bigl \vert \frac{\tr_n(\alpha_n)}{[L_n:\mathbb{Q}]} \Bigr \vert_{\ell} &=& \Bigl \vert \frac{\tr_n(\alpha_n)}{b \ell^{nN}} \Bigr \vert_{\ell}\nonumber\\
&\sim& C_0 \ell^{n (N-1)} \Bigl \vert \frac{\tr_n(\alpha_n)}{\ell^{n}} \Bigr \vert_{\ell} \leq 1,
\end{eqnarray}
for some positive constant $C_0$. 
On the other hand, Proposition \ref{thm:Proposition13} implies that:
\begin{equation}
\vert a \vert_\ell \Bigl \vert \frac{\tr_n(\alpha_n)}{\ell^{n}} \Bigr \vert_{\ell}
\leq \Bigl \vert \tr_n(\alpha_n \mathfrak{D}_{L_n/\mathbb{Q}}^{-1}) \Bigr \vert_{\ell} \leq 1.
\end{equation}
But $\lim_{n\rightarrow \infty} C_0 \ell^{n (N-1)} = \infty$, as $N\geq 2$, whereas 
$\bigl \vert a \bigr \vert_{\ell} < \infty$. 
Thus, the precise form of $\alpha_n$ in Proposition \ref{thm:Proposition2}, {\em i.e.}, $\alpha_n=\Delta^\prime \ell^3/y^2(P)$,  
is at stake in this proposition concerning the $\ell$-adic norm, in addition to the strong condition 
$\rho_{\ell}(\mathcal{G}) = {\bf GL}_2(\mathbb{Z}_\ell)$. 
For the Archimedean norm, one also needs the precise form of $\alpha_n$, but not the condition on $\rho_{\ell}(\mathcal{G})$.
But for the other non-Archimedean norms, the mere fact that $\alpha_n$ is an integral element is sufficient, and this is the only trivial case. 
Thus, the result of Sen indicates that Proposition \ref{thm:Proposition2} is by no means elementary, as it seems. 
In particular, Serre's Theorems on the Galois group of $L_\infty/\mathbb{Q}$, in the case of non-CM elliptic curves over the rationals, 
played an essential role in our proof of Theorems \ref{thm:Theorem5} and \ref{thm:Theorem6}.
\\

%%%%%%%%%%%%%%%%%%%%%%%%%%%%%%%%%%%%%%%%%%%%%%%%%%%%%
%%%%%%%%%%%%%%%%%%%%%%%%%%%%%%%%%%%%%%%%%%%%%%%%%%%%%

\subsection*{Acknowledgments}
We are grateful to Professors Karl Rubin and Christian Wuthrich for their helpful comments on the first version of this article.
Moreover, we acknowledge the helpful comments of Professor Rubin on drafts of the second version. 
The elliptic curve in Example 3 that was mentioned to us by Professor Wuthrich was very helpful in developing the second version of 
this work. 

%%%%%%%%%%%%%%%%%%%%%%%%%%%%%%%%%%%%%%%%%%%%%%%%%%%%
%%%%%%%%%%%%%%%%%%%%%%%%%%%%%%%%%%%%%%%%%%%%%%%%%%%%

%%%%%%%%%%%%%%%%%%%%%%%%%%%%%%%%%%%%%%%%%%%%%%%%%%%%
%%%%%%%%%%%%%%%%%%%%%%%%%%%%%%%%%%%%%%%%%%%%%%%%%%%%

\end{document}